\documentclass[11pt,twoside,a5paper]{book}
\usepackage{makeidx}
\usepackage{amsmath, amsthm, graphicx, color, amssymb}

\usepackage{latexsym}\usepackage{comment}

\usepackage{amssymb,amsmath,amsfonts,amsthm,color}

\makeatletter
\newcommand{\MK}{\mathbb{K}}
\newcommand{\cp}{characteristic polynomial }

\newcommand*\bigcdot{\mathpalette\bigcdot@{.5}}
\newcommand*\bigcdot@[2]{\mathbin{\vcenter{\hbox{\scalebox{#2}{$\m@th#1\bullet$}}}}}
\makeatother
\newtheorem{Ex}{Example}
\newtheorem{Conj}{Conjecture}
\newtheorem{theorem}{Theorem}
\newtheorem{lemma}[theorem]{Lemma}
\newtheorem{corollary}[theorem]{Corollary}

\theoremstyle{definition}

\theoremstyle{definition}
\newtheorem{definition}[theorem]{Definition}

\theoremstyle{definition}

\theoremstyle{definition}
\newtheorem{example}[theorem]{Example}

\title{Describing subalgebras of $\MK[x]$ using derivatives}
\author{Rode Gr\"onkvist, Erik Leffler, Anna Torstensson, Victor Ufnarovski}

\makeindex

\begin{document}

\maketitle



\chapter*{Abstract}
\vspace{-1cm}
We introduce the concept of subalgebra spectrum, $Sp(A)$, for a subalgebra $A$ of finite codimension in $\mathbb{K}[x]$. The spectrum is a subset of the underlying field. We also introduce a tool, the characteristic polynomial of $A$, which has the spectrum as its set of zeroes. The characteristic polynomial can be computed from the generators of $A$, thus allowing us to find the spectrum of an algebra given by generators. We proceed by using the spectrum to get descriptions of subalgebras of finite codimension. More precisely we show that $A$ can be described by a set of conditions that each is either of the type $f(\alpha)=f(\beta)$ for $\alpha,\beta$ in $Sp(A)$ or of the type stating that some sum of derivatives of different orders evaluated in elements of $Sp(A)$ equals zero. We use this type of conditions to, by an inductive process, find explicit descriptions of subalgebras of codimension up to three. These descriptions also include SAGBI bases for each family of subalgebras.

\tableofcontents

\chapter{Introduction}

\section{Introductory Examples}

Let $\MK$ be an algebraically closed field of characteristic zero and $A$ a subalgebra in $\MK[x].$ To begin with we give several non-trivial examples of such subalgebras.

\begin{Ex}\label{ex:monomial}$A=\{f(x)|f'(0)=f''(0)=f^{(5)}(0)=0\}.$
\end{Ex}
\begin{Ex}\label{ex:spectere4} Let $\varepsilon$ be a primitive root of order $8.$ $$A=\{f(x)|f(1)=f(-1),f(\varepsilon)=f(\varepsilon^7),f(\varepsilon^3)=f(\varepsilon^{5})\}.$$
\end{Ex}
\begin{Ex}Let $\varepsilon$ be a primitive root of order $12.$ $$A=\{f(x)|f'(0)=0,f(\varepsilon)=f(\varepsilon^5),f(\varepsilon^7)=f(\varepsilon^{11})\}.\label{ex:ex3}$$
\end{Ex}

\begin{Ex}\label{ex:monomial4}Let $\varepsilon$ be a primitive root of order $3.$ $$A=\{f(x)|f(1)=f(\varepsilon)=f(\varepsilon^2),f'(1)+\varepsilon^2f'(\varepsilon)+\varepsilon f'(\varepsilon^{2})=0\}.$$
\end{Ex}
It is not difficult to verify directly that we really get  subalgebras. One can check that in fact, if given by generators, they are:
$$1)\langle x^3,x^4\rangle\ \hspace{1cm} 2)\langle x^4,x^3-x\rangle\ $$ $$3)\langle x^4-x^2,x^3\rangle\ \hspace{1cm} 4)\langle x^4-x,x^3\rangle.$$
We want to find general principles for how descriptions of the type in our examples relates to descriptions in forms of generators and other characteristics of subalgebras.

We restrict ourselves to subalgebras of finite codimension $n$ and give a classification for small $n.$
\section{SAGBI bases}

One of our aims is to get a deeper understanding for the structure of SAGBI bases, for example to find ways to add an extra element to a SAGBI basis in ways that result in a new SAGBI basis.
For this reason we remind the reader of some definitions. When possible we adapt them to our univariate situation. More general definitions  can be found for example in \cite{tor} or \cite{Rob}

If $A$ is a subalgebra in $\MK[x]$ the set $S$ of all possible degrees of the non-constant polynomials in  $A$ form a \textbf{numerical semigroup}\index{numerical semigroup} (that is an additive semigroup consisting of positive integers).
It is well-known that such a semigroup is finitely generated. For any finite generating set we can find a finite set of polynomials $G$  such that our set is exactly $\{\deg g_i|g_i\in G\}.$ We call $G$ \textbf{SAGBI basis}\index{SAGBI basis} for $A.$ A proper subset of $G$ can be a SAGBI basis itself, but if there are no such subsets we say that $G$ is \textbf{minimal}.\index{minimal SAGBI basis}

For any non-constant polynomial $f$ of the degree $s\in S$ we can find a product $g=\prod_{g_i\in G} g^{c_i}$ such that $\deg g=\sum c_i\deg g_i=s.$ Forming $f-\alpha g$ with a suitable constant $\alpha\in \MK$ we can obtain a polynomial of smaller degree. We call this operation \textbf{subduction}.\index{subduction}  If the degree of the obtained polynomial still belongs to $S,$ then we can use make another subduction. The importance of SAGBI basis lies in the fact that $f\in A$ if and only if there exists a sequence of subductions reducing $f$ to a constant.

\section{Monomial subalgebras}

As we have seen Example \ref{ex:monomial} in fact describes the subalgebra generated by $x^3$ and $x^4$. This result can easily be generalized.
\begin{theorem}\label{th:monomial} Let $A$ be a \textbf{monomial subalgebra},\index{monomial subalgebra}  thus $A$ is spanned over $\MK$ by monomials $\{x^s,s\in S\},$ where $S$ is a numerical semigroup.  Then
	$f(x)\in A$ if and only if $f^{(i)}(0)=0$ for each $i$ that does not belong to $S.$
\end{theorem}
\begin{proof} First we check that the derivative conditions describe a subalgebra $A'$. The conditions are linear so we need only to be sure that if $f(x)$ and $g(x)$ satisfy the conditions then the same is true for the product $f(x)g(x).$
	Indeed if $i\not\in S$ then we have
	$$(fg)^{(i)}=\sum_j \binom{i}{ j}f^{(j)}g^{(i-j)}$$
	and either $j$ or $i-j$ does not belong to $S$ (otherwise $i\in S)$ and in any case $f^{(j)}(0)g^{(i-j)}(0)=0.$
	Secondly we see directly that any monomial $x^s,s\in S$ satisfies the conditions. In fact only the monomials $x^i$ with $i\not\in S$ do not satisfy the conditions. So certainly $A\subseteq A'$, but we can say more:
	if $f(x)\in A'$ then subduction by $A$ reduces $f(x)$ to another polynomial that satisfies the conditions but is a linear combination of the monomials $x^i$ with $i\not\in S$. Such a polynomial must be zero and therefore
	$f(x)\in A$ and $A'\subseteq A.$ We conclude that $A'=A$.
	
\end{proof}

Here is another useful property of monomial algebras.
\begin{theorem} \label{th:monsols}  Let  $A=\langle x^{a_1},x^{a_2}, \ldots , x^{a_t} \rangle$ be  a monomial subalgebra. There exists $\alpha\neq\beta$ such that $f(\alpha)=f(\beta)$ for all $f(x) \in A$ if and only if $d=gcd(a_1,a_2, \ldots , a_t) > 1$.
\end{theorem}

\begin{proof}   If $d>1$ let $\varepsilon$ is a primitive $d$-th root of unity, $\varepsilon^d=1.$ Then for any nonzero $\beta$ we can find $\alpha=\varepsilon\beta$ such that $f(\alpha)=f(\beta)$.

	If $f(\alpha)=f(\beta)$ for all $f(x)\in A$ with $\alpha\neq \beta$ then $\beta\neq 0.$ Let $d=\sum c_ia_i.$ Then
	$$\left(\frac{\alpha}{\beta}\right)^{a_i}=1\Rightarrow \left(\frac{\alpha}{\beta}\right)^{d}=\prod\left(\left(\frac{\alpha}{\beta}\right)^{a_i}\right)^{c_i}=1\Rightarrow d>1.$$
\end{proof}

Note that if $d>1$ then the subalgebra $A$ is contained in $\MK[x^d]$ and therefore it has infinite codimension. Such $A$ are outside the scope of our work.

\section{Subalgebras of codimension one}
Next, let us look at subalgebras of codimension one (in $\MK[x]).$ Although relatively simple, these algebras give some insight. Obviously such subalgebra cannot contain $x,$ but do contain polynomials of degree $2$ and $3$, which generate our subalgebra. Using variable substitution we can restrict ourselves to the case where the polynomial of degree two is $x^2$. (Note that all constants are always in any subalgebra). Now the polynomial of degree three can be chosen as $x^3-ax$. (Again, the constants are not essential and $bx^2$ can be subtracted). If $a=0$ then we get a monomial case and know how to describe it from Theorem \ref{th:monomial}.

If $a\neq 0$ then the replacement $x\rightarrow\alpha x$ with $\alpha^2= a$ reduces the situation to the case $x^3-x.$ So it is sufficient to study subalgebra $A=\langle x^3-x,x^2\rangle .$ Note that for each odd $k>1$ we have $x^k-x=(x^{k-2}-x)x^2+(x^3-x)\in A$ by induction. So
$f(x)=\sum a_ix^i$ can be subduced to $ax$ where $a=a_1+a_3+a_5+\cdots$. Thus $f(x)\in A\Leftrightarrow a=0\Leftrightarrow f(1)-f(-1)=0.$ This gives us the following result:

\begin{theorem} \label{th:codm1} For any subalgebra $A$ of codimension one either the there exists $\gamma$ such that $f(x)\in A\Leftrightarrow f'(\gamma)=0$ or there exists $\alpha\neq\beta$ such that
	$f(x)\in A\Leftrightarrow f(\alpha)=f(\beta)$.
\end{theorem}
\begin{proof} We only need to recover the old variable. Then the monomial case corresponds to the first case and $f(1)=f(-1)$ to the second.
\end{proof}

The above theorem already displays some ideas that we will try to generalize later on.
\chapter{The Subalgebra Spectrum}

\section{Derivations}
\begin{definition}Let $\alpha\in \MK.$ A linear map $D:A\rightarrow \MK$ is called an $\alpha-$\textbf{derivation} \index{derivation} \index{$\alpha-$derivation} if it satisfies the condition
	$$D(f(x)g(x))= D(f(x))g(\alpha)+f(\alpha)D(g(x))$$
	for any $f(x),g(x)\in A.$ We simply call it a \textbf{derivation} if it is an $\alpha-$derivation for some $\alpha.$
\end{definition}
Note that the set of $\alpha-$derivations is a vector space over $\MK,$ but the set of all derivations is not.
Nevertheless it is important for the future to note that a $\beta-$derivation is also an $\alpha-$derivation if $f(\alpha)=f(\beta)$ for any $f(x)\in A.$

Now we can formulate an important result obtained in \cite{Gorin}, that will turn out to be pivotal for our continued exploration.
\begin{theorem} \label{th:codimGorin} Any subalgebra $A$ of codimension $n>1$  is contained in  subalgebra $B$ of   codimension $n-1.$ Moreover $A$ can be defined in $B$ either as the kernel of some $\alpha-$derivation of $B$ or as  $A=\{f(x)\in B|f(\alpha)=f(\beta)\}$ for some $\alpha,\beta\in\MK.$
\end{theorem}

Note that in \cite{Gorin} derivations are defined in a more general way, by the condition $D(fg)= D(f)\varphi(g)+\varphi(f)D(g),$ for some ring homomorphism $\varphi: B\rightarrow \MK$.
But in the same article is shown that any homomorphism $A\rightarrow \MK$ can be lifted to a homomorphism $B\rightarrow \MK$.  Induction over codimension shows that
in our situation such homomorphism is simply a homomorphism $\MK[x]\rightarrow \MK$ which is nothing else than a map $f(x)\rightarrow f(\alpha)$ for some $\alpha\in\MK.$ For that reason we can use $\alpha-$derivation in our reformulation.

 \section{Subalgebras conditions}\label{sec:lincond}

A straightforward induction argument using Theorem \ref {th:codimGorin} shows that any subalgebra $A$ of codimension $n$ can be described by $n$ linear conditions $L_i(f)=0$ where $L_i$ is either a derivation of some subalgebra containing $A$ or has the form $L(f)=f(\alpha_i)-f(\beta_i)$ for some constants $\alpha_i,\beta_i \in \MK.$

 Our main hypothesis when initiating this work (which will proved later) was that linear conditions defining subalgebras can be stated in a neater way. Namely we hoped that for any subalgebra of finite codimension $m$ there would exist a finite set, which we will call the spectrum of algebra, and $m$ linear conditions expressed in terms of $f(x)$ and finitely many derivatives $f^{(k)}$ evaluated in the elements of the spectrum which determine if $f(x)\in A.$ We have seen such conditions in Theorem \ref{th:monomial} and Theorem \ref{th:codm1} and in Examples \ref{ex:monomial}-\ref{ex:monomial4} and want to understand their nature.

 We want them to be \textbf{subalgebra conditions}, \index{subalgebra conditions} i.e. that the set of all polynomials satisfying the conditions form  a subalgebra. Since our conditions are linear we only need to demand two things for them to be subalgebra conditions.
 Firstly, a trivial one: that constants should satisfy the conditions. Secondly, a non-trivial one: that whenever $f(x)$ and $g(x)$ satisfy the conditions, so does the product $f(x)g(x)$.

 For example the condition $f(\alpha)=0$ is not an subalgebra condition, because the non-zero constants does not satisfy it. But the condition $f(\alpha)=f(\beta)$ is a subalgebra condition. The same is true for the condition $f'(\alpha)=0.$

 The singe condition $f'(\alpha)+f'(\beta)=0$ is not subalgebra condition, but together the conditions $f(\alpha)=f(\beta),f'(\alpha)+f'(\beta)=0$ are subalgebra conditions. As this example shows being subalgebra conditions is a property of a {\it set} of conditions. (The set may, however, as in the first two examples, consist of just one element.)

 In general,  any condition
 $\sum c_if'(\alpha_i)=0$
  combined with  $f(\alpha_1)=f(\alpha_2)=\cdots =f(\alpha_k)$  gives subalgebra conditions.

 Indeed since the conditions are linear we only need to check that if $f(x)$ and $g(x)$ satisfies the conditions then the same is true for $f(x)g(x).$ We have
 $$\sum c_i(fg)'(\alpha_i)=\sum c_if'(\alpha_i)g(\alpha_i)+c_if(\alpha_i)g'(\alpha_i)=$$
 $$\left(\sum c_if'(\alpha_i)\right)g(\alpha_1)+f(\alpha_1)\left(\sum c_ig'(\alpha_i)\right)=0.$$
 One can find generalisations including derivatives of higher order, but we skip this for now and show only one spectacular example of subalgebra conditions:
 $$f'(0)=0;\quad f'''(0)=3f''(0);\quad f^{(5)}(0)=10f^{(4)}(0).$$

\section{Spectrum} Now we want to introduce the main definition of this chapter.
\begin{definition} Let $A$ be a subalgebra of finite codimension. Its  \textbf{spectrum} \index{spectrum} consists of $\alpha\in\MK$ such that either $f'(\alpha)=0$ for all $f(x)\in A$ or there exists $\beta\neq\alpha$ such that $f(\alpha)=f(\beta)$ for all $f(x)\in A$. In the second case $\beta$ obviously belongs to the spectrum as well. We write $Sp(A)$ to denote the spectrum of the algebra $A$.
 \end{definition}

  Unfortunately the word spectrum already has a specific meaning, so it would be more correct to use something like ``subalgebra spectrum'', but because we believe that this notion is very important and that the word spectrum reflects this concept very well we use the word ``spectrum''.  This makes our article more readable and in our context the interpretation should be unambiguous.

 We have already seen in  Theorem \ref{th:codm1} how the spectrum naturally arises in the description of subalgebras of codimension one.

  One trivial but useful remark is the following.
\begin{theorem}\label{lm:reversing} If $A\subseteq B$  are two subalgebras  in $\MK[x]$ then $Sp(B) \subseteq Sp(A)$. Thus the spectrum has the reversing inclusions property.
\end{theorem}
\begin{proof}Each condition that holds in $B$ hold in $A$ as well.
\end{proof}

 \begin{theorem} Each proper subalgebra $A$ in $\MK[x]$ has non-empty spectrum. 
 \end{theorem}

 \begin{proof} Induction and Theorem \ref{th:codimGorin} shows that $A$ is a subalgebra of an subalgebra of codimension $1.$ Then theorems \ref{lm:reversing} and \ref{th:codm1} finish the proof.
 \end{proof}

 One of our main results can be formulated as follows.
 \begin{theorem}\label{th:main} If $A$ is a proper subalgebra of finite codimension  then   only the values of $f(x)$ and finitely many of its derivatives $f^{(j)}(x)$ in the elements of the spectrum determine if $f(x)\in A.$
\end{theorem}

We will prove this later.  We already have done it for monomial subalgebras and for subalgebras of codimension one.

Before moving on we give some equivalent definitions of the spectrum.

\begin{theorem}\label{th:spectrumdef} Let $A$ be a subalgebra of finite codimension and $\alpha\in \MK$. The following is equivalent.
\item[(i)] $\alpha$ belongs to the spectrum of $A.$
   \item[(ii)] There exists $\beta\in \MK$ such that $(x-\alpha)(x-\beta)$ divides $f(x)-f(\alpha)$ for any $f(x)\in A$.
   \item[(iii)] There exists $\beta\in \MK$ and a SAGBI basis $G$ of $A$ such that $(x-\alpha)(x-\beta)$ divides each element in $G.$
   \item[(iv)] $\alpha$ belongs to the spectrum of the subalgebra $\langle p(x),q(x) \rangle$ for each pair of monic $p(x),q(x)\in A$ with relatively prime degrees.
\end{theorem}

\begin{proof} (ii) is a simple reformulation of (i). (Note that we can take $\beta=\alpha$ when the condition is $f'(\alpha)=0).$

(ii) implies  (iii) almost directly. We choose any SAGBI basis and replace each element $g$ by $g-g(\alpha)$ obtaining a new SAGBI basis.

(iii) implies (ii) because any $f(x) \in A$ can be subduced to a constant $c$. In each subduction step a polynomial divisible by $(x-\alpha)(x-\beta)$ is subtracted. Hence $f(x)-c$ is divisible by $(x-\alpha)(x-\beta)$. It is easy to see that we must have $c=f(\alpha)$.

By theorem  \ref{lm:reversing}  (i) implies  (iv). The opposite, that  (iv) implies (i)  is more difficult.
  If there exists $f(x)\in A$  such that $f'(\alpha)\neq 0$ we need to find $\beta.$ Subtracting a constant we can suppose that $f(\alpha)=0$ and let $\beta_1,\ldots,\beta_k$ be the other roots of $f(x),$
  which exist because $A$ is a proper subalgebra. Then $\beta$ should equal some $\beta_i.$ If the implication does not hold then for each $i$ there exists $g_i(x)\in A$ such that $g_i(\beta_i)\neq g_i(\alpha).$ Subtracting a constant we can suppose that
  $g_i(\alpha)=0, $ but $g_i(\beta_i)\neq 0.$ Now, using that our field is infinite, we can easily construct a linear combination $g(x)$ of the $g_i$, such that $g(\alpha)=0$ but $g(\beta_i)\neq 0$ for each $i.$ Since $A$ has a finite codimension we can for each large degree find a polynomial that belongs to $A.$ We choose such a monic polynomial $h(x)$ that has degree larger than $\deg g(x)$ and relatively prime to $\deg f(x).$ We can also suppose that $h(\alpha)=0.$

  The next step is to construct a polynomial $p(x)=h(x)+cg(x)$ that has the same property as $g(x),$ namely $p(\alpha)=0$ but $p(\beta_i)\neq 0$ for each $i.$ Again, this is possible because our field is infinite. Let $q(x)$ be $f(x)$ divided by its leading coefficient. Consider the subalgebra $\langle p(x),q(x)\rangle.$ Because $\alpha$ belongs to its spectrum and $q'(\alpha)\neq 0$ there exists $\beta$ such that $p(\alpha)=p(\beta)$ and $q(\alpha)=q(\beta).$
  But $q(\alpha)=0$ so $\beta=\beta_i$ for some $i.$ On the other hand $0=p(\alpha)\neq p(\beta_i)$ and we get a contradiction. This proves that our assumption that (iv) does not imply (i) must have been wrong.

\end{proof}

\section{The size of the spectrum}
How large can the spectrum of a subalgebra of finite codimension $n$ be? To answer this question we first prove an important statement, which essentially says that elements in the spectrum appears in a natural way and there are no ``ghost'' elements in the spectrum.
\begin{theorem}\label{th:ghost}Suppose that the subalgebra $A$ is obtained from the subalgebra $B$  by adding an extra condition $L(f(x))=0$ where  either  $L(f(x))=f(\alpha)-f(\beta)$ or
 $L$ is some $\alpha-$derivation. If $\lambda \not \in Sp(B) \cup \{\alpha,\beta\}$ then $\lambda \not \in Sp(A).$
\end{theorem}
\begin{proof} Suppose the opposite. Then for any $f=f(x)\in A$ we have $l(f)=0,$ where either $l(f)=f(\lambda)-f(\mu)$ or $l(f)=f'(\lambda).$

 We need to consider four different situations (two alternatives for $L$ and two for $l$). Let us first see what they all have in common. First of all $A=\ker L$, and we have supposed that $A=\ker l$ as well and want to get a contradiction.

Note that for any $f(x),g(x)\in B$ we have $$L(f)g-L(g)f\in \ker L=\ker l\Rightarrow L(f)l(g)-L(g)l(f)=0\Rightarrow$$
\begin{equation}\label{eq:Ll}
L(f)l(g)=L(g)l(f).
\end{equation}

 Our next step is to choose a SAGBI  basis $\{g_j\}$ for $B$ inside $M_\lambda=\{f(x)\in B|f(\lambda)=0\}.$ Because $\lambda$ is not in the spectrum we can find $g_i$ of minimal degree such that $l(g_i)\neq 0.$ Subtracting  it we can suppose WLOG that $l(g_j)=0$ for all $j\neq i.$ Note first that $L(g_i)\neq 0,$ otherwise $g_i\in A,$ which is impossible because $l(g_i)\neq 0.$

On the other hand for $j\neq i$ we have  $g_j \in \ker l=\ker L$, thus $L(g_j)=0.$ Note also that in any of the two alternatives $l(g_i^kg_j)=0$ because $g_i,g_j\in M_\lambda.$ Using (\ref{eq:Ll}) we get
  \begin{equation} \label{eq:Lk}L(g_i^kg_j)l(g_i)=l(g_i^kg_j)L(g_i)=0\Rightarrow L(g_i^kg_j)=0.\end{equation}

Now it is time to consider different alternatives. Suppose first that $L(f)=f(\alpha)-f(\beta).$ Then we get that for each $k$ we have
$$g_i(\alpha)^kg_j(\alpha)=g_i(\beta)^kg_j(\beta).$$
 For $k=0$ we get $g_j(\alpha)=g_j(\beta).$ Because $g_i(\alpha)\neq g_i(\beta)$ (otherwise $g_i\in A$) we should have
$$g_j(\alpha)=g_j(\beta)=0$$ for each $j\neq i.$ This implies $g_i(\alpha)\neq 0$, otherwise  we would have $f(\lambda)=f(\alpha)$ for all elements in our SAGBI basis and $\lambda$ would be an element of the spectrum of $B.$ Similarly we have $g_i(\beta)\neq 0.$

Let $k\ge 2.$ Using (\ref{eq:Ll}) again we get
$$L(g_i)l(g_i^k)=L(g_i^k)l(g_i).$$
Now we consider alternatives for $l.$ If $l$ is $\lambda$-derivation we get $l(g_i^k)=kg_i^{k-1}(\lambda)l(g_i)=0,$ because $g_i(\lambda)=0.$ Thus
$$L(g_i^k)=0\Rightarrow g_i(\alpha)^k=g_i(\beta)^k.$$ Let $a=g_i(\alpha),b=g_i(\beta).$ We have $a\neq b$ but $a^k=b^k$ for all $k\ge 2$ which is impossible and we get our first contradiction.

If instead $l(f)=f(\lambda)-f(\mu)$ then $c=l(g_i)=-g_i(\mu)$ and $c\neq 0.$
Then
we get from the equation above that
$$(a-b)c^{k}=(a^k-b^k)c\Rightarrow(a-b)c^{k-1}=a^k-b^k.$$ Because $a\neq b$ we get from $k=2,3$:
$$c=a+b; c^2=a^2+ab+b^2\Rightarrow (a+b)^2=a^2+ab+b^2\Rightarrow ab=0$$ which contradicts $a\neq0,b\neq 0$ obtained above.

Consider now the case when $L=D$ is some $\alpha-$derivation.

 Condition (\ref{eq:Lk}) now looks as $D(g_i^kg_j)=0$ and for $k=0$ implies $D(g_j)=0$ for $j\neq i.$ Because $g_i$ does not belong to $A=\ker D$
we have from $k=1$ that $$0=D(g_ig_j)=D(g_i)g_j(\alpha)+g_i(\alpha)D(g_j)\Rightarrow g_j(\alpha)=0.$$
This implies $a=g_i(\alpha)\neq 0,$ otherwise we would have $f(\alpha)=f(\lambda)$ in $B.$

Suppose first that $l(f)=f(\lambda)-f(\mu),$ thus and $l(g_i^k)=(-g_i(\mu))^k=c^k$ if we put  $c=l(g_i)=-g_i(\mu)\neq 0.$

Equation (\ref{eq:Ll}) gives now for $k\ge 2$ that
 $D(g_i^k)l(g_i)=D(g_i)l(g_i^k).$ Applying the same notations as above we rewrite this as
$$D(g_i^k)=D(g_i)c^{k-1}.$$ For $k=2,3$ we obtain
$$2D(g_i)a=D(g_i)c;\quad  3D(g_i)a^2=D(g_i)c^2\Rightarrow c=2a, c^2=3a^2$$ and we again get $a=0$ which is a contradiction.

It remains only the case where $l(f)=f'(\lambda).$ Here we use equation (\ref{eq:Ll}) again. Because $l(g_i^2)=2g_i(\lambda)l(g_i)=0$ we have
$$0=D(g_i)l(g_i^2)=l(g_i)D(g_i^2)=2l(g_i)g_i(\alpha)D(g_i)\Rightarrow g_i(\alpha)=0,$$ which is the last contradiction we needed.
\end{proof}

Now we get a nice corollary.
\begin{theorem}\label{th:spectrumsize} Let $A$ be a
 subalgebra  in $\MK[x]$ of codimension $n.$ Then
 \begin{itemize}
   \item $|Sp(A)| \leq 2n.$
   \item $|Sp(A)|=2n$ if and only if $A$ can be described by $n$ conditions of the form $f(\alpha_i)=f(\beta_i), i=1,\ldots, n$,  all $\alpha_i,\beta_i$ being different.
   \item $|Sp(A)|=2n-1$ if and only if $A$ can be described by $n-1$ conditions of the form $f(\alpha_i)=f(\beta_i), i=1,\ldots, n-1$ and one extra condition either of the form
    $f'(\alpha_0)=0$ or of the form $f(\alpha_0)=f(\alpha_1),$ all $\alpha_i,\beta_i$ being different. The second alternative is possible only if $n>1.$
 \end{itemize}

\end{theorem}
\begin{proof}The first two statements follow directly by induction from the previous theorem and Theorem \ref{th:codimGorin}. For the last statement we need to describe the induction in greater detail. For $n=1$ the statement is trivial. If $n>1$ and $A$ is obtained from $B$ by an extra condition
  then $|Sp(B)| \geq 2n-3$. If $|Sp(B)|=2n-3$ the extra condition is of the form $f(\alpha)=f(\beta),$ where $\alpha,\beta$ does not belong to the spectrum of $B$ and we can simply use the induction hypothesis. If $|Sp(B)|>2n-3$  then by Theorem~\ref{th:ghost} it must be  $2n-2$. If the extra condition is of the form $f(\alpha)=f(\beta)$ exactly one of $\alpha$ or $\beta$ should belong to the spectrum of $B.$ WLOG it coincides with $\alpha_1.$   Otherwise the extra condition is an $\alpha-$derivation for some $\alpha$ that does not belong to $Sp(B).$ Using
  Theorem \ref{th:derivativenotinspectrum} (which we will prove later) we can replace it by $f(x)\rightarrow f'(\alpha)$ and it remains to rename $\alpha$ to $\alpha_0.$
\end{proof}

\section{Clusters}
Let us now introduce a natural equivalence. For a given algebra $A$ we define $\alpha\sim \beta$ \index{$\alpha\sim \beta$} if and only if $f(\alpha)=f(\beta)$ is valid for all $f \in A.$ Then the spectrum of the subalgebra $A$ is a disjoint union of equivalency classes that we call \textbf{clusters.} \index{cluster} If $A$ is obtained from $B$ by a linear condition $L(f)=0$ then Theorem \ref{th:ghost} gives us a simple connection between clusters in $B$ and $A.$

If $L$ is an $\alpha-$derivation then the clusters are the same if $\alpha \in Sp(B)$ and $\{\alpha\}$ constitutes an additional cluster in $A$ if $\alpha \not \in Sp(B)$.

If $L(f)=f(\alpha)-f(\beta)$ there are several possibilities. If neither $\alpha$ nor $\beta$ belong to the spectrum of $B$ then they together form a new cluster.

If exactly one of them (say $\alpha$)  belongs to the spectrum of $B$ then we simply add $\beta$ to the cluster containing $\alpha.$

At last if both $\alpha$ and $\beta$ belong to the spectrum of $B$ then they should lie in different clusters and as a result those two clusters will be joined in $A.$


From now on we will use the notion
$A(C) = \{f(x) | f(\alpha)=f(\beta) \mbox{ for all } \alpha,\beta \in C\}$
for the subalgebra defined by the fact that all its elements have the same value on the cluster $C$.
\section{The Main Theorem }
Now we want to prove Theorem \ref{th:main}. We begin with the following.
\begin{theorem}\label{th:power} Let $A$ be a proper subalgebra of $\MK[x]$ with $Sp(A)=\{\alpha_1,\ldots,\alpha_s\}$ and let $\pi_A=(x-\alpha_1)\cdots(x-\alpha_s).$ Then there exists $N>1$
such that $x^i\pi_A^N\in A$ for any $i\ge 0.$
\end{theorem}
\begin{proof}We use induction on the codimension $n.$ The base for the induction is guaranteed by  theorem  \ref{th:codm1} so let $n\ge 2.$ Let $A$ be obtained from $B$ as the kernel of $L.$ Let $C=Sp(B),$ $\pi_B=\prod_{\gamma\in C}(x-\gamma)$ and $N_B$ be the  number $N$ for the subalgebra $B$  existing by the induction hypothesis. We consider several different cases.

Suppose first that $L(p)=p(\alpha)-p(\beta).$ We put $N=N_B.$

 If both $\alpha,\beta\in C$ then $\pi_A=\pi_B$. Because $N>0$ we get that all $x^i\pi_A^N\in \ker L=A$ directly.

If neither $\alpha$ nor $\beta$ belongs to the spectrum of $B$ then  $\pi_A=\pi_B(x-\alpha)(x-\beta).$  Note that $x^i\pi_A^N\in B$ and $x^i\pi_A^N\in \ker L=A.$

If only $\alpha\in C$ then $\pi_A=\pi_B(x-\beta)$ and  again  $x^i\pi_A^N \in \ker L=A$ directly.

If $L$ is an $\alpha-$derivation and $\alpha\not\in C$ then, as we will show in theorem  \ref{th:derivativenotinspectrum} later, $L(f)=cf'(\alpha).$ We have that $\pi_A=\pi_B(x-\alpha)$ and  put $N=N_B.$ Because $N\ge 2$
we get that the multiplicity of $\alpha$ is at least  two and $x^i\pi_A^N\in \ker L=A.$

At last if $L$ is an $\alpha-$derivation and $\alpha\in C$ then $\pi_A=\pi_B$ and we put $N=2N_B.$
Then $$L(x^i\pi_A^{2N})=L(x^i)\pi_A(\alpha)^{2N}+\alpha^i2N\pi_A(\alpha)^{2N-1}L(\pi_A)=0.$$

In all cases we get that $x^i\pi_A^N\in \ker L=A.$

\end{proof}

\begin{theorem}\label{th:main2} Let $A$ be a subalgebra of codimension $n> 1$  with $Sp(A)=\{\alpha_1,\ldots,\alpha_s\}.$
\begin{enumerate}
\item Then there exists $N>1$ such that $A$ can be described by $n$ conditions of the form
$$\sum_{i=0}^{N-1}\sum_{j=1}^s c_{ij}p^{(i)}(\alpha_j)=0.$$ Thus $p(x)\in A$ if and only if all $n$ conditions are valid.
\item If $A$ has only one cluster then we can choose $s-1$ conditions as $f(\alpha_1)=f(\alpha_j)$ for $j>1$ and the remaining as
$$\sum_{i=1}^{N-1}\sum_{j=1}^s c_{ij}p^{(i)}(\alpha_j)=0,$$ thus using pure derivatives (of some order).
\end{enumerate}

\end{theorem}

\begin{proof} $(1).$ We use the same notations as in Theorem \ref{th:power}. According to that theorem we have polynomials in $A$ of each degree greater then $Ns-1.$  If we complete them to a linear basis in
$A$ we get a set $Q,$ consisting of exactly $Ns-n$ new polynomials $q$ and we can suppose that $1\in Q.$

 Consider the vector space $V$ consisting of linear
 maps $$D:p(x)\rightarrow \sum_{i=0}^{N-1}\sum_{j=1}^s c_{ij}p^{(i)}(\alpha_j).$$  We have that $\dim V=Ns$.
 Consider its subspace $W$ of those maps that annihilate all $q\in Q.$ The subspace $W$ has dimension $n$ (because the condition $D(q)=0$ is a homogeneous linear equation on the set of the coefficients $c_{ij}).$
 We choose a basis  in $W$ consisting  of $n$ maps $D$ and claim that the conditions $D(p)=0$ for each $D$ from this basis describes $A.$ Indeed those conditions by construction describes exactly the subspace
 generated by $q\in Q$ in the subspace of all the polynomials of the degree less then $Ns.$ It remains to show that each $x^i\pi_A^N$ is annihilated by $D.$

 Let $D_0$ be the map
 $$D_0:p(x)\rightarrow \sum_{j=1}^s c_{0j}p(\alpha_j).$$ Because $\pi_A(\alpha_j)=0$ for each $j$ we have that $D_0(x^i\pi_A^N)=0$ and it is sufficient to consider $D_1=D-D_0$ consisting of only the derivatives.
   $D_1$ annihilate all the elements of the form $x^i\pi_A^N$ because it has derivatives of degree at most $N-1$ and the same is true for $D.$

   Thus our conditions are valid on all basis elements in $A$ and describe the vector space they generate, which is $A.$ In other words the conditions that $E_i(p(x))=0$ for our basis elements $E_i\in W$ determine the subalgebra $A$. Note that this automatically implies that we get subalgebra conditions.

   $(2)$ If $A$ has only one cluster then we have conditions $F_j(p(x))=0,$ where $F_j:p(x)\rightarrow p(\alpha_1)-p(\alpha_j)$ for all $j>1.$ This means that $F_j\in W$ and we can choose them as a part of the basis in $W$ (no one if $s=1$).

 Because $D(1)=0$ we get that $\sum_j c_{0j}=0$ and therefore the part $D_0$ in the previous proof can be written as $\sum c_{0j}F_j$. Thus we can replace $D$ by $D-D_0,$ which is a linear combination of pure derivatives of some order, so our $k$ elements in the basis of $W$ are exactly what we are looking for: either $F_j$ or linear combinations of pure derivatives (of some order $>0$).

 \end{proof}

 So, assuming Theorem \ref{th:derivativenotinspectrum} which we will prove later, we have now proven our main theorem.

 \chapter{Characteristic polynomial}
 Now we want to understand how to find the spectrum. We start from a special case.
 \section{Subalgebra $<p,q>$}

 Let $p(x),q(x)$ be two monic polynomials. Consider the following polynomials in two variables:
 $$P(x,y)=\frac{p(x)-p(y)}{x-y},\ Q(x,y)= \frac{q(x)-q(y)}{x-y}.$$
 We now introduce a notation that will be helpful when searching the spectrum of the subalgebra generated by $p$ and $q$.
 \begin{definition}The \textbf{characteristic polynomial} \index{$\chi_{p,q}$} \index{characteristic polynomial} $\chi_{p,q}$ is the resultant  $$\chi_{p,q}(x)=Res_y(P(x,y),Q(x,y))$$ of polynomials $P$ and $Q$ considered as polynomials in $y.$
 \end{definition}

For example, if $p(x)=x^3-x,q(x)=x^2$ then  $P(x,y)=y^2+yx+x^2-1, Q(x,y)=y+x$ and
 $$\chi_{p,q}(x)=\left|
               \begin{array}{ccc}
                 1   & x & x^2-1 \\
                 1 & x & 0 \\
                 0 & 1 & x \\
               \end{array}
             \right|=x^2-1.$$
 Its roots are $1$ and $-1$ and this gives some insight into why $f(1)=f(-1)$ was the subalgebra condition for $A=\langle x^3-x,x^2\rangle.$

 It is easy to check that get $\chi_{x^3,x^2}(x)=x^2$ and this can be easily generalised, as shown below.

 \begin{theorem} If  $(m,n)=1$ then $\chi_{x^m,x^n}(x)=x^{(m-1)(n-1)}$.
 \end{theorem}
 \begin{proof}
 Assume without loss of generality that $n>m$. First note that the polynomials $P(x,y) = \frac{x^n - y^n}{x-y}$ and $Q(x,y) = \frac{x^m - y^m}{x-y}$ can be expressed as $P = \sum_{i=0}^{n-1} y^ix^{n-1-i}$, $Q = \sum_{i=0}^{m-1} y^ix^{m-1-i}$ respectively. This means that
 $$\chi_{p,q}(x) = \left|
 		\begin{array}{ccccccc}
 		1 & x & \ldots & x^{n-1} & 0 & 0 & \ldots \\
 		0 & 1 & \ldots	& x^{n-2} & x^{n-1} & 0 & \ldots \\
 		\vdots & \ddots &  & \ddots & \ddots & \ddots & \\
 		0 & \ldots & \ldots & 1 & x & \ldots & x^{n-1} \\	
 		1 & x & \ldots & x^{m-1} & 0 & 0 & \ldots \\
 		0 & 1 & \ldots	& x^{m-2} & x^{m-1} & 0 & \ldots \\
 		\vdots & \ddots &  & \ddots & \ddots & \ddots & \\
 		0 & \ldots & \ldots & 1 & x & \ldots & x^{m-1} \\	
 		\end{array}
 \right|.
$$
If $m = 1$, this determinant is upper triangular and equal to $1 = x^{(m-1)(n-1)}$. This will be the base case for a proof by induction. If $m \neq 1$, for $i \in \{1, ..., m-1\}$ subtract row $m - 1 + i$ from row $i$. Now rows $1, ..., m-1$ will have $x^{m}$ as first nonzero element, in column $m+i$. Break out a factor $x^m$ from each of these rows. Now, after rearranging, $\chi_A(x)$ is a block determinant on the form
$$
\left| \begin{array}{cc}
A & B \\
0 & C \\
\end{array}
\right|
$$
where A is an upper triangular $(m-1)$-matrix with ones on the main diagonal. Expanding the determinant along the first column $m-1$ times and rearranging gives
$$
(x^m)^{m-1}\left|
\begin{array}{ccccccc}
1 & x & \ldots & x^{m-1} & 0 & 0 & \ldots \\
0 & 1 & \ldots	& x^{m-2} & x^{m-1} & 0 & \ldots \\
\vdots & \ddots &  & \ddots & \ddots & \ddots & \\
0 & \ldots & \ldots & 1 & x & \ldots & x^{m-1} \\	
1 & x & \ldots & x^{n-m-1} & 0 & 0 & \ldots \\
0 & 1 	& x & \ldots & x^{n-m-1} & 0 & \ldots \\
\vdots & \ddots &  & \ddots & \ddots & \ddots & \\
0 & \ldots & \ldots & 1 & x & \ldots & x^{n-m-1} \\	
\end{array}
\right|
$$
which is of size $(n-2)$. Note that this is exactly the characteristic polynomial of $\langle x^m, x^{n-m}\rangle $ multiplied by $(x^m)^{m-1}$. Assuming, by induction hypothesis, that $\chi_{\langle x^m, x^{n-m}\rangle }(x) = x^{(m-1)(n-m-1)}$ gives $\chi_A(x) = x^{m(m-1)}x^{(m-1)(n-m-1)} = x^{(m-1)(n-1)}$.
The induction hypothesis can be used since
$(n-m,m)=(n,m)=1.$
 \end{proof}

\begin{theorem}\label{th:partialderivative} If $m=\deg p(x), n=q(x)$  and $(m,n)=1$ then
\begin{itemize}
  \item $\chi_{p,q}(x)$ is a polynomial of degree $(m-1)(n-1).$
  \item If $F(p,q)$ is the resultant of $p(x)-p,q(x)-q$ then $$\frac{\partial F}{\partial p}|_{p=p(x),q=q(x)}=\pm\chi_{p,q}(x)q'(x).$$
   $$\frac{\partial F}{\partial q}|_{p=p(x),q=q(x)}=\mp\chi_{p,q}(x)p'(x).$$

\end{itemize}

 \end{theorem}
 \begin{proof}
 Let us look at the monomial case above again. In a complete expansion of the determinant we choose in each column $j$ either $x^{j-i}$ (if we choose  row $i$ from the first $m-1$ rows
 ) or we choose
 $x^{j-i+(m-1)}$ (if we choose a row $i$ between the last $n-1$ rows). Because $\sum j=\sum i $ we get a total degree in the product equal to $(n-1)(m-1).$ We can never get larger degree. The difference when we use $p(x)$ and $q(x)$
 instead is that we add some terms of smaller degree  in each element of the matrix. But they cannot effect  our maximum total degree term $x^{ (n-1)(m-1)}$ so the highest coefficient in $\chi_{p,q}(x)$ at $x^{ (n-1)(m-1)}$
 is the same as for the monomial case.

 To prove the second statement we use a well-known fact (see \cite{Bur}) that
  $$F(p,q)=\prod_{\alpha}p(\alpha)-p$$ where the product is taken over all roots of $q(y)-q$ in some field extension and multiplicity.  When we evaluate this for $p=p(x)$ and $q=q(x)$ we get zero because $y=x$ is one of the roots. If we take a partial derivative over $p$ first and evaluate in  $p=p(x)$ and $q=q(x)$ after that we get a sum over roots where all terms except one (corresponding the root  $y=x$) are zero.  But we can get this remaining term in another way
 if we replace $q(x)-q$ by $\frac{q(y)-q(x)}{y-x}$ and $p(y)-p$ by $p(y)-p(x)$ in our resultant.
 Thus (up to sign) we get the resultant $\operatorname{Res}_y\left(p(y)-p(x), \frac{q(y)-q(x)}{y-x}\right).$

Now, using another property of the resultant we get
 $$\operatorname{Res}_y\left(p(y)-p(x), \frac{q(y)-q(x)}{y-x}\right)=$$
 $$\operatorname{Res}_y\left(\frac{p(y)-p(x)}{y-x}, \frac{q(y)-q(x)}{y-x}\right)\operatorname{Res}_y\left(y-x,\frac{q(y)-q(x)}{y-x}\right)$$
 $$=\chi_{p,q}(x)q'(x)
,$$ where all resultants above are evaluated in $y.$ Here we have also used that for any polynomial $f(x)$ we have $f'(x)=\frac{f(x)-f(y)}{x-y}|_{y=x}$ because this is obviously true  for $f(x)=x^k.$
The second formula we obtain in a similar way and  the signs should be opposite because $(F(p(x),q(x))'=F'_pp'(x)+F'_qq'(x)$ should be zero.

 \end{proof}

 We have seen that $(m,n)>1$ then $\chi_{x^m,x^n}=0$, and we will now generalize this result.

 \begin{theorem}\label{chizeroforcomp} Let $p(x)$ and $q(x)$ be non-constant polynomials.
  Then  $\chi_{p,q}(x)=0$ if and only if there exists a polynomial $h(x)$ of degree at least two such that $p(x),q(x)\in \MK[h]$.

 \end{theorem}
\begin{proof} Suppose first that $p=\pi \circ h.$ We know $\pi(a)-\pi(b)=(a-b)\rho(a,b)$ for some $\rho$ so
$$p(x)-p(y)=\pi(h(x))-\pi(h(y))=(h(x)-h(y))\rho(h(x),h(y)).$$ This means that $P(x,y)=\frac{p(x)-p(y)}{x-y}$ has a factor $\frac{h(x)-h(y)}{x-y}$ which is a polynomial in $y$ of degree at least one.
Similarly if $q(x)\in \MK[h]$ then $Q(x,y)$ also has this factor so they have a common factor  as polynomials in $y$ over $\MK(x)$ and as a consequence their resultant $\chi_{p,q}(x)$ is equal to zero.

To prove the opposite assume now that $\deg p(x)=n$ and $\deg q(x)=m$. Let $F(p,q)$ be the resultant of $p(x)-p,q(x)-q$, as before. We know from lemma 19 in \cite{tor} that $F(p,q)=\sum_{in+jm \leq nm}c_{ij}p^iq^j$ where $c_{ij}$ are constants in $\MK$. Moreover, it follows from that lemma that $p^m$ has non-zero coefficient and all other terms contain $p$ to a power strictly lower than $m$. Assume now that $\chi_{p,q}(x)=0$. Then it follows from Theorem \ref {th:partialderivative} that we can differentiate $F$ with respect to $p$ and get another identity involving $p$ and $q$. Regarding $p$ as variable this identity is a polynomial of degree $m-1$ with coefficients in $\MK(q)$, showing that adjoining $p$ to the field $F(q)$ is an extension of degree at most $m-1$. From lemma 13 in \cite{tor} we get the first equality in $m=[\MK(x):\MK(q)]=[\MK(x):\MK(p,q)][\MK(p,q):\MK(q)]$. Now it follows that $[\MK(x):\MK(p,q)] \geq 2$. On the other hand we know by theorem 14 in \cite{tor}  that $\MK(p,q)=\MK(h)$ for some polynomial $h$ and this means that we have a polynomial $h$ of degree $[\MK(x):\MK(p,q)] \geq 2$ such that $p(x),q(x)\in \MK[h]$.
	\end{proof}

\section{How the spectrum relates to $\chi_{p,q}(x)$}
Now we want to compare   the roots of the \cp  with the spectrum.

 To start with we will focus our attention on a special case - an algebra $A$ generated by two monic polynomials $p(x),q(x)$ of degrees $m>n$ with $(m,n)=1.$ It is known that they form SAGBI basis for $A$ (see \cite{tor}) and therefore $A$ has codimension \index{$g(m,n)$}$g(m,n)=(m-1)(n-1)/2$. (Here $g(m,n)$ is the genus of the corresponding semigroup of degrees.) So if we want to describe this algebra we need to find  $g(m,n)$ subalgebra conditions. For $m=3,n=2$ we have done that in Theorem \ref{th:codm1}.

 \begin{theorem}\label{th:spectrum} Let $A=\langle p(x),q(x)\rangle$ and $\alpha\in \MK.$ The following is equivalent.
\begin{description}
 \item[(i)] $\alpha$ belongs to the spectrum, thus either $f'(\alpha)=0$ for any $f(x)\in A$ or there exists $\beta\neq\alpha$ such that $f(\alpha)=f(\beta)$ for any $f(x)\in A$.
   \item[(ii)] Either $p'(\alpha)=q'(\alpha)=0$ or there exists $\beta\neq\alpha$ such that $p(\alpha)=p(\beta)$ and $q(\alpha)=q(\beta).$
   \item[(iii)]  $\alpha$ is a root of the  \cp of $A.$
 \end{description}

 \end{theorem}
 \begin{proof}
 	The alternatives (i) and (ii) are equivalent since each of the two conditions stated in (ii) are closed under sums and products, so we need only to prove that (i) and (iii) are equivalent.
 	By the  fundamental property of the resultant (see e.g.\cite{Bur}) we know that $\alpha$ is a root of the \cp if and only if there is some $\beta \in \MK$ such that $P(\alpha,\beta)=Q(\alpha,\beta)$.
 	
 	We now regard two different cases. The first case is when $\beta \neq \alpha$. In this case we have that $p(\alpha)-p(\beta)=(\alpha-\beta)P(\alpha,\beta)=0$ and similarly $q(\alpha)=q(\beta)$. Thus the second statement of (ii) holds.

 	The other case is that $\alpha=\beta$ which means that $0=P(\alpha,\alpha)=p'(\alpha)$. (The second equality can easily be derived from the definition of $P$ as $P(x,y)=(p(x)-p(y))/(x-y)$.) In the same manner we find that $q'(\alpha)=0$ so in this case the first statement of (ii) holds.
 	
 \end{proof}

 This shows that the \cp allows us to find the spectrum explicitly, for the subalgebras we currently study.
 Note that the theorem also shows that the \cp is never a constant, because the spectrum is always non-empty.

 Also note that Theorem \ref{th:spectrumdef} gives us a theoretical way to find the spectrum for any subalgebra. In most practical cases it is sufficient to consider only $\chi_{p,q}$ for each pair $\{p,q\}$ of generators, but the problem is that their degrees are not always relatively prime.

Here is another  application of the theorem.
 \begin{theorem} If $a(x)$ is a polynomial of degree at least two that divides both $p(x)$ and $q(x)$ then all the roots of $a(x)$ are roots of $\chi_{p,q}(x)$.

 \end{theorem}
\begin{proof}If $(x-\alpha)(x-\beta)|a(x)$ then $(x-\alpha)(x-\beta)|f(x)-f(\alpha)$ for any $f(x)\in A$ because $p$ and $q$ generate $A$ and are divisible by $a(x).$ The rest follows from theorems \ref{th:spectrumdef} and \ref{th:spectrum}.

\end{proof}

It would be interesting to know if the following is true.
\begin{Conj}
$a(x)|\chi_{p,q}(x)$.
\end{Conj}

To find the algebraic conditions is less straightforward, but knowing the spectrum helps a lot. Consider for example,  the  subalgebra $A=<x^4-x^2,x^3>.$

We see that $x^2$ divides both generators so it should divide
the \cp as well. Thus zero is in the spectrum. Moreover $f'(0)=0$ is valid for both generators and therefore is one of the conditions. Because $g(4,3)=3$ we should find two extra subalgebra conditions.
The \cp can be found using Maple and it is equal to   $x^2(x^4-x^2+1).$

Thus, besides zero we have four other elements in the spectrum, which are in fact primitive roots of degree $12.$ If we name one of them $\varepsilon,$ the remaining ones will be $\varepsilon^5,\varepsilon^7,\varepsilon^{11}.$
We cannot find more conditions involving derivations so we need to arrange those roots in pairs to get the conditions of the form $f(\alpha)=f(\beta).$ It is easy to check now that we get the example \ref{ex:ex3} in the very beginning of this article.

\section[Derivations in $<p,q>$]{Derivations in a subalgebra generated by two polynomials}\label{sec:dir2}
Now we want to formulate some general statements about possible derivations of subalgebras $A$ generated by two polynomials $p(x)$ and $q(x)$ of relatively prime degrees.
As we know (see \cite{tor}) $p(x), q(x)$ form SAGBI basis and have one relation $F(p,q)=0$ arising from the corresponding resultant. Our aim is to study
possible non-zero derivations $D:A\rightarrow\MK$, thus linear maps such that
$$
D(f(x)g(x)=D(f(x))g(\alpha)+f(\alpha)D(g(x))$$
for some $\alpha\in\MK.$

Denote $D(p(x))=Dp$ and $D(p(x))=Dq$.
Note first that for any polynomial $G(p,q)$ we have $$D(G(p(x),q(x))=\frac{\partial G}{\partial p}(p(\alpha), q(\alpha))Dp+\frac{\partial G}{\partial q}(p(\alpha), q(\alpha))Dq.$$
If we denote $\frac{\partial F}{\partial p}(p(\alpha), q(\alpha))$ by $F'_p(\alpha)$ and $\frac{\partial F}{\partial q}(p(\alpha), q(\alpha))$ by $F'_q(\alpha)$ then we get that
$$F'_p(\alpha)Dp+F'_q(\alpha)Dq=0$$
is necessary and sufficient condition for a linear map $D$ to be a derivation of $A.$

Note also that taking ordinary derivative in $\alpha$ we get
$$F'_p(\alpha)p'(\alpha)+F'_q(\alpha)q'(\alpha)=0.$$

Depending on $\alpha$  three different situations are possible.
\begin{enumerate}
  \item $p'(\alpha)=q'(\alpha)=0.$ This is the most difficult case. But at least $\alpha$ belongs to the spectrum.
  \item $(p'(\alpha),q'(\alpha))\neq (0,0),$ but $F'_p(\alpha)=F'_q(\alpha)=0.$ By Theorem \ref{th:partialderivative} we get that $\chi_{f,g}(\alpha)=0$ thus $\alpha$ belongs to the spectrum and there exists $\beta\neq\alpha$
  such that $f(\alpha)=f(\beta)$ for all $f(x)\in A.$ We can try to construct $D$ as $Af'(\alpha)+Bf'(\beta).$ We get
  $$Ap'(\alpha)+Bp'(\beta)=Dp; Aq'(\alpha)+Bq'(\beta)=Dq$$ and find a solution if the determinant
  $$\left|
      \begin{array}{cc}
        p'(\alpha) & p'(\beta) \\
        q'(\alpha) & q'(\beta) \\
      \end{array}
    \right|\neq 0.$$

  \item $(p'(\alpha),q'(\alpha))\neq (0,0),\ (F'_p(\alpha),F'_q(\alpha))\neq (0,0).$ This is the easiest case because it means that $(Dp,Dq)=C(p'(\alpha),q'(\alpha))$ and
  we simply have $D(f(x))=Cf'(\alpha).$
\end{enumerate}

Now we can obtain an important corollary that we will need to finish the proof of our main theorem.

\begin{theorem}\label{th:derivativenotinspectrum}
Let $A$ be an arbitrary subalgebra of finite codimension and $D$ be an $\alpha-$derivation on $A$. Suppose that $\alpha$ does not belong to the spectrum of $A$ (we call such derivations \textbf{trivial}). \index{trivial derivation} Then there exists $c\in\MK$ such that $D(f(x))=cf'(\alpha)$
for any $f(x)\in A.$

\end{theorem}
\begin{proof} First we prove that if $\alpha$ is a double root of $f(x)$ then $D(f(x))=0.$ Suppose the opposite. Let $\beta_1,\ldots,\beta_k$ be the other roots of $f(x)$.
   For each $i$ there exists $g_i(x)\in A$ such that $g_i(\beta_i)\neq g_i(\alpha).$ Subtracting a constant we can suppose that
  $g_i(\alpha)=0, $ but $g_i(\beta_i)\neq 0.$ Beside that there exists $g_0(x)\in A$ such that $g(\alpha)=0,$ but $g_0'(\alpha)\neq 0$ (all this because $\alpha$ does not belong to the spectrum).  Now, using that a field is infinite, we can easily construct a linear combination $g(x)$ of the $g_i$ such that $g(\alpha)=0$ but $g(\beta_i)\neq 0$ for each $i>0$ and $g'(\alpha)\neq 0.$ Since $A$ has a finite codimension we can for each large degree find a polynomial of that degree that belongs to $A.$ We choose such a monic polynomial $h(x)$ that has degree larger than $\deg g(x)$ and relatively prime to $\deg f(x).$ We can also suppose that $h(\alpha)=0.$

  Our next step is to a construct polynomial $p(x)=h(x)+cg(x)$ that has the same property as $g(x),$ namely $g(\alpha)=0,g'(\alpha)\neq 0$ and $g(\beta_i)\neq 0$ for each $i>0.$ Again, this is possible because our field is infinite.
   Let $q(x)$ be $f(x)$ divided by its leading coefficient. Consider subalgebra $B=\langle p(x),q(x)\rangle.$ By construction $\alpha$ does not belong to its spectrum, so the restriction of $D$ to $B$ should fit the third
   case considered above and therefore $D(f(x))=cf'(\alpha)=0$ which is a contradiction.

  The rest is easy.  Any polynomial in $A$ can be written as a linear combination of $g_0(x)$, some constant and some polynomial $f(x)$ having $\alpha$ as double root. Therefore only value on
  $g_0(x)$ determine the value of $D$, so it is sufficient to find $c$ such that $D(g_0(x))=cg_0'(\alpha).$ \end{proof}

\section[About $\chi_A$]{About \cp $\chi_A(x)$}

 We would like to generalize  the theorem  \ref{th:spectrum} to arbitrary subalgebras. For this we need to define \cp for an arbitrary subalgebra.

 Let us look at the case where $A$ has more than two generators.  It is not evident how to extend the definition. The resultant is defined only for pairs of polynomials.
 A  naive attempt is to use a gcd of all $\chi_{g_i,g_j}$ where $g_i$ generate $A.$
Let us first look at an example:

\begin{example}
	Let $p(x)=x^{12}+3x^6$, $q(x)=x^{15}$ and $r(x)=x^{10}$ and $A=\langle p(x),q(x),r(x) \rangle$ the subalgebra they generate. We can form the characteristic polynomial of any pair of generators. If we look at the pair $p$ and $q$ for example, it is obvious that they both belong to $\MK[x^3]$. Hence their characteristic polynomial is zero by Theorem \ref{chizeroforcomp}. In the same way the other two pairs of generators have zero as characteristic polynomial. In contrast, if we form $P$ and $Q$ as before and additionally $R(x,y)=(r(x)-r(y))/(x-y)$, then $P(x,y)=Q(x,y)=R(x,y)=0$ has only a finite set of solutions. In particular the possible $x$-values are the 24 solutions of $x^{24}+6x^{18}+26x^{12}+81x^6+81$ and $x=0$. (This can be obtained by solving the system in for example Maple.)
\end{example}

 The above example suggests that looking at pairs of generators of the algebra is not enough to define the characteristic polynomial in a suitable way. Instead we can try another definition:

 \begin{definition}Let $A$ be a subalgebra of finite codimension.  Its \textbf{characteristic polynomial} \index{$\chi_A(x)$}\index{characteristic polynomial} $\chi_A(x)$ we define as a gcd of all $\chi_{p,q}(x)$  where $p$ and $q$ are monic polynomials in $A$ with relatively prime degrees.
 \end{definition}

Another alternative is to consider all possible pairs $p,q$.

Note that in any case $\chi_A(x)$ is non-zero, because $A$ contains polynomials of relatively prime degrees.

A third alternative definition is given below.

Note that  we get by theorems \ref{th:spectrumdef} and \ref{th:spectrum} that $\chi_A(\alpha)=0$ if and only if $\alpha$ belongs to the spectrum of $A.$

It is not so obvious in the second alternative, but we can consider a third one.

Let us first assume, for simplicity, that we have three generators $A = \langle p(x),q(x) ,r(x) \rangle$. We also assume that $\deg q(x) \geq \deg r(x)$. Introduce $P(x,y), Q(x,y)$ as before and analogously $R(x,y)$. Then form the resultant $R(x,y,z)=\operatorname{Res}_y\left(P(x,y),zQ(x,y)+wR(x,y) \right)$. An $x$-value $x=\alpha$ that makes this resultant disappear for all values of $z$ and $w$ means an $x$-value for which there is some $y=\beta$ such that $P(\alpha,\beta)=0$ and $zQ(\alpha,\beta)+wR(\alpha,\beta)=0$ regardless of the values of $z$ and $w$. In other words $P(\alpha,\beta)=Q(\alpha,\beta)=R(\alpha,\beta)=0$.

Now it follows from the construction of the resultant as a certain determinant and the fact that the determinant depends linearly on the columns of the matrix that $R$ can be written as $R(x,z,w)=\sum_{j=0}^{n-1} d_j(x)z^{n-1-j}w^{j}$. Here $d_j(x)$ is a polynomial in $\MK[x]$ that can be computed by starting from the resultant-matrix of $P$ and $Q$, then replace $j$ columns of coefficients from $Q$ by the corresponding coefficients of $R$. Finally sum over all choices of $j$ such column replacements. That sum of determinants equals $d_j(x)$. The $x$-values $x=\alpha$ that make $R(\alpha,z,w)=0$ are those which satisfy $d_j(\alpha)=0$ for each $j$ or equivalently those $x=\alpha$ that are zeroes of $d(x)=gcd(d_1(x),d_2(x), \ldots ,d_{n-1}(x))$. It is straightforward to generalise this idea to more than three generators. We therefore make the following definition:
\begin{definition} Let $A=\langle p_1(x), p_2(x), \ldots ,p_t(x) \rangle$ and $n=\deg p_1(x)$. Further, let $P_i(x,y)=(p_i(x)-p_i(y))/(x-y)$ and
	$$R(x,z_2,z_3, \ldots, z_t)=$$ $$=\operatorname{Res}_y\left(P_1(x,y), z_2P_2(x,y)+z_3P_3(x,y)+ \cdots + z_tP_t(x,y)\right).$$ Then $R$ can be expressed as
	\begin{equation}\label{multressum}R(x,z_2,z_3, \ldots , z_t)=\sum d_{(a_2,a_3, \ldots , a_t)}(x)z_2^{a_2}z_3^{a_3} \cdots z_t^{a_t},\end{equation} where the sum is taken over all natural numbers $a_i$ satisfying $a_2+a_3+ \cdots a_t=n-1$
	The \cp of $A$ is given by $\chi_A(x)=gcd(\{d_{(a_2,a_3, \ldots , a_t)}\})$ where gcd is taken over the set of all polynomials $d_{(a_2,a_3, \ldots , a_t)}$ occurring in the sum (\ref{multressum}).
	
\end{definition}	

This definition looks complicated, so let us see how it works in our previous example
\begin{example}
	Let $p,q,r$ and $P,Q,R$ be as in the previous example. In this case we need to compute the resultant $$=\operatorname{Res}_y\left(P(x,y), zQ(x,y)+wR(x,y)\right)=\sum_{j=0}^{10} d_{(10-j,j)}(x)$$ (Here we have replaced $z_2$ by $z$ and $z_3$ by $w$ to improve readability.) Computing this in Maple we obtain:
	$$d_{(11,0)}=d_{(10,1)}=0$$
	 $$d_{(9,2)}=4x^{60}a(x)^2b(x)^3$$ $$d_{(8,3)}=18x^{55}a(x)b(x)^2c(x)$$  $$d_{(7,4)}=3x^{50}b(x)d(x)$$
	  where $a(x)=2x^6+3$, $b(x)=x^{24} + 6x^{18} + 36x^{12} + 81x^6 + 81$, $c(x)=2x^{30} + 5x^{24} + 30x^{18} + 90x^{12} + 135x^6 + 81$ and $d(x)=52x^{60} + 300x^{54} + 2025x^{48} + 8100x^{42} + 24300x^{36} + 65610x^{30} + 153090x^{24} + 262440x^{18} + 295245x^{12} + 196830x^6 + 59049$ Thus the gcd of the first five polynomials $d_{(11-j,j)}$ is $x^{50}b(x)$. One can check that the remaining $d_{(11-j,j)}$ also are divisible by  $x^{50}b(x)$. (In particular $d_{(0,11)}=0$, while the other polynomials are non-zero.)
	
\end{example}

\chapter{Applications}
Now we want to show some applications of the spectrum.
\section{One element in the spectrum}
We start from the subalgebras which have only one element in the spectrum.

\begin{theorem} Let $A$ be a  subalgebra of codimension $k\ge 1.$ The following is equivalent.
	\begin{enumerate}
		\item The spectrum of $A$ consists of a single element $\alpha.$
		\item $A$  contains two elements $(x-\alpha)^m, (x-\alpha)^n$ with $(m,n)=1.$
		
		\item $A$ is defined by $k$ linearly independent conditions of the form
		$\sum_{i=1}^{N}c_if^{(i)}(\alpha)=0$ for some $N>0.$
	\end{enumerate}

\end{theorem}
\begin{proof}
	
	We can use induction on $k.$ The base for the induction is guaranteed by  theorem  \ref{th:codm1}.  Let $k\ge 2.$ Using the change of variable $\widehat{x}=x-\alpha$ we can restrict ourself to the case $\alpha=0.$
	
	$(1)\Rightarrow (2).$ According to Theorem \ref{th:codimGorin} the algebra $A$ is obtained from $B$ as a kernel of some linear map. This map should be $0-$derivation $D$, otherwise we have more than one element in the spectrum.
	By Theorem \ref{lm:reversing}, $B$ should have zero spectrum and according to the induction hypothesis $B$ contains some monomials $x^m,x^n$ with $(m,n)=1.$ Note that $m,n>1$ because $B$ is a proper subalgebra.
	Using that $D(f^k)=kf^{k-1}(0)D(f)$ we find that the monomials $(x^m)^m=x^{m^2},\ (x^n)^n=x^{n^2}$ belong to the $\ker D=A.$
	
	$(2)\Rightarrow (1).$ Because subalgebra generated by $x^m$ and $x^n$ has spectrum zero, by  Theorem \ref{lm:reversing} the spectrum of $A$ cannot have any other elements than zero.
	
	$(1)\Rightarrow (3)$ Follows from Theorem \ref{th:main}
	
	$(3)\Rightarrow (2)$ All the monomials $x^m$ with $m>N$ satisfy the conditions.
	
\end{proof}

\section[Polynomial of degree $2$]{Subalgebras containing a polynomial of degree  $2$}
Suppose that the subalgebra $A$ contains a polynomial $q(x)$ of degree two. Two trivial cases are $A=\langle q(x)\rangle $ and $A=\MK[x].$ In non-trivial cases we should have a polynomial $p(x)$ of odd degree $2l+1\ge 3.$
If we suppose that $l$ is as small as possible then it is easy to see that $A=\langle p(x),q(x)\rangle .$
Using variable substitution  we can suppose that $q(x)=x^2.$ Subtracting even terms
we can WLOG suppose that $p(x)$ is an odd polynomial, thus $$p(x)=a(x^2)x,\ q(x)=x^2$$
for some monic polynomial $a(x)$ of degree $l.$
We want to show that the spectrum  of $A$ consists of the roots of $a(x^2).$ (In fact the \cp is  equal to $a(x^2)$ but that requires longer proof).

Indeed, if $q'(\alpha)=0$  then $\alpha=0$ and $p'(0)=0$ implies $a(0)=0.$ If $q(\alpha)=q(\beta)$ for $\alpha\neq \beta$ then $\beta=-\alpha$ and  $p(\alpha)=p(-\alpha)=-p(\alpha)$ implies $p(\alpha)=0\Rightarrow a(\alpha^2)=0.$

Now we are ready for general statement.
\begin{theorem}\label{th:deg2} Any  proper subalgebra $A$ of finite index in $\MK[x]$ containing a polynomial $q(x)$ of degree two has a spectrum consisting of $g>0$ elements for some $g$. The spectrum has $k=\left[\frac{g}{2}\right]$
	pairs $\{\alpha_i,\beta_i\},$   $i=1,\ldots, k$ such that for each $i$ the sum $\alpha_i+\beta_i$ has a constant value $2\alpha_0$ and (for odd $g$) one extra element, namely $\alpha_0(=\beta_0).$  For each $0\le i\le k$ there exists numbers $m_i\ge 0$
	such that $f(x)\in A$ if and only if
	\begin{itemize}\item
		$f^{(j)}(\alpha_i)=(-1)^jf^{(j)}(\beta_i)$ for each  $0<i\le k$ and each $0\le j\le m_i,$
		\item  $f^{(j)}(\alpha_0)=0,\ j=1,3,\ldots, 2m_0-1$ (for odd $g$ only).
	\end{itemize}
	Vice versa, if an algebra satisfies  such conditions, then it is generated by  $$(x-\alpha_0)^2,\ (x-\alpha_0)^{2m_0+1}\prod_{i\ge 1}(x-\alpha_i)^{m_i+1}(x-\beta_i)^{m_i+1}.$$
	
\end{theorem}
\begin{proof}
	
	Since the codimension is finite and the subalgebra is proper we can after substitution suppose that $A$ is generated by $p(x)=a(x^2)x,\ q(x)=x^2,$ where $a(x)$ is a monic polynomial of degree $l>0.$
	Here we put $\alpha_0=0$ and  for each
	non-zero root $\mu_i $ of  $a(x)$ with $i=1,\ldots,k$ we can put $\alpha_i=\sqrt{\mu_i}$ and $\beta_i=-\alpha_i.$ We define $m_0$ to be  multiplicity of zero as a zero of $a(x)$ and put $g=2k$ if $m_0=0$ and $g=2k+1$ if $m_0>0.$
	Now $a(x)=x^{m_0}\prod(x-\mu_i)^{m_i+1}$ and $$p(x)=x^{2m_0+1}\prod(x^2-\mu_i)^{m_i+1}=$$
	$$x^{2m_0+1}\prod(x-\alpha_i)^{m_i+1}(x-\beta_i)^{m_i+1}.$$ As we already discussed above the spectrum has exactly $g$ elements.
	To check the conditions note that they are trivial for $x^2$ and that $$p^{(j)}(\alpha_i)=p^{(j)}(-\alpha_i)=0$$ if $j\le m_i$  for $i>0.$ If $m_0>0$ then all the derivatives until $2m_0+1$ are zero as well. Therefore
	$p(x)$ and $q(x)$ satisfies the conditions and it is sufficient to check that if $f(x)$ and $g(x)$ satisfy the conditions the same is true for $f(x)g(x).$
	We have $$(f(x)g(x))^{(j)}(\alpha_i)=\sum_{j_1+j_2=j}\binom j {i_1} f^{(j_1)}(\alpha_i)g^{(j_2)}(\alpha_i)=$$
	$$\sum_{j_1+j_2=j}\binom j {i_1}(-1)^{j_1}f^{(j_1)}(-\alpha_i)(-1)^{j_2}g^{(j_2)}(-\alpha_i)$$
	and get the desired property both for $i>0$ and $i=0$ (because if $j$ is odd one of $j_1,j_2$ is odd as well). So $A$ satisfies the conditions. Let us now turn to the opposite direction.
	Our proof shows
	that the conditions determine some subalgebra that contains $A$ and we need to prove that it equals $A.$ If not there should be some  polynomial $f(x)$ which does not belong $A.$ Using subduction by $p(x)$ and $q(x)$
	we can suppose that it has an odd degree less than the degree of $p(x)$ and has only odd powers, and thus $f(-x)=-f(x).$
	
	Note that for an odd function $f(x)$ we have $$f^{(j)}(\beta_i) =f^{(j)}(-\alpha_i)=-(-1)^jf^{(j)}(\alpha_i).$$
	We get the opposite sign than in our conditions so all terms must be zero. Thus $\alpha_i$ and $\beta_i$ have multiplicity at least $m_i+1.$ as zeroes of $f(x).$ Similarly the second condition gives us that the multiplicity of zero as a zero is at least $2m_0+1.$ But then $f(x)$ cannot have degree less than degree of $p(x).$
	
	It remains to understand how we get back to the general case by using variable substitution back. Obviously $\alpha_0$ is the only root of the derivative in $q(x)$ and the spectrum is simply shifted by $\alpha _0.$
\end{proof}

To understand the conditions in the above theorem better we reformulate them in a special small case.

\begin{theorem}\label{th:52} Let $A$ be a subalgebra generated by polynomials of degree $2$ and $5$. Then depending of the size $s$ of its spectrum it can be described as
	\begin{description}
		\item[s=1] $A=\{f(x)|f'(\alpha)=f'''(\alpha)=0\}.$
		\item[s=2] $A=\{f(x)|f(\alpha)=f(\beta);f'(\alpha)+f'(\beta)=0\}.$
		\item[s=3] $A=\{f(x)|f(\alpha)=f(\beta);f'(\gamma)=0\};\ \alpha+\beta=2\gamma.$
		\item[s=4] $A=\{f(x)|f(\alpha)=f(\beta);f(\gamma)=f(\delta)\};\ \alpha+\beta=\gamma+\delta. $
	\end{description}
	Here $\alpha,\beta,\gamma,\delta$ are different elements of the spectrum.
\end{theorem}

\section{Relations between subalgebras}
The spectrum gives a much more clear picture of the inclusion of one subalgebra inside another. Just knowing the related semigroup inclusion can give us important information.
For example,  the subalgebra $A=<x^4,x^3-mx>$ is  not included in any subalgebra described in Theorem \ref{th:52}. Indeed, it has an element of degree $3$, which does not belong to the semigroup generated by $2$ and $5.$

We can use reversing inclusion of the spectra.  If that holds, one need to check if the subalgebra conditions are compatible with the given generators.
We have a complete classification of subalgebras of codimension $2$ in Theorem \ref{sec:deg2} below so let us study the list and decide which of them contains $A.$

If $m=0$ then $Sp(A)=\{0\}$ and our only candidate is $s=1$ with $\alpha=0.$ Obviously $b=0$ must hold so the only subalgebra of codimension two containing $A$ is defined by the conditions $f'(0)=f''(0)=0.$ This is the monomial algebra $<x^3,x^4,x^5>$.

 If $m\neq 0$ then using variable substitution we can restrict ourselves to
 the case $m=1.$ The \cp is equal to $(x^2-1)(x^4+1)$ and we get six elements in the spectrum. Thus $A$ is defined without derivatives and we only need to divide the elements of the spectrum into three pairs. Due to the fact that $x^4$ has the value $1$ on $1,-1$ and the value $-1$ on the other elements of the spectrum $1$ and $-1$ must pair up. Taking the values of $x^3-x$ into account we find the other pairs.

 In fact this is our Example \ref{ex:spectere4} from the very beginning of the text and now we easily find three possible algebras containing $A$ by omitting one of three conditions at a time. (It is important to note that the remaining two are still subalgebra conditions.)

Another application  is finding the intersection of two subalgebras: we take the union of their spectra and the union of their conditions and we only need to check if there are any linear dependencies between them. For example
we can easily find the situations when the intersection of two subalgebras is a monomial subalgebra. Both should have zero spectrum and the conditions of the subalgebras should complete each other so that we obtain conditions of the form $f^{(j)}(0)=0.$

We can go in the opposite direction as well: if we have two subalgebras $A_1,A_2$ we can easily construct the subalgebra they generate together. We take the intersection of the spectra and try to see which conditions  remain.
Let us take an example from \cite{Rob}. Is $<x^3-x,x^4,x^5-1>=\MK[x]?.$

The subalgebra $<x^4,x^5>$ is monomial, so its spectrum is zero. But zero is not in the spectrum of the subalgebra $<x^4,x^3-x>,$  so the intersection of their spectra is empty and we get $\MK[x].$

\chapter{The main conjecture.}
\section{Derivations}
One important corollary of  Theorem \ref{th:main2} is that each $\alpha-$derivation $D$ of a subalgebra $A$ with spectrum $C=\{\alpha_1,\ldots,\alpha_s\}$ can be written as
$$D(f)=\sum_{i=0}^N\sum_{j=1}^sc_{ij}f^{(i)}(\alpha_j),$$ if $\alpha\in C$ and as $cf'(\alpha)$ if $\alpha \not \in C$.
This can be proven by induction using  Theorem \ref{th:codimGorin}.

Our main conjecture is slightly stronger.
\begin{Conj}\label{conj:mainD} If $\alpha$ belongs to the spectrum then
		each $\alpha-$derivation  $D$ can be written as
\begin{equation}\label{eq:deriv} D(f)=\sum_{i=1}^N\sum_{\alpha_j\sim\alpha}c_{ij}f^{(i)}(\alpha_j),\end{equation} thus using pure derivatives (of some order) in the elements of the cluster containing $\alpha$.

\end{Conj}
 Note that if $\alpha$ does not belong to the spectrum then
according to Theorem \ref{th:derivativenotinspectrum} each $\alpha-$derivation can be written as $D(f)=cf'(\alpha).$

An alternative formulation of this conjecture is the following.
\begin{Conj}\label{conj:main} Each subalgebra $A$ of finite codimension can be described using only derivations of type (\ref{eq:deriv}) and conditions of the type $f(\alpha_i)=f(\alpha_j).$

\end{Conj}

We have already proven this  in Theorem \ref{th:main} in the case when $A$ has only one cluster.

\section{Main plan}\label{sec:mainplan}
We will now introduce some notation and present a general plan for attacking the main conjecture. The aim of the rest of this article is to realise this plan for algebras of small codimension and in this way get a classification of them.

Let $A$ be an algebra in $\MK[x]$ of finite codimension. We define its \textbf{type} \index{type of subalgebra} \index{$T(A)$} $T(A)=(d_1,\ldots, d_s)$  as an ordered list of degrees $d_i$ of the elements of a minimal SAGBI basis. Note that the numbers $d_i$ are simply the generators
of the numerical semigroup  $S=\{\deg f(x)|f(x)\in A\}$ consisting of all degrees of polynomials from $A.$ Thus the type is uniquely determined and for a fixed small codimension we can easily enumerate all possible types. For example, there is only one possible type $(2,3)$ for codimension one and two types, namely $(2,5)$ and $(3,4,5)$ for codimension two. For codimension three the possible types are: $$(2,7),(3,4),(3,5,7), (4,5,6,7).$$

For each $\alpha\in \MK$ we consider an important number $k_\alpha^A$\index{$k_\alpha^A$} which is the dimension of the vector space of all possible  $\alpha-$derivations. Normally we write only $k_\alpha$  if it is clear which subalgebra $A$ is used.

 Note that if $M_\alpha=\{f(x)\in A|f(\alpha)=0\}$ \index{$M_\alpha$} is the corresponding maximal ideal in $A$ then $A=\MK\oplus M_\alpha$ and (more importantly) that $$k_\alpha=\dim M_\alpha/M_\alpha^2.$$
Indeed, for any $\alpha-$derivation $D$ we have that $D(M_\alpha^2)=0.$ So if we choose a SAGBI basis in $M_\alpha$ and choose those $g_i$ from it that form a basis  modulo $M_\alpha^2$ then $D$ will be uniquely determined by the values of $D(g_i).$ On the other hand we can choose those values arbitrarily and the values of $D$ on the remaining elements in SAGBI basis will be uniquely determined.

Another important subspace is $\mathcal{D}^A_\alpha$ \index{$\mathcal{D}^A_\alpha$} consisting of those   $\alpha-$derivations that can be written as a linear combination ordinary derivations $f'(\beta),f''(\beta),\ldots$ for all $\beta\sim\alpha.$ Naturally $\alpha$ is one of such $\beta$ and may be the only one.
Our main conjecture in fact will be easier to prove if we simultaneously prove
\begin{Conj}\label{conj:dim}
		$\dim \mathcal{D}_\alpha^A=k_\alpha$ for each $\alpha.$
\end{Conj}

Our plan is to prove both conjecture \ref{conj:main} and \ref{conj:dim} simultaneously using induction on the codimension and consists of the following steps:

The base for the induction is codimension zero, thus $A=\MK[x]$. In this case all is obvious: all $k_\alpha=1$ and $\mathcal{D}^A_\alpha$ is generated by $f(x)\rightarrow f'(\alpha)$ which is obviously
an $\alpha-$derivation.

For the induction step we use Theorem \ref{th:codimGorin} and find a subalgebra $B$ of codimension one less that contains $A.$ Then conjecture \ref{conj:main} for $A$ will follow immediately, because we get $A$ from $B$ using an extra linear condition of the form $f(\alpha)=f(\beta)$ or by demanding that some derivation is equal to zero. Conjecture \ref{conj:dim} gives that the derivation will be an $\alpha-$derivation from $\mathcal{D}^B_\alpha$, so in both cases we get the extra condition of the desired form. Note also that we automatically get subalgebra conditions, because the kernel of a derivation is a subalgebra. So the main difficulty in the induction step will be to prove conjecture \ref{conj:dim}.

 The main challenge here  will be when $\alpha$ belongs to the spectrum of $A$, because otherwise we can simply use Theorem \ref{th:derivativenotinspectrum}.

 One extra improvement is based on the \textbf{semi-commu\-ta\-ti\-vity} \index{semi-commutativity} of the linear conditions $L_i(f(x))=0$ which we put on $\MK[x]$ to obtain $A.$ By this we mean the following: if the last condition is of the form $f(\beta)=f(\gamma),$
but an earlier one is of the form $D(f)=0$ for some $\alpha-$derivation we can interchange them, and hence suppose that $D(f)=0$ is the last condition. This changes the intermediate subalgebra $B$ but it will still be of the correct codimension so we can still use induction on the codimension.
The reason is that the derivation will still be a derivation  though for a smaller subalgebra.  Note also that the opposite exchange might be impossible, that is why we call it semi-commutativity.

This allows us to restrict ourselves to only two cases: either all conditions in $A$ are of the form $f(\alpha_i)=f(\beta_i)$ or $A$ was obtained from $B$ using some $\beta-$derivation. In the first case, as we will see later,  we can describe the derivations
directly.

%

In the second case we consider first an easy subcase, when $\alpha$ belongs to the spectrum of $A,$ but does not belong to the spectrum of $B.$ Because $f(\alpha)=f(\lambda)$ or $f'(\alpha)=0$ is valid in $A$ but not in $B$ we can, using the semi-commutativity, WLOG suppose
that the last of those condition was added to define $A$ (because it is a subalgebra condition and increases the codimension by one), thus $\alpha=\beta.$  Thus we have simply added $\alpha$ to the spectrum of $B.$  We  need to prove that $k_\alpha=2$ and $\mathcal{D}_\alpha$ is generated by $f''(\alpha),f'''(\alpha).$ For all other $\mu$ in the spectrum we should prove that $k_\mu$ and $\mathcal{D}_\mu$ are unchanged.

The case when $\beta$  belong to the spectrum of $B$ is more complicated and often demands the creation of non-trivial elements in $\mathcal{D}_\alpha$. We will show this in Chapter \ref{chapter:class} when applying the described approach for small codimensions.

\section{Subalgebras $A(C).$}
We will now realise some steps in our plan and prove the conjecture \ref{conj:dim} in the case were there are no derivations involved in the construction of $A.$
We start from a simplest  case when $A$ has only one cluster $C=\{\alpha_1,\ldots,\alpha_m\}.$
\begin{theorem}\label{oneclusternoderiv}  Let $A=A(C)$ be a
subalgebra of codimension $m-1$ that is defined by the conditions
$$f(\alpha_1)=f(\alpha_2)=\cdots f(\alpha_m).$$

\begin{itemize}
  \item $A$ has type $(m,m+1,\ldots, 2m-1).$
  \item  The elements
$$p_i=(x-\alpha_1)^i(x-\alpha_2)\cdots(x-\alpha_m)$$ for $i=1,\ldots m$ form a minimal SAGBI basis.
  \item For each $\alpha=\alpha_i$ any $\alpha-$derivations can be written as $f\rightarrow\sum c_jf'(\alpha_j).$
\item For $\alpha\not\in C$ any $\alpha-$derivation can be written as $f\rightarrow cf'(\alpha).$
\item The subalgebra $A$  satisfies  Conjecture \ref{conj:dim}.
\end{itemize}

\end{theorem}

\begin{proof}
	Obviously the polynomials $p_i$ satisfy the conditions and we cannot have
	(non-constant) polynomials of degree less then $m$ in $A.$ (If $\alpha_1$
	is a root then all $\alpha_i$ are roots). Because the semigroup generated
	by the degrees $m,m+1,\ldots, 2m-1$ does not contain $1,\ldots ,m-1$, but
	all other positive integers, we get that $p_i$ must form a SAGBI basis of
	$A$ (inside $M_{\alpha_1}=M_{\alpha_i}).$

	We have $m$ linearly independent derivations $f\rightarrow f'(\alpha_i)$
	and since $k_{\alpha_i}\le m$, conjecture \ref{conj:dim} is valid. The
	linear dependency can be checked directly on the given SAGBI basis -- the
	corresponding determinant can easily be reduced to a Vandermonde
	determinant and is equal to the product of some powers of
	$\alpha_i-\alpha_j$ with $i\neq j.$

	The rest of the statements are trivial.
\end{proof}

\section[different clusters]{Algebras with different clusters defined without derivations}

Our next step is  to generalize  Theorem~\ref{oneclusternoderiv} and  prove conjecture \ref{conj:dim}  for  algebras with several clusters but no derivations in its definition.

We begin with an easy but important statement that is valid for any subalgebra.

\begin{lemma} \label{lm:alphabeta} Suppose that  $\alpha\not\sim\beta,$ that is $\alpha$ and $\beta$ belong to different clusters in the subalgebra $B.$ Suppose that $D_1$ is a non-zero  $\alpha-$derivation and $D_2$ is a  $\beta-$derivation. Then
 \begin{itemize}
                \item  $D_1\neq D_2$, thus if $D$ is both $\alpha-$derivation and $\beta-$derivation then $D=0.$
                \item Moreover, $D_1$ and $D_2$ cannot coincide on the subalgebra $A$ obtained by adding the condition $f(\alpha)=f(\beta)$ to $B.$
              \end{itemize}

\end{lemma}
\begin{proof} Suppose the opposite, $D_1=D_2=D.$ By the condition there exists a polynomial $f$  such that $f(\alpha)\neq f(\beta).$ Because $D$ is both a $\beta-$ and an $\alpha-$derivation we obtain:
$$D(f^2)=2f(\beta)D(f)=2f(\alpha)D(f)\Rightarrow D(f)=0.$$

Also, $D$ is non-zero, so there exists a polynomial $g$ such that $D(g)=1.$ Now we  have:
$$D(fg)=f(\beta)D(g)+D(f)g(\beta)=f(\beta).$$ But the same arguments gives $D(fg)=f(\alpha)$ and we get a contradiction.

Now we want to prove the second statement. Suppose the opposite. By the first statement we can find $f\in B$ such that $D_1(f)\neq D_2(f).$  Denote $a=f(\alpha), b=f(\beta).$ Note that $a\neq b$ otherwise $f\in A$ and $D_1(f)=D_2(f).$ Let $L(p)=p(\alpha)-p(\beta)$. Then
$h=f^kL(f)-fL(f^k)\in A$ and because $D_1$ and $D_2$ coincide on $A$ we get:
$$0=D_1(h)-D_2(h)=$$
$$(D_1-D_2)(f^kL(f))-(D_1-D_2)(fL(f^k))=$$
$$\left(ka^{k-1}D_1(f)-kb^{k-1}D_2(f)\right)(a-b)-$$
$$\left(D_1(f)-D_2(f)\right)(a^k-b^k).$$

Cancelling $(a-b)$ we get for each $k\ge 2$ the equation
$$(ka^{k-1}-a^{k-1}-a^{k-2}b-\cdots- b^{k-1})D_1(f)=$$
$$(kb^{k-1}-a^{k-1}-a^{k-2}b-\cdots- b^{k-1})D_2(f).$$
Considering this as a linear system for $D_i(f)$ we get that the corresponding $2\times2$ determinants  should equal zero and from each pair of equations we get a polynomial equation in $a,b$.
Already for the equations from $k=2,3,4,5$ we find, using the Maple package Groebner, that all the solutions of the system satisfy  $a=b$. This provides us with the contradiction that completes our proof.
\end{proof}

\begin{theorem}\label{th:noderivations}

	Consider the algebra $$A=A(C_1) \cap A(C_2) \ldots \cap A(C_t)$$ containing all polynomials that are constant on each of the clusters $C_i.$ Then
\begin{itemize}
  \item For each $\alpha\in C_i$ any $\alpha-$derivation can be written as $f\rightarrow\sum_{\alpha_j\in C_i} c_jf'(\alpha_j).$
\item For $\alpha \not \in Sp(A)$ any $\alpha-$derivation can be written as $f\rightarrow cf'(\alpha).$
\item The subalgebra $A$  satisfies  Conjecture \ref{conj:dim}.
\end{itemize}

\end{theorem}

\begin{proof}
	Let us first consider the case with two clusters only, I.e $A = A(C_1) \cap
	A(C_2)$. Let $C=C_1\bigcup C_2.$ If $D$ is an $\alpha-$derivation over $A$
	with $\alpha\in C_1$ then $D$ is also an $\alpha-$ derivation when
	restricted to the subalgebra $A' = A(C).$ By Theorem
	\ref{oneclusternoderiv} this derivation is a linear combination $\sum
	c_iD_i$ of the derivations $D_i:f\rightarrow f'(\alpha_i)$ with
	$\alpha_i\in C.$  Subtracting the $\alpha-$derivation $\sum_{\alpha_i\in
	C_1} c_iD_i$ from $D$ we get an $\alpha-$derivation $D'$ which coincides
	with the $\beta-$derivation $\sum_{\alpha_i\in C_2} c_iD_i$ with $\beta\in
	C_2$ on $A(C).$ By Lemma \ref{lm:alphabeta} this is possible only if $D'=0$
	which proves our statement.

In the general case we can now easily use the same argument and induction on the number of clusters. If $\beta\not\sim \alpha$ we add an extra condition $f(\alpha)=f(\beta)$ and get a subalgebra $A'$ with a smaller number of clusters. Thus the restriction $D$ to $A'$ can be written as a linear combinations of $f'(\alpha_i)$ with $\alpha_i\sim\alpha$ or $\alpha_i\sim\beta.$ Subtracting the $\alpha-$ part from $D$ we get a derivation on $A$ which is both an $\alpha-$ and a $\beta-$derivation on $A$ and again lemma \ref{lm:alphabeta} finishes the proof.

To prove the conjecture is now easy because we have an upper bound and simultaneously sufficiently many explicitly given derivations.
\end{proof}

\chapter{SAGBI Bases and Derivations}
Now it is time to understand how the SAGBI basis changes when we add an extra condition to a subalgebra.

\section{Constructing SAGBI bases}\label{sec:SAGBI}

One useful thing we want to mention is that the inductive approach which we
have used throughout the article also allows us to relatively easily create
SAGBI bases in $A.$ Namely, when we have a SAGBI basis $G$ for $B$ and get $A$
by adding the condition $L(f)=0$ we do the following to obtain a SAGBI basis of
$A$.  All elements of $G$ that satisfy the extra condition $L(f)=0$ will remain
in the SAGBI basis. There must, however, be at least one element that does not
satisfy the condition. Let us choose such a $g\in G$ of minimal degree $d$,
thus $L(g)\neq 0.$ Note that exactly this degree $d$ should disappear from the
numerical semigroup $S$ of degrees. Thus we know the new semigroup $S_A=S \setminus \{d\}$
and can easily find the type $(s_1,\ldots, s_m)$ of the subalgebra $A.$ For
each degree $s_i$ we find  a polynomial $h_i\in B$ and our new SAGBI basis
consists of $f_i=L(g)h_i-L(h_i)g,$. If we wish to make them monic we can just
divide each $f_i$ by its highest coefficient.  In order to further simplify
calculations we want the basis elements to be inside $M_\alpha,$ and there are
several ways to do this. The simplest one is to replace $f_i(x)$ by
$f_i(x)-f_i(\alpha),$ but a more efficient way is to choose $h_i$ and $g$ in
$M_\alpha$ from the start. Sometimes it may be clever to choose a linear
combination with the previous $f_j$ to get as high degree of the factor
$x-\alpha$ as possible.

 We summarize this as follows.
\begin{theorem}\label{th:SAGBIbasis} Let $G$ be a SAGBI basis for $B$ chosen inside $M_\alpha^B.$ Let $g=g_i$ be an element of minimal degree in this basis that does not belong to $A.$ Suppose WLOG that  $L(g)=1.$
\begin{itemize}
  \item The set consisting of polynomials $g_j,\ h_j=gg_j-L(gg_j)g$ with $g_j \in G$, $j\neq i$ and two polynomials $f_k=g^k-L(g^k)g$ for $k=2,3$ forms a  SAGBI basis for $A$ inside $M_\alpha^A.$ (Not necessary a minimal one.)
  \item If $A$ has type $(s_1,\ldots,s_m)$ then to construct a minimal SAGBI basis one should for each $s=s_j$ find a polynomial $p_s \in B$ of degree $s$  and take $p_s-L(p_s)g.$  If all $p_s$ are chosen inside $M_\alpha^B$ then the obtained SAGBI basis will be inside $M_\alpha^A.$
   \end{itemize}
 \end{theorem}
 \begin{proof} If $f(x)\in B$ then $L(f-L(f)g)=0,$ thus $f-L(f)g$ belongs to $A=\ker L.$ This immediately proves the second statement because we get elements of degree $s_i$ in $A$. To prove the first statement
   we need to find polynomials built up from our basis elements of each degree $d\neq \deg g$ occurring in $B$. We can express $d$ as the degree of some   $g^lu$ where $u$ is a product of $g_j,$ where $j\neq i,$ but repetitions are allowed. Because each such $g_j$ belongs to  $M_\alpha^A$ the same is true for $u,$ so suppose that $l>0.$
  If $l\ge 2$ we can use $f_2^a f_3^bu$ where $l=2a+3b$ to get the degree $d.$ At last if $l=1$ and $u=g_jv$ for some  $g_j$ we can use $h_jv$.
 \end{proof}

 Now we want to prove a theorem that can help to estimate $k_\alpha$ at least in one special case.
 \begin{theorem}   \label{th:derivationlift}
Let $g$ be as in Theorem~\ref{th:SAGBIbasis}. If $L(g^2)^2\neq L(g)L(g^3)$ then $k_\alpha^A\le k_\alpha^B+1.$

If additionally $g$  does not belong to $\left(M_\alpha\right)^2$ then we get a stronger inequality: $k_\alpha^A\le k_\alpha^B.$

\end{theorem}
\begin{proof}Because $L(g)\neq 0$ we can replace $L$ by $L/L(g)$ and suppose WLOG that $L(g)=1.$  We have that $L(g_j)= 0$ for all $j<i$.  Subtracting  $cg$ we can WLOG  suppose that
$L(g_j)=0$ for all $j\neq i$ (because this subtraction does not change the degree). We can use the first part of the previous theorem to get a SAGBI basis inside $M_\alpha.$

 This means we have a generating set and need to study linear dependencies between its elements modulo $\left(M_\alpha^A\right)^2.$  Let $a= L(g^3)$ and $b= L(g^2),$ so that $a\neq b^2$ and $f_3=g^3-ag,f_2=g^2-bg.$
 First we get two elements in $\left(M_\alpha^A\right)^2:$
 $$f_3^2-f_2^3=(g^3-ag)^2-(g^2-bg)^3=$$
 $$3bg^5-(2a+3b^2)g^4+b^3g^3+a^2g^2.$$
 $$3bf_3f_2-2af_2^2=3b(g^3-ag)(g^2-bg)-2a(g^2-bg)^2=$$
 $$3bg^5-(3b^2+2a)g^4+abg^3+ab^2g^2.$$
 But their difference is equal to
 $$(b^3-ab)g^3+ (a^2-ab^2)g^2=(b^2-a)(bf_3-af_2)$$
 and we get a first linear dependency between $f_3$ and $f_2,$ because $a\neq b^2$.

 Let $c=L(gg_j)$ and $j\neq i.$ Another element in $\left(M_\alpha^A\right)^2$ is
 $$f_3g_j-f_2h_j-bf_2g_j=$$
 $$(g^3-ag)g_j-(g^2-bg)(gg_j-cg)-b(g^2-bg)g_j=$$
 $$cg^3+(b^2-a)gg_j-bcg^2=(b^2-a)h_j+ cf_3-bcf_2 $$ which shows that $h_j$ is a linear combination of $f_2,f_3$ modulo $\left(M_\alpha^A\right)^2.$

 If $g$ does not belong to $\left(M_\alpha\right)^2$
  we can complete it by some $g_j$ with $j\in J$ to form a basis in $M_\alpha^B/\left(M_\alpha^B\right)^2.$
  Then $k_\alpha^B=|J|+1.$ Otherwise, if $g$  belongs to $\left(M_\alpha\right)^2$, we have $k_\alpha^B=|J|,$ where we suppose that $g_j,j\in J$  form a basis in $M_\alpha^B/\left(M_\alpha^B\right)^2.$

 Thus to finish the proof it is sufficient to show that all elements in $M_\alpha^A$ can be expressed  modulo $\left(M_\alpha^A\right)^2$ as a linear combination of $g_j$ with $j\in J$ and $f_3,f_2.$ We have already done this
 for $h_j$ and it remains to do it for any $g_j$ with $j\neq i.$

 First we prove that $$\left(M_\alpha^B\right)^2\subseteq U=\MK g+ \MK f_2+\MK f_3+ \left(M_\alpha^A\right)^2.$$ Consider $\prod g_{j}$ with at least two factors. We use induction on the number of factors equal to $g_i=g$ to prove that this product belongs to $U.$

 This is obvious if all $g_j$ are different from $i.$, so suppose that $g$ is a factor. It is obvious for $g^2=f_2+bg$ as well.
 For $g_ig=h_j+cg$   it follows from the fact that we can replace $h_j$ by linear combination of $f_2$ and $f_3$ modulo $\left(M_\alpha^A\right)^2$. So we can suppose that we have more than two factors and at least one of them is $g.$

 If there are three factors we reduce this  to two factors directly: $$g^3=f_3+ag;\quad g^2g_j=f_2g_j+bgg_i;\quad gg_jg_j'=h_jg_{j'}+cgg_{j'}.$$

 Now we can suppose that we have at least four factors and at least one factor is $g.$
 If there is another factor $g$ we can replace $g^2$ by $f_2+bg$ and for both terms we can use induction. What remains is the case when $g$ appears only once. Then we have some factor $g_j$ and our product is written as
 $gg_ju$ for some shorter product $u$ which does not contain $g.$ We can replace
 $gg_j$ by $h_j+cg.$   Now it is immediate that  $h_ju$ belongs to  $\left(M_\alpha^A\right)^2$ because $u$ does not contain any factor $g$ and we can use induction for $gu$.

 Now we are ready  to consider $g_k$ with $k\neq i.$  We can write it as $ag+\sum_{j\in J} a_jg_j +m$ with $m\in \left(M_\alpha^B\right)^2.$
 This means that $g_k$  can be written as
 $$ag+\sum_{j\in J} a_jg_j +a'g+b'f_2+c'f_3+m'$$
 with $m'\in \left(M_\alpha^A\right)^2.$ But $g_k\in A=\ker L$ so $a+a'=0$ and we have finished our proof.

\end{proof}

\begin{corollary} \label{cor:lift}If $L(g^2)^2\neq L(g)L(g^3)$, $g$  does not belong to $\left(M_\alpha\right)^2$ and  $\dim \mathcal{D}_\alpha^B=k_\alpha^B$ then $\dim \mathcal{D}_\alpha^A=k_\alpha^A.$
\end{corollary}
\begin{proof} We use the same notation as in the theorem. By the conditions  there exist  $k_\alpha^B=1+|J|$ linearly independent $\alpha-$derivations in $\mathcal{D}_\alpha^B$. If we apply them to the basis $g_j$ with $j\in J$ and $g$
the corresponding determinant is different from zero. (Because any $\alpha-$derivation annihilates $\left(M_\alpha^B\right)^2$.) The restrictions of those derivations to $B$ belong to $\mathcal{D}_\alpha^B.$

 If we apply them to the same $g_j$ with $j\in J$ and either $f_2$ (if $b\neq 0$) or $f_3$ (if $b=0,a\neq 0$) we obtain a determinant where  for each derivation $D$ in the last row is replaced by $D(f_2)$ or $D(f_3).$
Because $g\in M_\alpha$ for any $\alpha-$derivation $D$ we have $D(g^2)=D(g^3)=0$ we get the same determinant multiplied by $a\neq 0$ or by $b\neq 0.$ In each case $f_2$ or $f_3$ complete $g_j$ with $j\in J$ to a basis modulo $\left(M_\alpha^B\right)^2$ and
we get that all our chosen derivations are still linearly independent, which proves the statement. Note that we automatically get $k_\alpha=k_\beta.$
\end{proof}

Note that the conditions are essential because in general the difference $k_\alpha^A-k_\alpha^B$ can be arbitrarily large. For example if $\alpha$ does not belong to $C$ then for $B=A(C)$ we have $k_\alpha^B=1.$ If we add $\alpha$ to the spectrum
with the condition $f(\alpha)=f(\alpha_1)$ we get $A(C\bigcup\alpha)$ and $k_\alpha^A=m+1$, where $m$ is the size of $C$ (according to Theorem \ref{oneclusternoderiv}).

Here is another useful application of what we have learned about how the SAGBI bases of $B$ and $A$ are related.
\begin{theorem}\label{th:trivialD} If $D$ is a \textbf{trivial}\index{trivial derivation} $\alpha-$derivation of an algebra $B$ (thus $\alpha$ does not belong to its spectrum) and $A=\ker D,$ then all $\alpha-$derivations of algebra $A$ can be written as $$f(x)\rightarrow af'''(\alpha)+bf''(\alpha).$$
\end{theorem}
\begin{proof} By Theorem \ref{th:derivativenotinspectrum} we can suppose that $D(f)=f'(\alpha)$ and  $k_\alpha^B=1.$ Then $f\rightarrow f^{(k)}(\alpha)$ for $k=2,3$ are two derivations and it is sufficient  to prove that they are linearly independent and that $k_\alpha^A\le 2.$ The linear independence is obvious if we restrict those maps to $g^2,g^3$ only so let us concentrate on the inequality.

As usual we choose a SAGBI basis $\{g_j\}$ in $M_\alpha^B$ such that $g=g_i$, $D(g_i)=1$ and $D(g_j)=0$ for $j\neq i.$ Because $k_\alpha^B=1$ we have that
$$M_\alpha^B=\MK g+\left( M_\alpha^B\right)^2.$$
In particular for $j\neq i$ we have $$g_j\in \left( M_\alpha^B\right)^2.$$
Indeed, $$g_j=cg+m,\ m\in \left( M_\alpha^B\right)^2\Rightarrow$$
$$0=D(g_j)=c+D(m)=c+0\Rightarrow c=0.$$

Using the fact that $D(g^2)=D(g^3)=D(gg_j)=0$ we get by the first part of Theorem \ref{th:SAGBIbasis} that $M_\alpha^A$ is generated by $g^2,g^3,g_j, gg_j$ with $j\neq i.$
Note that all those elements belong to $\left( M_\alpha^B\right)^2.$ Since both $\left( M_\alpha^B\right)^2$ and $M_\alpha^A$ have codimension one in $M_\alpha^B$ we conclude that
$$M_\alpha^A=\left( M_\alpha^B\right)^2.$$
Next we want to study which of the products $p=\Pi g_j$ with at least two elements that belong to $\left( M_\alpha^A\right)^2.$ It depends on the number of factors $g_j$ that equal $g=g_i.$ If there is no factor $g$ then
$p\in\left( M_\alpha^A\right)^2.$ This also holds if at least four factors equal $g$, because $g^n$ with $n\ge 4$ can be written as a product of $g^2$ and $g^3.$

If $p=g^3u$ or $p=g^2u$ where $u$ does not contain $g$, the only exception is $p=g^3$ and $p=g^2,$ because otherwise $u\in M_\alpha^A.$

At last if $p=gu$ then  $u=g_jv$ and the only exception is $p=gg_j.$ In all other cases $p=(gg_j)v\in\left( M_\alpha^A\right)^2.$

Because the products  $\Pi g_j$ span $\left( M_\alpha^B\right)^2$ we conclude that

\begin{equation}\label{eq:inclusion1}\left( M_\alpha^B\right)^2\subseteq \MK g^2+\MK g^3 +\sum_{j\neq i} \MK gg_j+ \left( M_\alpha^A\right)^2.\end{equation}

We know that $g_k\in \left( M_\alpha^B\right)^2$ for $k\neq i$. As a result
$$gg_k\in  \MK g^3+\MK g^4 +\sum_{j\neq i} \MK g^2g_j+ g\left( M_\alpha^A\right)^2.$$
Using the facts that $g^4=(g^2)^2, (g^2)g_j\in \left( M_\alpha^A\right)^2$  and $g M_\alpha^A\subseteq M_\alpha^A$
we find that
$$gg_k\in \MK g^3+\left( M_\alpha^A\right)^2.$$

Applying this for $k=j$ in (\ref{eq:inclusion1}) we can improve this to
$$M_\alpha^A=\left( M_\alpha^B\right)^2\subseteq \MK g^2+\MK g^3 + \left( M_\alpha^A\right)^2.$$
From this it is clear that $$k_\alpha^A=\dim M_\alpha^A/\left( M_\alpha^A\right)^2\le 2.$$

\end{proof}

\section{$\beta-$derivations}

Let us go back to our plan to prove that $k_\alpha=\dim D_\alpha.$ Using Theorem \ref{th:noderivations}   and  semi-commutativity we can suppose that $A$ is obtained from $B$ by some $\beta-$derivation $D$.
Let us first concentrate on the important case when $\alpha$ is not equivalent to $\beta.$  Using the same notation as in section \ref{sec:SAGBI} with $L=D$ we can suppose  WLOG that $D(g)=1$. By Theorem \ref{th:SAGBIbasis} we have a SAGBI basis $g_j,f_2,f_3, h_j$ for $A.$
\begin{theorem} Suppose that $g(\beta)=0.$
 If $\dim \mathcal{D}_\alpha^B=k_\alpha^B$ then $\dim \mathcal{D}_\alpha^A=k_\alpha^A.$
\end{theorem}
\begin{proof}
We  obtain directly that $$D(g^2)=D(g^3)=0;$$
$$ D(gg_j)=g(\beta)D(g_j)+D(g)g_j(\beta)=g_j(\beta)=c_j, $$
 $$f_2=g^2,f_3=g^3, h_j=gg_j-c_jg.$$

 Because $\alpha$ is not equivalent to $\beta$ we should have $c_k=g_k(\beta)\neq 0$ for some $k\neq i.$ Taking $g_k$ to be of minimal degree with this property we can subtract $d_jg_k$ from $g_j$ for suitable constant $d_j$.  Therefore we can suppose WLOG that $c_j=g_j(\beta)=0$ for $j\neq k.$
Then we have for $j\neq i,k$ that $h_j=gg_j$ and $$h_jg_k-h_kg_j=gg_jg_k-\left(gg_k-c_kg\right)g_j=
c_kgg_j=c_kh_j,$$
which means that all $h_j$ with $j\neq i,k$ are equal to zero modulo $\left(M_\alpha^A\right)^2$.
Moreover:
$$f_3g_k-f_2h_k=g^3g_k-g^2\left(gg_k-c_kg\right)=c_kg^3,$$
$$h_k^2-f_2g_k^2+2c_kg_kf_2=\left(gg_k-c_kg\right)^2-g^2g_k^2+2c_kg_kg^2=c_k^2g^2,$$
thus $f_3=g^3$ and $f_2=g^2$ also equal zero modulo $\left(M_\alpha^A\right)^2$.

Next we want to prove that for any $u\in M_\alpha^B$ we have
\begin{equation}\label{eq:inclusionshort}gu\in u(\beta) g+ \frac{u(\beta)}{c_k} h_k+ \left(M_\alpha^A\right)^2.
 \end{equation}

 It is sufficient to prove this for $u=\Pi g_j$
 and we use induction on the number of elements in the product.
 If we have only one element in the product then we consider two cases. For $u=g_k$ we have $$gg_k=c_kg+h_k=g_k(\beta)g+\frac{g_k(\beta)}{c_k}h_k.$$ For $u=g_j$ with $j\neq k$ we have $u(\beta)=0$ and
using  that $g^2,gg_j=h_j$ are equal to zero modulo $\left(M_\alpha^A\right)^2$ we confirm (\ref {eq:inclusionshort}).

Now suppose that we have at least two factors in $u.$ If $g$ is one of the factors then  $gu=g^lv$ where $l\ge 2$ and $v$ does not contain $g.$ Writing $l=2a+3b$ and using that $g^2,g^3\in  \left(M_\alpha^A\right)^2$ we get that
$gu\in\left(M_\alpha^A\right)^2.$ Because $u(\beta)=0$ we confirm (\ref {eq:inclusionshort}) in this case.

So, suppose that there are no factors $g$ in $u.$

If one factor is $g_j$ with $j\neq k$ then $u(\beta)=0$ and simultaneously $gg_j$ is equal to zero modulo $\left(M_\alpha^A\right)^2,$  and we get (\ref {eq:inclusionshort}) also in this case. It only remains to consider the case when $u=g_k^l$ with $l\ge 2.$ In this case we get
$$gu=gg_k^l=h_kg_k^{l-1}+c_kgg_k^{l-1}.$$ The first summand belongs to $\left(M_\alpha^A\right)^2,$ and for the second we can use induction to see that
$$gu \in c_kg_k(\beta)^{l-1} g+ c_k\frac{g_k(\beta)^{l-1}}{{c_k}}h_k+\left(M_\alpha^A\right)^2=$$
$$g_k(\beta)^{l}g+ \frac{g_k(\beta)^{l}}{{c_k}}h_k+\left(M_\alpha^A\right)^2,$$ which finishes the proof of (\ref {eq:inclusionshort}).

As result we get the inclusion
\begin{equation}\label{eq:inclusion}\left(M_\alpha^B\right)^2\subseteq U=\MK g+ \MK h_k+ \left(M_\alpha^A\right)^2.
 \end{equation}
 Indeed, if the product  $\Pi g_j$ contains at least two factors and none of them is $g$ it is obvious because all other $g_j$ belong to $M_\alpha^A.$ For the remaining products it follows from (\ref {eq:inclusionshort}).

 Now we are ready to estimate the dimensions.

Suppose  first that  $g$ does not belong to $\left(M_\alpha^B\right)^2$.
  We can complete it by some $g_j$s with $j\in J$ to form a basis in $M_\alpha^B/\left(M_\alpha^B\right)^2.$
  Then $k_\alpha^B=|J|+1.$  We want to prove that $g_j$ with $j\in J$ together with $h_k$ span $M_\alpha^A$ modulo $\left(M_\alpha^A\right)^2.$ We already know that $h_j$ and $f_2,f_3$ can be obtained and it remains to show that
  the same is true for $g_j$ with $j\neq i.$  We know $g_j$ can be written  as $ag+\sum_{j\in J} a_jg_j +m$ with $m\in \left(M_\alpha^B\right)^2.$
 Using (\ref{eq:inclusion}) we obtain that $g_j$  can be written as
 $$ag+\sum_{j\in J} a_jg_j +a'g+b'h_k+m'$$
 with $m'\in \left(M_\alpha^A\right)^2.$ But $g_j\in A=\ker L$ thus $a+a'=0$ and we have finished the proof which shows that $k_\alpha^A\le |J|+1=\dim k_\alpha^A.$

To obtain that this number equals $\dim \mathcal{D}_\alpha^A$ as well we can use the same argument as in the proof of corollary \ref{cor:lift}, but an even easier ways is to note that $h_k$ can replace $g$ in
the basis of $M_\alpha^B/\left(M_\alpha^B\right)^2.$ Thus the linear independence of the derivatives remains since it can be checked on the same set of polynomials.

Let us now consider the remaining case when $g$ belongs to $\left(M_\alpha^B\right)^2$. This means that
$g=gu+m$ with $u\in M_\alpha^B$ and $m\in\left(M_\alpha^A\right)^2,$ where $m$ is the sum of all terms which does not contain $g_i.$ Then $$1=D(g)=D(g)u(\beta)+g(\beta)D(u)+D(m)=$$
$$1\cdot u(\beta)+0\cdot D(u)+0=u(\beta).$$
Now inclusion (\ref {eq:inclusionshort}) results in
$$ g=u(\beta)g+\frac{u(\beta)}{c_k}h_k+m'=g+\frac{1}{c_k}h_k+m'$$ with $m'\in \left(M_\alpha^A\right)^2.$ We  get that $h_k\in \left(M_\alpha^A\right)^2.$
Thus $g_j$ with $j\neq i$ span $M_\alpha^A$  modulo $\left(M_\alpha^A\right)^2$. We chose $g_j$ with $j\in J$ that form a basis for $M_\alpha^B/\left(M_\alpha^B\right)^2$ and note that $J$ does not contain $i$ because $g=g_i$ is zero modulo $\left(M_\alpha^B\right)^2$.
 We want to show that
each element of $M_\alpha^A$ can be obtained modulo $\left(M_\alpha^A\right)^2$ using these basis elements. It is sufficient to check this for $g_k$ with $k\neq i.$
We simply repeat the above argument.

 Using (\ref{eq:inclusion}) we find that $g_k$  can be written as
 $$ag+\sum_{j\in J} a_jg_j +a'g+b'h_k+m'$$
 with $m'\in \left(M_\alpha^A\right)^2.$ But $g_k\in A=\ker L$ thus $a+a'=0.$ In addition we know that $h_k\in \left(M_\alpha^A\right)^2$ and this finishes the proof which shows that $k_\alpha^A\le |J|=\dim k_\alpha^B.$

To show that this number equals $\dim \mathcal{D}_\alpha^A$ as well we simply note that we already have that many derivations in $\mathcal{D}_\alpha^B$
and the linear independence of the derivations remains as it can still be checked on the same set of polynomials.

  \end{proof}

\begin{corollary} \label{colD2D3} Suppose that $D(g^2)^2=D(g)D(g^3).$
 If $\dim \mathcal{D}_\alpha^B=k_\alpha^B$ then $\dim \mathcal{D}_\alpha^A=k_\alpha^A.$
\end{corollary}
\begin{proof}WLOG $D(g)=1.$ It now follows from the assumption that $$(2g(\beta))^2=3g(\beta)^2\Rightarrow g(\beta)=0.$$
\end{proof}

The above Corollary holds also in the case when $L(f)=f(\beta)-f(\gamma)$ instead of a derivation. The proof is similar and we leave it out.
%

In Theorem \ref{oneclusternoderiv} we gave a description of the subalgebras with just one cluster and no derivations among its subalgebra conditions. We are now ready to move one step up in complexity.

\begin{theorem}\label{oneclusternoderivandderiv}  Let $A$ be a
subalgebra of codimension $m$ that is defined by the conditions
$$f(\alpha_1)=f(\alpha_2)=\cdots f(\alpha_m),\quad f'(\beta)=0.$$ with $\beta\neq\alpha_i.$
Let $p(x)=(x-\alpha_1)\cdots (x-\alpha_m).$
\begin{itemize}
  \item if $p'(\beta)\neq 0$ then  $T(A)=(m+1,m+2,\ldots, 2m+1)$ and a SAGBI basis is given by
  $$p(x)[(x-\beta)p'(\beta)-p(\beta)];$$
  $$p(x)(x-\beta)^k,\quad k=2,3,\ldots,m+1.$$

  \item
   if $p'(\beta)=0$ then  $$T(A)=(m,m+2,m+3,\ldots,2m-1, 2m+1)$$ (all integers  in the interval $[m,2m+1]$ except $m+1$ and $2m$) and a SAGBI basis is given by
  $$p(x),\quad
  p(x)(x-\beta)^k,\quad k=2,3,\ldots,m-1, m+1.$$
  \item For each $\alpha=\alpha_i$ any $\alpha-$derivation can be written as $f\rightarrow\sum c_jf'(\alpha_j).$
  \item Any $\beta-$derivation has the form\\ $D(f)=af''(\beta)+bf'''(\beta).$
\item For $\alpha\neq\alpha_i, \alpha\neq \beta$ any $\alpha-$derivation can be written as $f\rightarrow cf'(\alpha).$
\item The subalgebra $A$  satisfies  Conjecture \ref{conj:dim}.
\end{itemize}

\end{theorem}
\begin{proof} The algebra $A$ is obtained from $B=A(C)$ by adding the condition $L(f)=f'(\beta)=0$ so it is natural to use Theorem \ref{oneclusternoderiv} as a starting point. But we choose our SAGBI basis in $M_\alpha^B$ as
$$g_k=(x-\beta)^kp(x),\quad k=0,\ldots m-1.$$ Note first that for $k\ge 2$ we get $g_k\in \ker L=A.$ Moreover
$$L(g_1)=(p(x)(x-\beta))'|_{x=\beta}=p(\beta)\neq 0.$$
So, if we need to choose $g$ to apply corollary \ref{cor:lift} or \ref{colD2D3} we get either $g=g_0=p$ (if $p'(\beta)\neq 0$) or $g=g_1$ (if $p'(\beta)=0$). This explains why there are alternative SAGBI bases. We only need to verify
that the conditions for one of the corollaries are satisfied. Indeed, $g\not\in M_\alpha^2$ because it is not divisible by $(x-\alpha)^2.$ So, if
$L(g^2)^2=L(g)L(g^3)$ we can apply corollary \ref{colD2D3}, otherwise corollary \ref{cor:lift}.

The last step that $k_\beta^A=2$ follows from Theorem \ref{th:trivialD}.
\end{proof}


We finish this chapter by formulating a more general conjecture regarding what derivations we have in the intersection of two subalgebras.
\begin{Conj}\label{Conj:intersection} Let $A_1,A_2$ be two subalgebras of finite codimension such that their spectra have no common elements.
Then the set of derivations of $A=A_1\bigcap A_2$ is the union of the derivations in $A_1$ and $A_2$ (restricted to $A$).
\end{Conj}

\chapter{Classifications}\label{chapter:class}

Let us see how the method described in section \ref{sec:mainplan} can be realised starting with subalgebras of codimension one and moving step by step to higher codimensions.
\section{Subalgebras of codimension one}
For codimension one we start from $\MK[x]$ (which has $x$ as SAGBI basis and from which we can get $A$ either by the condition $f'(\alpha)=0$ or by the condition $f(\alpha)=f(\beta)).$  We now get Theorem \ref{th:codm1} without any effort thanks to Theorem \ref{th:codimGorin}.
Now we want to prepare for the next codimension and for this we need to find SAGBI bases and derivations for the different subalgebras of codimension one. We obviously have that
$ \mathcal{D}_\alpha$ contains $f''(\alpha), f'''(\alpha)$ in the first case and  $f'(\alpha),f'(\beta)$ in the second case.

 Because $k_\alpha$ and $k_\beta$ are not greater than the number of generators, which equals two, Conjecture \ref{conj:dim} is obviously valid and we have found all nontrivial derivations.

Type $(2,3)$ is the only possible semigroup of degrees, so an easy way to construct a SAGBI basis   is to use the second part of Theorem \ref{th:SAGBIbasis}. An even more convenient way in this case is to use Theorem \ref{th:deg2} to get the basis directly. We will however, in order to practice using our algorithm to get a SAGBI basis inside $M_\alpha$, use the first part of Theorem \ref{th:SAGBIbasis} instead.

 First we choose $g=x-\alpha$ as a single-element SAGBI basis of $\MK[x]$ inside  $M_\alpha$. We let $L:f\rightarrow f'(\alpha)$  in the first alternative and $$L: f\rightarrow \frac{f(\alpha)-f(\beta)}{\alpha-\beta}$$
 in the second alternative to get $L(g)=1$ (while we still have $A=\ker L$).

 Now, according Theorem \ref{th:SAGBIbasis} the elements $g^k-L(g^k)g$ for $k=2,3$ form the desired SAGBI bases. We get $L(g^k)=0$ and
 $$(x-\alpha)^2, (x-\alpha)^3$$ as the SAGBI basis for the first alternative, that is when $L$ is a derivation. For the second alternative we get
 $$f_k=(x-\alpha)^k-\frac{0-(\beta-\alpha)^k}{\alpha-\beta}(x-\alpha).$$
 For $k=2$ this results in
 $$f_2=(x-\alpha)\left(x-\alpha - (\alpha-\beta)\right)=(x-\alpha)(x-\beta).$$
  and for $k=3$ we get
 $$f_3=(x-\alpha)\left((x-\alpha)^2 - (\alpha-\beta)^2\right)=(x-\alpha)(x-\beta)(x-2\alpha+\beta).$$
  Adding $(\alpha-\beta)f_2$ to $f_3$ we get an even nicer SAGBI basis:
 $$(x-\alpha)(x-\beta); (x-\alpha)^2(x-\beta).$$

 The last thing that we formally need to do to finish the proof of Conjecture \ref{conj:dim} for codimension one is to check that the found derivations in $\mathcal{D}_\alpha$ are linearly independent. Even if that is quite obvious here we want to show how to do it. We simply find the values of the derivations on equally many  elements of the SAGBI basis and calculate the determinant. In the first case, for our derivations $f''(\alpha),f'''(\alpha)$, we get
 $$\left|
     \begin{array}{cc}
       2 & 0 \\
       0 & 6 \\
     \end{array}
   \right|\neq 0.$$ In the second case, for $f'(\alpha),f'(\beta)$, we get
 $$\left|
     \begin{array}{cc}
       \alpha-\beta& \beta-\alpha \\
       0 & (\beta-\alpha)^2 \\

     \end{array}
   \right|=(\alpha-\beta)^3\neq 0.$$

\section{Subalgebras of codimension two} \label{sec:deg2}
We now turn to subalgebras of codimension two. By Theorem \ref{th:codimGorin} they can be obtained by applying one extra condition to a subalgebra $B$ of codimension one. This means we need to study how those conditions look. In the case when the extra condition is $f(\alpha)=f(\beta)$ we simply add one or two elements to the spectrum and obtain the algebra $A.$
This is an easy case. A more difficult case is when we need to describe a kernel of some derivation. But since we already have proven Conjecture \ref{conj:dim} for codimension one we know the derivations in each of the two cases considered above. Thus we are prepared to make a classification of all codimension two subalgebras:

\begin{theorem} \label{th:codim2} Let $A$ be a subalgebra of codimension two. Then it is either type $(2,5)$  or  type $(3,4,5)$.
The spectrum contains  $s\le 4$ elements and depending on $s$ we have the following possibilities:
\begin{description}
  \item[s=1] $A=\{f(x)|f'(\alpha)=0; af''(\alpha)+bf'''(\alpha)=0\}.$ \\If $a=0,b\neq 0$ then $T(A)=(2,5)$ and  if $a\neq 0$ then $T(A)=(3,4,5).$
  \item[s=2] $A=\{f(x)|f(\alpha)=f(\beta);af'(\alpha)+bf'(\beta)=0\}.$\\If $a=b\neq 0$ then $T(A)=(2,5)$ and if $a\neq b$ then $T(A)=(3,4,5).$
  \item[s=2] $A=\{f(x)|f'(\alpha)=f'(\beta)=0\}.$\\ In this case $T(A)$ is always $(3,4,5).$
   \item[s=3] $A=\{f(x)|f(\alpha)=f(\beta);f'(\gamma)=0\};$\\ If $\alpha+\beta=2\gamma$ then $T(A)=(2,5),$ and if $\alpha+\beta\neq 2\gamma$ then $T(A)=(3,4,5).$
   \item[s=3] $A=\{f(x)|f(\alpha)=f(\beta)=f(\gamma)\}.$ \\ In this case $T(A)$ is always $(3,4,5).$
    \item[s=4] $A=\{f(x)|f(\alpha)=f(\beta);f(\gamma)=f(\delta)\}.$\\
    If $\alpha+\beta=\gamma+\delta $ then $T(A)=(2,5)$ and if $\alpha+\beta\neq\gamma+\delta$ then $T(A)=(3,4,5)$.
\end{description}
Here $\alpha,\beta,\gamma,\delta$ are different elements of the spectrum.

\end{theorem}
\begin{proof} We know that the spectrum has at most four elements. We start with the case where there are no derivations in the subalgebra conditions. Either we have two clusters and get the only case with $s=4$ or we have only one cluster of size $3$ and get the second case with $s=3.$

If some $\gamma-$derivation is used then by semi-commutativity we can suppose that it was added to a subalgebra of codimension one.
If $\gamma$ was not in the spectrum of this codimension one subalgebra, then $\gamma$ is a trivial derivation $f\rightarrow f'(\gamma)$ and we get either the second case with $s=2$ (with $\gamma=\beta$) or the first case with $s=3.$

At last if $\gamma$ belongs to the spectrum we can WLOG suppose that $\gamma=\alpha$ and use that we know all $\alpha-$derivations. We get cases with $s=1,2.$

It is easy to check that $(2,5)$ and $(3,4,5)$ are the only choices for the numerical semigroup of degrees. To see which choice is valid we only need to check if the element of degree $2$  in the SAGBI basis satisfies the added condition. If so we get type $(2,5)$, otherwise type $(3,4,5)$. Alternatively we can use Theorem \ref{th:52} which tells us exactly when $T(A)=(2,5)$.

\end{proof}

\section[Subalgebras of codimension three]{General plan for classifying Subalgebras of codimension three} \label{sec:plancodim3}
Our plan now is to prove the main conjecture  for subalgebras $A$ of codimension three and get  descriptions of them including SAGBI bases. It is not hard to verify that there are exactly four numerical semigroups of genus three: $$(2,7),\quad(3,4),\quad(3,5,7),\quad(4,5,6,7).$$ Thus the type of a subalgebra of codimension three must be one of these four listed types.

We use our  natural way to get $A$ from a subalgebra $B$ of codimension two using Theorem \ref{th:codimGorin}. So let $A$ be obtained from $B$ by a condition $L=0.$

 We choose a SAGBI basis $\{g_i\}$ in $B$ and choose $g$ from this basis of minimal degree among the basis elements with $L(g)\neq 0.$ Subtracting suitable multiples $d_jg$ from each $g_j$ we can suppose that $L(g_j)=0$ for all other elements in the basis. Thus all of them belong to $A$ and can be chosen as a part of SAGBI basis of $A.$

  To detect the type and complete them to a SAGBI basis in $A$ we have the following alternatives.

 If $T(B)=(2,5)$ and $\deg g =2$ then $T(A)=(4,5,6,7)$ because $A$ has no elements of degree $3.$ If $g_2\in \ker L$ is the remaining element in the SAGBI basis then the SAGBI basis of $A$ consists of
 $$L(g)g^2-L(g^2)g,\, g_2,\, L(g)g^3-L(g^3)g,\, L(g)gg_2-L(gg_2)g.$$

 If $T(B)=(2,5)$ and $\deg g =5$ then $T(A)=(2,7)$ and we can use Theorem  \ref{th:deg2}.

 Now suppose $T(B)=(3,4,5)$ and $\deg g_i=i+2$ for $ i=1,2,3.$

 If $g_1,g_2\in \ker L$ then $T(A)=(3,4)$ and they form a SAGBI basis.

 If $g_1,g_3\in \ker L,$ but $g=g_2\not\in \ker L$ then  $T(A)=(3,5,7)$ and
 $$g_1,\, g_3,\, L(g)gg_3-L(gg_3)g$$ form a SAGBI basis.

 At last, if $g=g_1\not\in \ker L,$ but $g_2,g_3\in \ker L$ then $T(A)=(4,5,6,7)$ and

 $$g_2,\, g_3,\, L(g)g^2-L(g^2)g,\, L(g)gg_2-L(gg_2)g$$
 form a SAGBI basis.

 To realise this plan we first need to prove Conjecture \ref{conj:dim} for codimension two and find the corresponding SAGBI bases. After that we will be able to give a more detailed classification. In many case we can find elements in a SAGBI bases for $A$ explicitly.

Each size of the spectrum,  $s=|Sp(A)|$, is considered separately, but a common approach is the following. Either $A$ is obtained without derivations (this is possible when $s\ge 4$ only) and we can use Theorem \ref{th:noderivations} or we can use semi-commutativity and suppose that
$L$ is some $\alpha-$derivation. If so we have that $|Sp(A)-Sp(B)|\le 1.$ If we have equality then $\alpha\not\in Sp(B)$ and $L$ is a trivial derivation.

\section[type $(2,2k+1)$]{Derivations of subalgebras of type $(2,2k+1)$}\label{sec:22k+1}

In this section we will to study $B$ with $T(B)=(2,5).$ But in fact we can get a more general, for type $(2,2k+1)$ with $k\ge 1$, without much extra work, since we already have a full description of such algebras (including their SAGBI basis) in Theorem  \ref{th:deg2}. We will use the same notations as in this theorem.

 As we already mentioned it is sufficient to prove Conjecture \ref{conj:dim} for the elements in the spectrum, so start from $\alpha_i$ with $i \leq 1.$ We only have two elements in SAGBI basis, and as a consequence we have $k_{\alpha_i}\le 2$. Thus it is sufficient to find two derivations in $\mathcal{D}_{\alpha_i}.$ One of them is obviously $f(x)\rightarrow f'(\alpha_i)$. Another one
is $$D:f(x)\rightarrow f^{(m_i+1)}(\alpha_i)-(-1)^{m_i+1}f^{(m_i+1)}(\beta_i).$$ Indeed,

$$(fg)^{(m_i+1)}(\alpha_i)=$$
$$f^{(m_i+1)}(\alpha_i)g(\alpha_i)+f(\alpha_i)g^{(m_i+1)}(\alpha_i)+$$
$$\sum_{j=1}^m\binom{m_i+1} j f^{(m_i+1-j)}(\alpha_i)g^{(j)}(\alpha_i)=$$
$$f^{(m_i+1)}(\alpha_i)g(\alpha_i)+f(\alpha_i)g^{(m_i+1)}(\alpha_i)+$$
$$\sum_{j=1}^m(-1)^{m_i+1-j+j}\binom{m_i+1} j f^{(m_i+1-j)}(\beta_i)g^{(j)}(\beta_i).$$

On the other hand
$$(-1)^{m_i+1}(fg)^{(m_i+1)}(\beta_i)=$$
$$(-1)^{m_i+1}f^{(m_i+1)}(\beta_i)g(\alpha_i)+f(\alpha_i)(-1)^{m_i+1}g^{(m_i+1)}(\beta_i)+$$
$$\sum_{j=1}^m(-1)^{m_i+1}\binom{m_i+1} j f^{(m_i+1-j)}(\beta_i)g^{(j)}(\beta_i)$$
and we see that $D(fg)=D(f)g(\alpha_i)+f(\alpha_i)D(g).$

To check that those two derivations are linearly independent is sufficient to check their values on the SAGBI basis. We skip the details and restrict ourselves to the observation that
$(\alpha_i-\alpha_0)^{2m_0+1}(\alpha_i-\beta_i)^{m_i+1}=-(-1)^{m_i+1}(\beta_i-\alpha_0)^{2m_0+1}(\beta_i-\alpha_i)^{m_i+1}$ and $\alpha_i-\beta_j=\alpha_j-\beta_i.$

The situation with $\beta_i$ is similar and it only remains to look at $\alpha_0$ (when $m_0>1.$) Again one derivation $f(x)\rightarrow f''(\alpha_0)$ is trivial. We here get the second derivation as
$D:f(x)\rightarrow f^{(2m_0+1)}(\alpha_0)$. To verify that it is a derivation we use that $2m_0+1$ is odd and all smaller odd derivatives are already zero. Linear independence between those two derivations is easy to check.

\section[Type $(3,4,5)$]{SAGBI bases and derivations of subalgebras of type $(3,4,5)$}\label{sec:345}
It now remains to prove  Conjecture \ref{conj:dim} for algebras of  type $(3,4,5).$ Again the only interesting case is when $\alpha$ belongs to the spectrum. When we have only one cluster the conjecture follows from Theorem \ref{th:main2}.
The case $s=4$ is covered by Theorem \ref{th:noderivations}. Only two cases remain.

The case where $s=2$ and $A$ is defined by the condition
$$f'(\alpha) = f'(\beta)=0$$
is attacked in a straightforward way - the condition means that $f \rightarrow f'(\alpha)$ does not work as it equals zero, but higher derivatives work well. Let $D_1 : f \rightarrow f''(\alpha)$, $D_2 : f \rightarrow f^{(3)}(\alpha)$. We first confirm that these are in fact $\alpha$-derivations:
\begin{align*}
D_1(fg) &= f''(\alpha)g(\alpha) + 2f'(\alpha)g'(\alpha) + f(\alpha)g''(\alpha) \\
&= f''(\alpha)g(\alpha) + f(\alpha)g''(\alpha) \\
&= D_1(f)g(\alpha) + f(\alpha)D_1(g)
\end{align*}
and similarly for $D_2$ since any terms containing $f'(\alpha)$ vanish due to the condition on $A.$ Next, we show that $k_\alpha \leq 2$. To do this, pick a SAGBI basis in $M_\alpha$:
\begin{align*}
q =& (x-\alpha)^2(2x+\alpha-3\beta)\\
p =& (x-\alpha)^2(x-\beta)^2\\
r =& (x-\alpha)^3(x-\beta)^2.
\end{align*}
It is obvious that $q'(\alpha) = p'(\alpha) = r'(\alpha) = 0$, and easy to verify that $q'(\beta) = p'(\beta) = r'(\beta) = 0$. Then, by subduction, we find the relation
\begin{align*}
4\left(4p^2 - 2rq - (\alpha-\beta)pq\right) - (\alpha-\beta)^2q^2 \\
= 4(\beta-\alpha)^3r + 3(\beta-\alpha)^4p + (\beta-\alpha)^5q
\end{align*}

and applying any $\alpha$-derivation $D$ to this relation gives
$$
4(\beta-\alpha)^3Dr + 3(\beta-\alpha)^4Dp + (\beta-\alpha)^5Dq = 0.
$$
Since $\alpha \neq \beta$ this means $Dq$, $Dp$, $Dq$ are linearly dependent and hence $k_\alpha \leq 2$. Thus $D_1$, $D_2$ are sufficiently many $\alpha$-derivations and we are done because they are obviously linearly independent.
(Check from their values on the SAGBI basis!)

 It remains to consider the subalgebras
$$ \{f(x) | f(\alpha)=f(\beta), f'(\gamma)=0 \}. $$ with $\alpha+\beta\neq 2\gamma.$ But here we can apply Theorem \ref{oneclusternoderivandderiv}.

From the theorem we also get a SAGBI basis for the algebra, but it is also quite easy to find the basis directly in the following way: As elements of degree four and five we can choose
$$(x-\alpha)(x-\beta)(x-\gamma)^i,\. i=2,3.$$
For degree three we take
$$(x-\alpha)(x-\beta)\left(x-\frac{3\gamma^2-2(\alpha+\beta)\gamma+\alpha\beta}{2\gamma-\beta-\alpha}\right).$$ The only thing we need to compute in order to check that this is in fact a SAGBI basis is the derivative in $\gamma$ of the basis element of degree three.
\
In almost all cases we can easily describe the possible derivations, but we need to do  more work when higher derivatives are involved.

Note first that we can easily describe the derivations for any monomial algebra: if $A=\langle (x-\alpha)^s,s\in S\rangle$ where $S$ is a semigroup and $\{s_1, s_2, \ldots s_k\}$ a minimal generating set of $S$. Then all the maps $D_{s_i}:f(x)\rightarrow f^{(s_i)}(\alpha)$ are derivations. A SAGBI basis for $A$ is given by $$\{(x-\alpha)^{s_1}, (x-\alpha)^{s_2}, \ldots, (x-\alpha)^{s_k}\}$$ and by applying the derivations to the elements of the SAGBI basis we find that they are independent.

Next we need to study the case $s=1$ more carefully, so
 let $A$ be an algebra with a single element $\alpha$ in the spectrum. WLOG we can suppose that $\alpha=0.$ We know according to Theorem \ref{th:codim2} that $A$ is defined by the conditions $f'(0)=af''(0)+bf'''(0)=0, a\neq 0.$ If $b=0$ we get a monomial case, otherwise we can suppose  (in order to get a nice SAGBI basis) that $b\neq 0, a=3,$
 thus $$bf'''(0)+3f''(0)=0.$$
Then we can choose $p=x^4,q=x^3-bx^2, r=x^5.$ as generators of $A$. Note that
$$p^2-rq-bpq=bx^7-bpq=b^2x^6=b^2(q^2+2br-b^2p).$$ As $p(0)=q(0)=r(0)=0$ and $b\neq 0$ we get $2Dr=bDp.$
Thus we can take $k_0=2$ which means that we only need to find two derivations of the desired form.

One is obviously the second derivative, $D_1: f(x)\rightarrow f''(0),$
but we cannot use the third derivative because in our algebra it is proportional to $D_1.$ So we need to try higher derivatives $D_2:f(x)\rightarrow cf^{(4)}(0)+df^{(5)}(0).$ Our condition
   $2D_2r=bD_2p$ is equivalent to $2\cdot d\cdot 5! =bc\cdot 4!$ so we can try $c=10,d=b$ and only have to check that this is a derivation in $A.$
   We have (skipping terms that obviously equal zero)
   $$10(fg)^{(4)}(0)+b(fg)^{(5)}(0)-$$
   $$\left(10f^{(4)}(0)+bf^{(5)}(0)\right)g(0)-
   f(0)\left(10g^{(4)}(0)+bg^{(5)}(0)\right)=$$
   $$10\binom 4 2 f''(0)g''(0)+b\binom 5 2 f'''(0)g''(0)+b\binom 5 3 f''(0)g'''(0)= $$
   $$10\cdot 6 f''(0)g''(0)-10\cdot 3 f''(0)g''(0)-10\cdot 3 f''(0)g''(0)=0$$ and we are done with this case.

   The next case is when $A$ is defined by the conditions
   $$f(\alpha)=f(\beta), af'(\alpha)+bf'(\beta)=0, a-b\neq 0.$$
   We choose a SAGBI basis of the form $$g_k=(x-\alpha)(x-\beta)(x^k-\gamma_k), k=1,2,3.$$ The conditions give us that
   $$a(\alpha-\beta)(\alpha^k-\gamma_k)+b(\beta-\alpha)(\beta^k-\gamma_k)=0\Leftrightarrow \gamma_k=\frac{a\alpha^k-b\beta^k}{a-b},$$

 We already have a derivation $f(x)\rightarrow a'f'(\alpha)+b'f'(\beta)$, when $(a',b')$ is not proportional to $(a,b).$ If $b=0$ we have two more, defined by $f''(\alpha)$ and $f'''(\alpha).$
Similarly, if $a=0$ we get that the derivations defined by $f''(\beta)$ and $f'''(\beta)$ are two new $\alpha-$derivations as well. So in those cases we have the three necessary derivations and WLOG we can from now on assume that
$a=1,b\neq 0,$ or in other words $$f'(\alpha)=-bf'(\beta).$$

First we try to create a new derivation $D$ of the form $f(x)\rightarrow cf''(\alpha)+df''(\beta).$
We have (using that $f(\alpha)=f(\beta))$ and $g(\alpha)=g(\beta)):$
$$D(f(x)g(x))-D(f(x))g(\alpha)-f(\alpha)D(g(x))=$$
$$c(f''(\alpha)g(\alpha)+2f'(\alpha)g'(\alpha)+f(\alpha)g''(\alpha))+$$
$$d(f''(\beta)g(\beta)+2f'(\beta)g'(\beta)+f(\beta)g''(\beta))-$$
$$(cf''(\alpha)+df''(\beta))g(\alpha)-f(\alpha)(cg''(\alpha)+dg''(\beta))=$$
$$2cf'(\alpha)g'(\alpha)+2df'(\beta)g'(\beta)=2(b^2c+d)f'(\beta)g'(\beta)$$
so we can choose $c=1,d=-b^2$ to get a new derivation.

It remains to show that $k_\alpha\le 2$ and for this we need to study the relations between our generators. Using Maple we find that (for $a=1$) we have:
\begin{equation}\label{eq:subdcase345}
	g_2^2-g_1g_3-c_1g_1g_2-c_2g_1^2=c_3g_3+c_4g_2+c_5g_1
\end{equation} where
$$c_1=\frac{\alpha-\beta b}{b-1};c_2=\frac{2b(\alpha^2+\beta^2)-(\alpha+\beta b)^2}{(b-1)^2};$$
$$c_3=\frac{b(b+1)(\alpha-\beta)^3}{(1-b)^3};$$
$$c_4=\frac{b(\alpha-\beta)^3(b^2(2\beta+\alpha)-(2\alpha+\beta))}{(b-1)^4};$$
$$c_5=\frac{b(\alpha-\beta)^3(\alpha^2+2\alpha\beta-b^2(2\alpha\beta+\beta^2))}{(b-1)^4}.$$

It follows from \ref{eq:subdcase345} that $$c_3Dg_3+c_4Dg_2+c_5Dg_1=0$$ for any $\alpha-$derivation  $D$ so if at least one of $c_3,c_4,c_5$ is non-zero we get $k_\alpha\le 2.$
Because we already supposed that $b\neq 0$ we see from $c_3$ that the only interesting case is $b=-1.$
But in this case $c_4\neq 0$ and we are finished with our last case.

\section[Codimension $3$, $s=1.$]{Subalgebras of codimension three with a single element in the spectrum.}
Now we can apply the information about the derivations obtained above to classify the subalgebras $A$ of codimension three. Their spectra contain $s\le 6$ elements and in this section we consider  the case $s=1.$

\begin{theorem}
If an algebra $A$ of codimension $3$ has a spectrum consisting of single element $\alpha$ then $A$ is one of the following algebras
\begin{enumerate}
\item $A=\{f(x)|f'(\alpha)=f''(\alpha)=af'''(\alpha)+bf^{(4)}(\alpha)+cf^{(5)}(\alpha)=0\}.$

If $a\neq 0$ then $T(A)=(4,5,6,7)$ and for $a=1$ a
SAGBI basis is:
  $$(x-\alpha)^4-4b(x-\alpha)^{3},(x-\alpha)^{5}-20c(x-\alpha)^{3},(x-\alpha)^6,(x-\alpha)^7.$$

If $a=0$ and $b\neq 0$ then $T(A)=(3,5,7)$  and for  $b=1$ a
SAGBI basis is:
  $$ (x-\alpha)^3,(x-\alpha)^5-5c(x-\alpha)^4,(x-\alpha)^7.$$

For $a=b=0,c\neq 0$ the type is $(3,4)$ and a SAGBI basis is
  $$(x-\alpha)^3,(x-\alpha)^4.$$

If $a=b=c=d=0$ the codimension is $2$.

  \item $A=\{f(x)|f'(\alpha)=f'''(\alpha)+3af''(\alpha)=f^{(5)}(\alpha)+10af^{(4)}(\alpha)+df''(\alpha)=0\}.$

  with $a\neq 0.$ If $d\neq 0$ then $T(A)=(4,5,6,7)$ and a
SAGBI basis is:
$$d(x-\alpha)^4-120(x-\alpha)^3+120a(x-\alpha)^2,$$
$$ad(x-\alpha)^5-60(x-\alpha)^3+60a(x-\alpha)^2,(x-\alpha)^6,(x-\alpha)^7.$$

If $d=0$ then $T(A)=(3,5,7)$  and a
SAGBI basis is:
  $$ (x-\alpha)^3-a(x-\alpha)^2,2a(x-\alpha)^5-(x-\alpha)^4,(x-\alpha)^7.$$

\item $A=\{f(x)|f'(\alpha)=f'''(\alpha)=cf^{(5)}(\alpha)+df''(\alpha)=0\}.$
If $d\neq 0$ then $T(A)=(4,5,6,7)$ and a
SAGBI basis is:

  $$(x-\alpha)^4,d(x-\alpha)^{5}-60c(x-\alpha)^2,(x-\alpha)^6,(x-\alpha)^7.$$

If $c\neq 0,d=0$ then $T(A)=(2,7)$ and a SAGBI basis is:

  $$ (x-\alpha)^2,(x-\alpha)^7.$$

If $c= 0,d= 0$ we get codimension $2.$
\end{enumerate}

\end{theorem}
\begin{proof} The subalgebra $A$ is contained in a subalgebra $B$ of codimension $2$. Because the spectrum of $B$ is a subset of the spectrum of $A$ the subalgebra $B$ should have a single (and the same) element $\alpha$ in the spectrum. Moreover, $A$ is obtained from $B$ as a kernel of some $\alpha-$derivation (all other possibilities would lead to a larger spectrum). So the result will follow from the description of all derivations of
the subalgebra $B=\{f(x)|f'(\alpha)=0;a_1f'''(\alpha)+b_1f''(\alpha)=0\}$ by adding an extra derivation. If $a_1=0$ we put $b_1=1$ and get case $1.$

 If $b_1\neq 0$ and $a_1\neq 0$ we put $a_1=3$ and $b_1=b$ first. The derivation, as we have proved above, is a linear combination of $f(x)\rightarrow f''(\alpha)$ and $f(x)\rightarrow bf^{(5)}(\alpha)+10f^{(4)}(\alpha)$ and we get the  case $2$ if we simply substitute $b=\frac{1}{a}$ and multiply by $a$ where necessary.  If $a_1=0$ (which corresponds to  $T(B)=(2,5)$) we put $b_1=1$ and  get the case $3.$

 When we get a description there is a straightforward way described above to get a SAGBI bases: we know the possible degrees and need only to search for elements of the degrees generating the semigroup that satisfy the subalgebra conditions. If we need to we may use the methods from section \ref{sec:SAGBI} to create such polynomials from the elements of a SAGBI basis of $B$.
\end{proof}

\section[Codimension $3$, $s=2$]{Subalgebras of codimension three with two elements in the spectrum.}

\begin{theorem} If algebra $A$ of codimension $3$ has a spectrum consisting of two elements  then $A$ is one of the following subalgebras:
\begin{enumerate}
\item  $A=\{f(x)|f'(\alpha)=f'(\beta)=0; af''(\alpha)+bf'''(\alpha)=0\},$ \\
If  $a(\alpha-\beta)\neq 6b$ then $T(A)=(4,5,6,7).$ For $a=0$
a SAGBI basis is:
  $$2(x-\alpha)^k-k(\beta-\alpha)^{k-1}(x-\alpha)^2,k=4,5,6,7.$$

For $a\neq 0$ a SAGBI basis is (for $a=3$ which can be suppose WLOG)
   $$K(x-\alpha)^k-k(\beta-\alpha)^{k-2}[(x-\alpha)^3-3b(x-\alpha)^2],k=4,5,6,7,$$
   where $K=3(\beta-\alpha-2b).$

If  $a(\alpha-\beta)=6b$ then $T(A)=(3,5,7).$ If $a\neq 0$ then WLOG $a=3$ and
 a SAGBI basis is
 $$ (x-\alpha)^3-3b(x-\alpha)^2, 4(x-\alpha)^k-k(\beta-\alpha)^{k-4}(x-\alpha)^4,k=5,7.$$

If $a=b=0$ the codimension is $2$.

\item $A=\{f(x)|f(\alpha)=f(\beta);f'(\alpha)=0,\\
af''(\alpha)+bf'''(\alpha)+cf'(\beta)=0\}.$

If $b=c(\beta-\alpha), a=c(\beta-\alpha)^2$ and $c\neq 0$ then $T(A)=(3,4)$ and a
SAGBI basis is:
$$(x-\alpha)^2(x-\beta),(x-\alpha)^3(x-\beta).$$

If $K=6a+2(\alpha-\beta)b+(\alpha-\beta)^2c\neq 0$\\ then $T(A)=(4,5,6,7)$ and a
SAGBI basis is:
$$K(x-\alpha)^3(x-\beta)-[6(\alpha-\beta)a-(\alpha-\beta)^3c](x-\alpha)^2(x-\beta),$$
$$K(x-\alpha)^4(x-\beta)-(\alpha-\beta)^4(x-\alpha)^2(x-\beta),$$
  $$(x-\alpha)^4(x-\beta)^2,(x-\alpha)^5(x-\beta)^2.$$

If $a=b=c=0$ we get codimension $2.$

Otherwise  the type is (3,5,7) and a SAGBI basis is:

  $$ (x-\alpha)^2(x-\beta),(x-\alpha)^5(x-\beta)^2.$$
  $$(x-\alpha)^4(x-\beta)[6a-(\alpha-\beta)^2c]-(\alpha-\beta)^3(x-\alpha)^3(x-\beta)c.$$

\item $A=\{f(x)|f(\alpha)=f(\beta);f'(\alpha)+f'(\beta)=0,\\
af'(\alpha)+bf''(\alpha)-bf''(\beta)=0\}.$

If $a=0,b\neq 0$ then $T(A)=(2,7)$ and a SAGBI basis is:
$(x-\alpha)(x-\beta),  (x-\alpha)^4(x-\beta)^3.$

If $a\neq 0$ then $T(A)=(4,5,6,7)$ and a SAGBI basis is:
$$(x-\alpha)^2(x-\beta)^2,(x-\alpha)^3(x-\beta)^3,  (x-\alpha)^4(x-\beta)^3,$$
$$a(x-\alpha)^3(x-\beta)^2-2b(\alpha-\beta)^2(x-\alpha)(x-\beta).$$

\item $A=\{f(x)|f(\alpha)=f(\beta);f'(\alpha)+bf'(\beta)=0,\\
af'(\beta)+cf''(\alpha)-cb^2f''(\beta)=0\}.$ Here $b \neq 1$.

If $b=-1,12c=a(\beta-\alpha)$ then $T(A)=(3,4)$ and a SAGBI basis is:
$$(x-\alpha)(x-\beta)\left(x^k-\frac{\alpha^k+\beta^k}{2}\right), k=1,2 .$$

If $K=a(\beta-\alpha)-2c(b-1)(b^2-b-1)\neq 0$ then the type is (4,5,6,7) and a SAGBI basis is:
$$g_kD(g_3)-D(g_k)g_3,\quad k=4,5$$
 $$(x-\alpha)^3(x-\beta)^3;$$
$$(x-\alpha)^4(x-\beta)^3.$$
with  $$g_k=(x-\alpha)(x-\beta)(x^k-\gamma_k)$$ and
   $$ \gamma_k=\frac{\alpha^k-b\beta^k}{1-b}.$$
If $a=c=0$ then codimension is $2.$ In the remaining case ($K=0,b+1\neq 0$) the type is (3,5,7) and a SAGBI basis is:
$$(x-\alpha)(x-\beta)\left(x-\frac{\alpha-\beta b}{1-b}\right),$$
$$(x-\alpha)^4(x-\beta)^3,$$
$$g_5D(g_4)-g_4D(g_5).$$
\end{enumerate}

\end{theorem}
\begin{proof} The subalgebra $A$ is contained in a subalgebra $B$ of codimension $2$. To get two elements in the spectrum we need at least one derivation. Using the semi-commutativity
 we can suppose that it was the derivation $D$ that was used to obtain $A$ from $B.$

 Consider first the case when $D$ is a trivial derivation (outside the spectrum). Then $B$ has a single element $\alpha$ in the spectrum and is defined by the conditions
 $f'(\alpha)=af''(\alpha)+bf'''(\alpha)=0.$ We only need to add $f'(\beta)=0$ to get case 1.

 If  $a=0$ then $B$ contains an element of degree two and it does not belong to $\ker D$ so two disappears from the semigroup and hence $T(A)=(4,5,6,7).$ The SAGBI basis is constructed directly using Theorem \ref{th:SAGBIbasis}
 because $(x-\alpha)^k\in B$ for $k\ge 4.$

If $a\neq 0$ then we can WLOG suppose $a=3$ and we know that $B$ has SAGBI basis $$r=(x-\alpha)^5,p=(x-\alpha)^4,q= (x-\alpha)^3-3b(x-\alpha)^2.$$
If $$q'(\beta)=3(\beta-\alpha)^2-6b(\beta-\alpha)\neq 0\Leftrightarrow (\beta-\alpha)\neq 2b$$ then it is three that disappears from the semigroup and $T(A)=(4,5,6,7).$ The SAGBI basis is  constructed  using Theorem \ref{th:SAGBIbasis}
 because again $(x-\alpha)^k\in B$ for $k\ge 4.$ We can cancel $\beta-\alpha$ to make it shorter. This case can be joined with the previous one using the common condition  $a(\alpha-\beta)\neq 6b$ which is obvious for $a=0$
 and works for $a=3$. The general case is reduced to this after division by $\frac{a}{3}.$

If $(\beta-\alpha)=2b$ then $b\neq 0$ and $p'(\beta)=4(\beta-\alpha)^3\neq 0$ so it is four that disappears from the semigroup and hence the obtained type is $(3,5,7).$
 The SAGBI basis is again constructed  using Theorem \ref{th:SAGBIbasis}.

When $D$ is a non-trivial $\alpha-$derivation we know that $B$ already has two elements $\alpha,\beta$ in the spectrum and we need to add one of the derivations which we have studied in sections \ref{sec:22k+1} and \ref{sec:345}.
Note that in the case when $B$ is defined by $f'(\alpha)=f'(\beta)=0$ we can suppose WLOG that $D$ is an $\alpha-$ derivation and this is the same case as above with trivial $\beta-$derivation applied to subalgebra $B'$ defined by
$f'(\alpha)=0, D(f)=0.$

Thus we only need to consider the case when $B$ is defined by $f(\alpha)=f(\beta),af'(\alpha)+bf'(\beta)=0.$ We may WLOG assume that $a=1$, since we may first assume that $a \neq 0$ (if not interchange $\alpha$ and $\beta$), and then divide by $a$. Then $f'(\beta)$ is one of the derivations and we have described two others. If $b=0$  then  $T(A)=(3,4,5)$ and we get that $D(f)=a'f'''(\alpha)+b'f''(\alpha)+cf'(\beta)$ and we can use $a,b$ instead for $a',b'.$  The SAGBI basis for $B$ in this case can be chosen as $$g=(x-\alpha)^2(x-\beta),(x-\alpha)g,(x-\alpha)^2g.$$ Let $$K=D(g)=6a+2(\alpha-\beta)b+(\alpha-\beta)^2c.$$ If $K\neq 0$ then three disappears from the semigroup and the obtained type is $(4,5,6,7).$
As $D(g^2)=0$ and $D(g^2)(x-\alpha)=0$ they can be included in the SAGBI basis directly. We get the remaining two elements using Theorem \ref{th:SAGBIbasis}.

If $K=0$ we need to look at $$D((x-\alpha)g)=6(\alpha-\beta)a-(\alpha-\beta)^3c=$$
$$(\alpha-\beta)(6a-(\alpha-\beta)^2c).$$
Suppose first that $$6a-(\alpha-\beta)^2c=0.$$ Then both the polynomials of degree three and four belong to $A=\ker D.$ This corresponds to type $(3,4)$.  Subtracting $K$ and dividing by $\alpha-\beta$ we get that $b=(\beta-\alpha)c.$ Note that
$c\neq 0$ in this case otherwise we get codimension $2$.

In the remaining case four disappears from the semigroup and $T(A)=(3,5,7).$ We include $g$ in SAGBI basis directly, check that $(x-\alpha)^5(x-\beta)^2$ belongs to $\ker D$ as well and need only to get the polynomial of degree five.

Now we go back to the old $a,b$ and consider the only remaining case when $af'(\alpha)+bf'(\beta)=0$ with $b\neq 0.$ We can suppose $a\neq 0$ (otherwise we get up to notation the previous case) and WLOG put $a=1.$
If $b=1$ then $B$ has type $(2,5),$ otherwise  $T(B)=(3,4,5).$ In both cases we know what the derivations are.

 We begin with the case $b=1.$ Then as we know from section \ref{sec:22k+1}- that $D=a'f'(\alpha)+b'(f''(\alpha)-f''(\beta)).$  If $a=0$ we get type $(2,7)$ and as a SAGBI basis we can choose
  $(x-\alpha)(x-\beta),  (x-\alpha)^4(x-\beta)^3.$  (Both polynomials satisfy the subalgebra conditions.)
  Otherwise WLOG $a=1$ and we can use $b$ instead of $b'$
 so now $A$ is defined by
 $$f(\alpha)=f(\beta), f'(\alpha)=-f'(\beta); f'(\alpha)+b(f''(\alpha)-f''(\beta))=0.$$ Degree two disappears from the semigroup and we get $T(A)=(4,5,6,7).$
 If we choose $(x-\alpha)(x-\beta),  (x-\alpha)^3(x-\beta)^2$ as a SAGBI basis for $B$ (both polynomials satisfy the subalgebra conditions) then $(x-\alpha)^2(x-\beta)^2,(x-\alpha)^3(x-\beta)^3,(x-\alpha)^4(x-\beta)^3$ can be included in the
  SAGBI basis for $A$ directly and it remains to use Theorem \ref{th:SAGBIbasis} to find a polynomial of degree five.

  Now recover the old variable $b$ and consider the remaining case $f'(\alpha)+bf'(\beta)=0$ with $b\neq 1.$ As we have found in the section \ref{sec:345} we can describe $D$ as
  $$D(f)=a'f'(\beta)+c\left(f''(\alpha)-b^2f''(\beta)\right).$$ The type of $B$ is now $(3,4,5)$ and  the SAGBI basis for $B$ is more complicated:
   $$g_k=(x-\alpha)(x-\beta)(x^k-\gamma_k), k=1,2,3$$ with
   $$ \gamma_k=\frac{\alpha^k-b\beta^k}{1-b}.$$

   First we need to study the case when
   $D(g_1)=D(g_2)=0$ which  gives the type $(3,4)$ and  $\{g_1,g_2\}$ as a SAGBI basis. We replace $a'$ by $a$ as usual. Using Maple we get a system
   \begin{equation}\label{eq:s2case4} -a\alpha-4cb+a\beta+2c+4cb^2-2cb^3=0;\end{equation}
   $$a\alpha^2-6cb\alpha+2cb^2\alpha+4c\alpha-4cb^3\beta-2cb\beta+a\beta^2+6cb^2\beta=0.$$

  Solving the system (using the package Groebner) we find that the only solution of interest to us is
     $$b=-1,12c=a(\beta-\alpha).$$

   If we do not have both these conditions satisfies, but only \ref{eq:s2case4} is valid then four disappears from the semigroup and we get $T(A)=(3,5,7)$.
   If the first equation fails we get $T(A)=(4,5,6,7).$ It remains to calculate SAGBI bases that is not specially nice.
   But we can at least find nice polynomials in degree six and seven:
   $$(x-\alpha)^k(x-\beta)^3, k=3,4$$ They belong to $A$
as we use at most second order derivatives in our subalgebra conditions.
The remaining  polynomials are ugly and  we get them using Theorem \ref{th:SAGBIbasis}.

\end{proof}
\section[Codimension $3$, $s=3$]{Subalgebras of codimension three with three elements in the spectrum.}

\begin{theorem}
	If the algebra $A$ has codimension $3$ and a spectrum of three elements $\alpha, \beta, \gamma$ then
	$A$ is one of the following algebras:

	\begin{enumerate}
		\item $A = \{f(x) | f'(\alpha) = f'(\beta) = f'(\gamma) = 0\}$.

By symmetry we can suppose that $2\gamma\neq \alpha+\beta.$
The type of $A$ is $T(A) = (4, 5, 6, 7)$ and a SAGBI basis is given by
			$$
				 \frac{x^4}{4} - \frac{(\alpha + \beta + \gamma) x^3}{3}
					+ \frac{(\alpha\beta + \alpha\gamma + \beta\gamma)x^2}{2} -
					\alpha\beta\gamma x, $$
			$$	 (x - \alpha)^2 (x - \beta)^2 \left(x - \left(\gamma +
					\frac{(\gamma - \alpha)(\gamma - \beta)}{4\gamma - 2(\alpha +
					\beta)}\right)\right), $$
 		$$ (x - \alpha)^2 (x - \beta)^2 (x-\gamma)^k,\quad k=2,3.$$
		
\item $A = \{f(x) | f(\alpha) = f(\beta); af'(\alpha) + bf'(\beta)=
			f'(\gamma) = 0\}$, where $a\neq 0.$  The structure of $A$ depends on the value of
			$\gamma$.  We have the following cases.
			\begin{itemize}
				\item $\gamma = \frac{\alpha + \beta}{2}$: If $a = b$ then
					$T(A) = (2, 7)$ and a SAGBI basis is
$$(x-\alpha)(x-\beta),\, (x-\alpha)^2(x-\beta)^2(x-\gamma)^3.$$
  Otherwise $T(A) =
					(4, 5, 6, 7)$ and  a SAGBI basis is given by $p_0, g, p_2,
					p_3$ where
					\begin{align*}
						p_i &= (x - \alpha)^2 (x - \beta)^2 (x - \gamma) ^ i, \\
						g &= (x - \alpha) (x - \beta) (x-\gamma)^2 \left(x - \frac{b \beta - a\alpha}{b - a}\right).
					\end{align*}
				\item $\gamma \not = \frac{\alpha + \beta}{2}$: Here  $T(A)=(3,4)$ is impossible.

                 If 	$b=
		\frac{a(\alpha-\gamma)(\alpha + 2 \beta - 3 \gamma)}{(\beta - \gamma) (\beta + 2 \alpha - 3 \gamma)}
	$ then $T(A)=(3,5,7)$ and a SAGBI basis is given by
$$q(x)=(x-\alpha)(x-\beta)\left(x-\frac{a\alpha-b\beta}{a-b}\right),$$
$$r(x)=(x - \alpha) (x - \beta) (x -\gamma)^2\times$$
$$\times \left(x
								- \frac{b \beta(\beta - \gamma)^{2} -
									a\alpha(\alpha - \gamma)^2}{b (\beta - \gamma)^{2} -
									a(\alpha - \gamma)^2}\right),$$
 $$(x-\alpha)^2(x-\beta)^2(x-\gamma)^{3}.$$
 The denominators are always non-zero in this case.

 Otherwise $T(A)=(4,5,6,7)$ and if
 \begin{align*}
						D(f) &= af'(\alpha) + bf'(\beta), \\
												p(x) &= (x-\gamma)^2 (x-\alpha)(x-\beta).
					\end{align*}
 then  a SAGBI basis is given by
							\begin{align*}
								&D(q)p(x) - q(x)D(p), \\
														&(x - \alpha)^2(x - \beta)^2(x - \gamma)^2, \\
								&(x - \alpha)^2(x - \beta)^2(x - \gamma)^3,
							\end{align*}
and either $p(x)$ or $r(x)$
														depending on whether $b(\beta - \gamma)^{2} = a{(\alpha - \gamma)^{2}}{}$ or not.

 \end{itemize}

		\item $A = \{f(x) | f(\alpha) = f(\beta) = f(\gamma); af'(\alpha) +
			bf'(\beta) + cf'(\gamma) = 0\}$, where at least one of $a,b,c$ is different from zero.
 Let
			\begin{align*}
				D(f) &= af'(\alpha) + bf'(\beta) + cf'(\gamma), \\
				p_i &= (x - \alpha) (x - \beta) (x - \gamma)^i.
			\end{align*}

 If $\quad \quad\quad\frac{a}{(\beta-\gamma)^2}=\frac{b}{(\gamma-\alpha)^2}=\frac{c}{(\alpha-\beta)^2}$\\

then $T(A) = (3, 4)$ and a SAGBI basis is given by $$p_1(x),\quad p_2(x).$$
				
 If  $a(\alpha-\beta)(\alpha
					- \gamma)+  b(\beta -\alpha)(\beta-\gamma)+c
					(\gamma - \alpha)(\alpha - \beta)\neq 0$\\

then $T(A) = (4,5,6,7)$ and  a SAGBI basis for $A$
					is given by
					\begin{align*}
					&D(p_1)p_2(x) - p_1(x)D(p_2), \\
					&D(p_1)p_3(x) - p_1(x)D(p_3), \\
					&(x - \alpha)^2 (x - \beta)^2 (x - \gamma)^2, \\
					&(x - \alpha)^2 (x - \beta)^2 (x - \gamma)^3.
					\end{align*}

		In the remaining case $T(A) = (3, 5, 7)$ and a SAGBI
					basis for $A$ is given by
					\begin{align*}
					&p_1(x), \\
					&D(p_2)p_3(x) - p_2(x)D(p_3), \\
					&(x - \alpha)^2 (x - \beta)^2 (x - \gamma)^3.
					\end{align*}

		\item $A = \{f(x) | f(\alpha) = f(\beta); f'(\gamma) = af''(\gamma) + b
			f'''(\gamma) = 0\}$. The structure of $A$ depends on the value of
			$\gamma$. We have the following cases
			\begin{itemize}
				\item $\gamma = \frac{\alpha + \beta}{2}$:
					If $a = 0$ then $T(A) = (2, 7)$ and a SAGBI basis is
$$(x-\alpha)(x-\beta),\, (x-\alpha)(x-\beta)(x-\gamma)^5.$$

 Otherwise 					$T(A) = (4, 5, 6, 7)$ and (for  $a=1$) a SAGBI basis is given by
					\begin{align*}
						 & (x - \gamma)^4, \\
						& (x - \alpha) (x - \beta) (x - \gamma)^2 (x - \gamma - 3 b), \\
						 & (x - \alpha) (x - \beta) (x - \gamma)^4, \\
						 & (x - \alpha) (x - \beta) (x - \gamma)^5.
					\end{align*}
				\item $\gamma \not = \frac{\alpha + \beta}{2}$: Here the
					structure of $A$ depends on the values of $a, b$ and
					we branch on weather $b$ is zero or not. Let
					\begin{align*}
						D(f) &= af''(\gamma) + bf'''(\gamma) \\
						q(x) &= (x-\gamma)^2
						\left(
							x -
							\frac{\alpha^{2} + \alpha \beta + \beta^{2} - 2 \, {\left(\alpha + \beta\right)} \gamma + \gamma^{2}}
							{\alpha + \beta - 2 \, \gamma}
						\right), \\
						p_i(x) &= (x-\gamma)^2 (x-\alpha)(x-\beta)^i.
					\end{align*}

					If $b = 0$ then we may assume that $a = 1$ and get the following cases
					\begin{itemize}

						\item If $3(\gamma - \alpha)(\gamma - \beta) \not = -(\alpha - \beta)^2$ then
							$T(A) = (4, 5, 6, 7)$ and
							\begin{align*}
								&(x - \alpha)(x - \beta)(x-\gamma)\left(
								x-\gamma -
								\frac{(\alpha - \gamma)(\beta - \gamma)}{	(\alpha + \beta - 2\gamma)}\right),\\
								&(x - \alpha)(x - \beta)(x - \gamma)^k,\quad k=3,4,5
													\end{align*}
							is a SAGBI basis for $A$.

						\item Otherwise $T(A) =
							(3, 5, 7)$ and
							\begin{align*}
								&q, \\
								&(x - \alpha)(x - \beta)(x - \gamma)^3, \\
								&(x - \alpha)(x - \beta)(x - \gamma)^5.
							\end{align*}
							is a SAGBI basis for $A$.					\end{itemize}

					Else, if $b \not = 0$ we may assume that $b = 1$ and get the following cases,
					\begin{itemize}
\item If $(\alpha-\gamma)^2+(\beta-\gamma)^2=0 $ and $a = \frac{3(\beta + \alpha - 2\gamma)}{(\gamma - \alpha) (\gamma - \beta)}$ then  $T(A) = (3, 4)$ and a SAGBI basis is given by $q, p_1$.
						\item If $a \not = \frac{3 \, {\left(\alpha + \beta - 2 \,
							\gamma\right)}}{(\alpha - \beta)^2 + 3(\gamma -
							\alpha)(\gamma - \beta)}$ then $T(A) = (4,5,6,7)$.
							A SAGBI basis for $A$ is given by
							\begin{align*}
								&D(q)p_2(x) - q(x)D(p_2), \\
								&D(q)p_3(x) - q(x)D(p_3), \\
								&(x - \alpha)(x - \beta)(x - \gamma)^4, \\
								&(x - \alpha)(x - \beta)(x - \gamma)^5.
							\end{align*}

						\item  If the remaining case $T(A) = (3, 5, 7)$. A SAGBI basis for
							$A$ is given by
							\begin{align*}
								&q, \\
								&D(q)p_3(x) - q(x)D(p_3), \\
								&(x - \alpha)(x - \beta)(x - \gamma)^5.
							\end{align*}

			\end{itemize}

			\end{itemize}
	\end{enumerate}
\end{theorem}

\begin{proof} We follow the plan that was pointed out in section \ref{sec:plancodim3} and can suppose that $A=\ker D$ for some subalgebra $B$ of codimension two.

If $D$ is a trivial derivation in $\gamma$ then $|Sp(B)|=2$ and we have two alternatives for $B.$

	In the first alternative $B = \{f(x) |
	f'(\alpha) = f'(\beta) = 0\}$, thus
$A = \{f(x) |
	f'(\alpha) = f'(\beta) =f'(\gamma)= 0\}.$

	 No non-constant polynomial of degree less than four has a
	derivative equal to zero in three points or more. Hence $T(A) = (4,5,6,7)$ and
	a it is easily verified that the given basis satisfies the conditions. \\

In the second alternative  $B= \{f(x) | f(\alpha) =
	f(\beta); af'(\alpha) +bf'(\beta) =0\}$, where WLOG $a\neq 0.$ We have
$$A = \{f(x) | f(\alpha) =
	f(\beta); af'(\alpha) +bf'(\beta)= f'(\gamma) = 0\}$$ and two subcases.

Suppose first that $a=b,$ thus $T(B)=(2,5)$ and $q=(x-\alpha)(x-\beta)\in B.$
If $$D(q)=0\Leftrightarrow \alpha+\beta=2\gamma$$ we get $T(A)=(2,7)$ and can use Theorem \ref{th:deg2}.

Otherwise $T(A)=(4,5,6,7)$ and we only need to check that the chosen elements for the SAGBI basis satisfy the conditions.

Now suppose that $a\neq b$. It follows that $T(B)=(3,4,5)$ and
 a SAGBI basis for $B$  can be chosen as
   $$g_1=(x-\alpha)(x-\beta)(x-\frac{a\alpha-b\beta}{a-b}),$$
   $$g_k=(x-\alpha)^2(x-\beta)^2(x-\gamma)^{k-2}, k=2,3.$$

  The first basis element is annihilated when
	$$
		D(g_1) = 0
		\Leftrightarrow
				b=
		\frac{a(\alpha-\gamma)(\alpha + 2 \beta - 3 \gamma)}{(\beta - \gamma) (\beta + 2 \alpha - 3 \gamma)}.
	$$

The second basis element is  annihilated when
	\begin{align*}
		D(g_2) = 0
		\Leftrightarrow&
		 2( \gamma-\alpha)(\gamma-\beta )^2 + 2(\gamma-\alpha)^2(\gamma - \beta) = 0 \\
			\Leftrightarrow&
		2\gamma=\alpha+\beta\Leftrightarrow \gamma-\alpha=\beta-\gamma. \\
	\end{align*}

Thus $T(A)=(3,4)$ is impossible because we get $b=a.$  We can easily construct a SAGBI basis when $T(A)=(3,5,7)\Leftrightarrow D(g_1)=0.$ Note that the denominator of $r(x)$ cannot be zero, otherwise we would have $D(g_2)=0.$

Otherwise $T(A)=(4,5,6,7)$ and here the denominator can be zero which gives us two cases.

Now suppose that $D$ is a non-trivial derivation, thus $|Sp(B)|=3.$ Here we again have two alternatives.

	Let first $B = \{f(x) | f(\alpha) = f(\beta) = f(\gamma) \}$. We know that
	$T(B) = (3, 4, 5)$ from Theorem \ref{th:codim2}, and choose a SAGBI basis as $p_i =
	(x-\alpha)(x-\beta)(x-\gamma)^i$ for $i=1,2,3.$   By
	Theorem \ref{th:noderivations} any non-trivial derivation $D$
	of $B$ will be of the form $D = af'(\alpha) + bf'(\beta) + cf'(\gamma)$, where at least one of $a,b,c$ is non-zero.

	The first basis element $p_1$ is annihilated when
	$$
	a(\alpha-\beta)(\alpha-\gamma) +
		b(\beta - \alpha)(\beta-\gamma) +
		(\gamma - \alpha)(\gamma-\beta)=0.
			$$
	The second basis element $p_2$ is annihilated when
	\begin{align*}
		D(p_2) = 0
		&\Leftrightarrow
		a(\alpha-\beta)(\alpha-\gamma)^2 +
		b(\beta - \alpha)(\beta-\gamma)^2
		= 0.
			\end{align*}
	Solving the system of equations $D(p_1)=D(p_2)=0$ by Maple, we get the conditions for $T(A)=(3,4).$
If $T(p_1)\neq 0$ we get $T(A)=(4,5,6,7).$
In the remaining case we have $T(A)=(3,5,7).$\\

The last remaining  case is a non-trivial derivation $D$ of
 $B = \{f(x) | f(\alpha) = f(\beta); f'(\gamma) = 0\}$. If it is an $\alpha-$ derivation we can assume this condition is imposed before the $\gamma-$derivation and in this way reduce the current case to a previous one.
 So it is sufficient to consider a $\gamma-$derivation $$D(f)= af''(\gamma) + bf'''(\gamma).$$

  This one is
	a bit trickier and we need to branch on the value of $\gamma$.

	We begin with the case when $\alpha + \beta = 2\gamma$, thus $T(B)=(2,5).$  $D$ annihilates the element of degree two if and only if $a=0.$  Then $T(A) = (2, 7)$
	and we can use Theorem \ref{th:deg2}. If  $a\neq 0$
	 then $A$ will not contain
	any polynomials of degree two and hence $T(A) = (4, 5, 6, 7)$. Moreover, as
	we are interested in the kernel of each derivation, we may assume that
	$a=1$. One may verify that the given basis satisfies the type
	and subalgebra conditions. \\

	In the case when $\alpha + \beta \not = 2\gamma$, we have $T(B) = (3, 4, 5)$
	and a SAGBI basis for $B$ is given by
	\begin{align*}
		q(x) &= (x-\gamma)^2
		\left(
			x -
			\frac{\alpha^{2} + \alpha \beta + \beta^{2} - 2 \, {\left(\alpha + \beta\right)} \gamma + \gamma^{2}}
			{\alpha + \beta - 2 \, \gamma}
		\right), \\
		p_1(x) &= (x-\gamma)^2 (x-\alpha)(x-\beta)^k,\quad k=1,2 \\
	\end{align*}
	 To simplify the
	computations, we consider two cases. First when $b = 0$ and then when $b \not =
	0$. In the first case we can assume $a=1$ and we have $A = \{f(x) |
	f(\alpha) = f(\beta); f''(\gamma) = f'(\gamma) = 0\}$. We now check conditions
	on the spectral elements for when different basis elements of $B$ are annihilated.
	The first basis element is annihilated when
	\begin{align*}
		D(q) = 0
		\Leftrightarrow&
		2\left(\gamma - \frac{\alpha^{2} + \alpha \beta + \beta^{2} - 2 \, {\left(\alpha + \beta\right)} \gamma + \gamma^{2}}
			{\alpha + \beta - 2 \, \gamma}
		\right) = 0 \\
		\Leftrightarrow&
		\frac{-\alpha^{2} - \alpha \beta - \beta^{2} + 3 \, {\left(\alpha + \beta\right)} \gamma - 3\gamma^{2}}
			{\alpha + \beta - 2 \, \gamma}
		= 0 \\
		\Leftrightarrow&
		\frac{-(\alpha - \beta)^2 - 3(\alpha - \gamma)(\beta - \gamma)}
			{\alpha + \beta - 2 \, \gamma}
		= 0 \\
	\end{align*}
	and since $\alpha + \beta - 2\gamma \not = 0$,
	$$
	D(q) = 0 \Leftrightarrow (\alpha - \beta)^2 = -3(\alpha - \gamma)(\beta - \gamma).
	$$
	The second basis element is annihilated when
	\begin{align}
		D(p_1) = 0
		\Leftrightarrow&
		2(\gamma - \alpha) (\gamma - \beta) = 0,
	\end{align}
	which is not possible as the spectral elements are not equal. The same
	holds for the third basis element. Hence we only get two cases, when $q$ is
	annihilated and when it is not. \\

	Now we treat the case when $b \not = 0$, where we are free to assume that
	$b=1$ and hence $A = \{f(x) | f(\alpha) = f(\beta); af''(\gamma) +
	f'''(\gamma) = f'(\gamma) = 0\}$. We check for conditions on $a$ that
	leads to annihilation of basis elements. The first basis element is
	annihilated when
	\begin{align*}
		D(q) = 0
		\Leftrightarrow&
		2a\left(\gamma - \frac{\alpha^{2} + \alpha \beta + \beta^{2} - 2 \, {\left(\alpha + \beta\right)} \gamma + \gamma^{2}}
			{\alpha + \beta - 2 \, \gamma}
		\right) + 6 = 0 \\
		\Leftrightarrow&
		a\frac{(\alpha - \beta)^2 + 3(\alpha - \gamma)(\beta - \gamma)}
			{\alpha + \beta - 2 \, \gamma}
		= 3 \\
		\Leftrightarrow&
		a = \frac{3(\alpha + \beta - 2 \, \gamma)}{(\alpha - \beta)^2 + 3(\alpha - \gamma)(\beta - \gamma)}.
	\end{align*}
	The second basis element is annihilated when
	\begin{align*}
		D(p_1) = 0
		\Leftrightarrow&
		2a(\gamma - \alpha) (\gamma - \beta)
		+ 6(\gamma - \beta) + 6(\gamma - \alpha)
		= 0 \\
		\Leftrightarrow&
		a = \frac{3(\beta + \alpha - 2\gamma)}{(\gamma - \alpha) (\gamma - \beta)}.
	\end{align*}

It is easy to see that to annihilate both elements simultaneously and get $T(A)=(3,4)$ we need
$$(\alpha - \beta)^2+2(\gamma - \alpha) (\gamma - \beta)=0\Leftrightarrow $$
$$(\alpha-\gamma)^2+(\beta-\gamma)^2=0.$$

	This concludes all the different cases.
\end{proof}

\section[Codimension $3$, $s=4$]{Subalgebras of codimension three with four elements in the spectrum.}
\begin{theorem}
	If the algebra $A$ of codimension three has a spectrum consisting of four elements $\alpha, \beta, \gamma, \delta$ then $A$ is one of the following algebras
	\begin{enumerate}
		\item $A=\{f(x)|f(\alpha)=f(\beta)=f(\gamma)=f(\delta)\}.$

		The type of this algebra is (4,5,6,7) and a SAGBI basis is:
		$$ (x-\alpha)(x-\beta)(x-\gamma)(x-\delta), (x-\alpha)^2(x-\beta)(x-\gamma)(x-\delta),$$ $$(x-\alpha)^3(x-\beta)(x-\gamma)(x-\delta), (x-\alpha)^4(x-\beta)(x-\gamma)(x-\delta) $$

		\item  $A=\{f(x)|f(\alpha)=f(\beta); f(\gamma)=f(\delta); af'(\alpha)+bf'(\beta)=0\}.$
		
		Case I $\alpha+\beta=\gamma+\delta$
		
		If $a=b \neq 0$ then $A$ is type (2,7) and a SAGBI basis is given by
		$$(x-\alpha)(x-\beta), (x-\alpha)^2(x-\beta)^3(x-\gamma)(x-\delta)$$
		
		If $a \neq b$, $b \neq 0$ then $A$ is type (4,5,6,7) and a SAGBI basis is given by
		$$(x-\alpha)^2(x-\beta)^2,$$ $$(x-\alpha)(x-\beta) \left[ (a-b)(x-\beta)(x-\gamma)(x-\delta)-a(\alpha-\beta)(\alpha-\gamma)(\alpha-\delta)\right]$$
		$$ (x-\alpha)^3(x-\beta)^3,$$ $$(x-\alpha)^2(x-\beta)^3(x-\gamma)(x-\delta).$$
		
		Case II $\Delta={\gamma+\delta-\alpha-\beta} \neq 0$. Let $\tau$ be defined by
		$$\tau=\frac {\gamma^2+\gamma\delta+\delta^2+\alpha\beta-(\alpha+\beta)(\gamma+\delta)}{\Delta}$$
		If $a=0$ then $A$ is of type (4,5,6,7) and a SAGBI basis is given by
		$$(x-\alpha)(x-\beta)\left[ (\beta-\tau)(x-\gamma)(x-\delta)+(\beta-\gamma)(\beta-\delta)(x-\tau)\right]$$
		$$(x-\alpha)(x-\beta)^2(x-\gamma)(x-\delta)$$
		$$(x-\alpha)^2(x-\beta)^2(x-\tau)^2$$
		$$(x-\alpha)^2(x-\beta)^2(x-\gamma)(x-\delta)(x-\tau)$$
		
		If $a=b \neq 0$ then $A$ also type (4,5,6,7) and the SAGBI basis has elements
		$$(x-\alpha)(x-\beta)\left[ (x-\gamma)(x-\delta)-(\alpha+\beta-\gamma-\delta)(x-\tau)\right],$$
		$$(x-\alpha)(x-\beta)\left[ (x-\beta)(x-\gamma)(x-\delta)-(\alpha-\gamma)(\alpha-\delta)(x-\tau)\right]$$
		together with the degree six and seven elements from the previous case.
		
		$a \neq b$, $a \neq 0$, $b \neq 0$. We may assume WLOG that $b=1$. In this case the type of $A$ depends both on a spectrum condition $C=0$ where $C=(\alpha+\beta)(\gamma+\delta)-2\alpha\beta-\gamma^2-\delta^2$ and on the value of $a$.
		
		If both $C=0$ and $a=\frac{\tau-\beta}{\tau-\alpha}$ hold then T(A)=(3,4) and $$(x-\alpha)(x-\beta)(x-\tau),$$ $$(x-\alpha)(x-\beta)(x-\gamma)(x-\delta)$$ is a SAGBI basis.
		
		If $a=\frac{\tau-\beta}{\tau-\alpha}$ but $C \neq 0$
		then T(A)=(3,5,7) and a SAGBI basis is given by
		$$(x-\alpha)(x-\beta)(x-\tau)$$
		$$(x-\alpha)(x-\beta)(x-\gamma)(x-\delta)(c(x-\alpha)+d)$$
		$$(x-\alpha)^3(x-\beta)^2(x-\gamma)(x-\delta)$$ where $c=a(\alpha-\gamma)(\alpha-\delta)-(\beta-\gamma)(\beta-\delta)$ and $d=-(\alpha-\beta)(\gamma-\beta)(\delta-\beta)$
		
		If $a \neq
		\frac{\tau-\beta}{\tau-\alpha}$ then T(A)=(4,5,6,7) and a SAGBI basis is given by
		$$(x-\alpha)(x-\beta)(c(x-\gamma)(x-\delta)+d(x-\tau))$$
		$$(x-\alpha)(x-\beta)(c(x-\alpha)(x-\gamma)(x-\delta)+e(x-\tau))$$
		$$(x-\alpha)^k (x-\beta)^2(x-\gamma)(x-\delta), k=2,3$$
		Here $c=a(\alpha-\tau)-(\beta-\tau)$, $d=(\beta-\gamma)(\beta-\delta)$ and $e=-(\beta-\alpha)(\beta-\gamma)(\beta-\delta)$
		\item $A=\{f(x)| f(\alpha)=f(\beta), f'(\gamma)=f'(\delta)=0\}.$
		
		If $\frac{2\beta+\gamma-3\delta}{2\alpha+\gamma-3\delta}=\frac{(\alpha-\gamma)^2}{(\beta-\gamma)^2}=\frac{(\beta-\delta)^2}{(\alpha-\delta)^2}$ then T(A)=(3,4) and a SAGBI basis is given by
		$$(x-\gamma)^2(2x+\gamma-3\delta), (x-\gamma)^2(x-\delta)^2$$

		If $\frac{2\beta+\gamma-3\delta}{2\alpha+\gamma-3\delta}=\frac{(\alpha-\gamma)^2}{(\beta-\gamma)^2} \neq\frac{(\beta-\delta)^2}{(\alpha-\delta)^2} $ then T(A)=(3,5,7) and a SAGBI basis is given by
		$$(x-\gamma)^2(2x+\gamma-3\delta),
		(x-\gamma)^2(x-\delta)^2(c+d(x-\gamma))$$
		$$(x-\gamma)^3(x-\delta)^2(x-\alpha)(x-\beta)$$
		where $c= \left( \alpha-\gamma \right) ^{3} \left( \alpha-\delta \right) ^{2}-
		\left( -\gamma+\beta \right) ^{3} \left( \beta-\delta \right) ^{2}$ and $d=\left( \alpha-\gamma \right) ^{2} \left( \alpha-\delta \right) ^{2}-
		\left( -\gamma+\beta \right) ^{2} \left( \beta-\delta \right) ^{2}
		$

		If $\frac{2\beta+\gamma-3\delta}{2\alpha+\gamma-3\delta} \neq \frac{(\alpha-\gamma)^2}{(\beta-\gamma)^2} $ then T(A)=(4,5,6,7)
		and a SAGBI basis is given by
		$$(x-\gamma)^2((x-\delta)^2+c(2x+\gamma-3\delta)),$$
		$$(x-\gamma)^2((x-\gamma)(x-\delta)+d(2x+\gamma-3\delta))$$
		$$(x-\alpha)(x-\beta)(x-\gamma)^2((x-\delta)^2$$
		$$(x-\alpha)(x-\beta)^2(x-\gamma)^2((x-\delta)^2$$
		where $c=-\frac{-{\alpha}^{2}+ \left( \gamma+\delta \right) \alpha-{\beta}^{2}+\left( \gamma+\delta \right) \beta-2\,\delta\,\gamma}{2\,{\alpha}^{2}+ \left( 2\,\beta-3\,\delta-3\,\gamma \right) \alpha+2\,{\beta}^{2}+ \left( -3\,\delta-3\,\gamma \right) \beta+6\,\delta\,\gamma}$ and $$d=({\alpha}^{4}+ \left( \beta-2\,\delta-3\,\gamma \right) {\alpha}^{3}+$$ $$+\left( {\beta}^{2}+ \left( -2\,\delta-3\,\gamma \right) \beta+{\delta}^{2}+6\,\delta\,\gamma+3\,{\gamma}^{2} \right) {\alpha}^{2}+$$
		$$ \left( {\beta}^{3}+ \left( -2\,\delta-3\,\gamma \right) {\beta}^{2}+
\left( {\delta}^{2}+6\,\delta\,\gamma+3\,{\gamma}^{2} \right) \beta-3\,
{\delta}^{2}\gamma-6\,{\gamma}^{2}\delta-{\gamma}^{3} \right) \alpha+$$ $${\beta}^{4}+\left(-2\,\delta-3\,\gamma \right) {\beta}^{3}+ \left( {\delta}^{2}+6\,\delta\,\gamma+3\,{\gamma}^{2} \right) {\beta}^{2}+$$ $$+ \left( -3\,
		{\delta}^{2}\gamma-6\,{\gamma}^{2}\delta-{\gamma}^{3} \right) \beta+3
		\,{\delta}^{2}{\gamma}^{2}+2\,\delta\,{\gamma}^{3}
		) / $$
		$$ (-2\,{\alpha}^{2}+ \left( 3\,\gamma-2\,\beta+2\,\delta \right) \alpha-2 {\beta}^{2}+\left( 3\,\gamma+2\,\delta \right) \beta-4\,\delta\,\gamma))$$
	
	\item $A=\{f(x)| f(\alpha)=f(\beta)=f(\gamma), f'(\delta)=0\}.$
	
	If $\frac{1}{\alpha-\delta}+\frac{1}{\beta-\delta}+\frac{1}{\gamma-\delta} \neq 0$ then T(A)=(4,5,6,7) and a SAGBI basis is given by
	$$(x-\alpha)(x-\beta)(x-\gamma)g_j(x)$$ where
	$$g_j(x)= 1+(\frac{1}{\alpha-\delta}+\frac{1}{\beta-\delta}+\frac{1}{\gamma-\delta})(x-\delta)+(x-\delta)^j,$$ $j=0,2,3,4$
	
	If $\frac{1}{\alpha-\delta}+\frac{1}{\beta-\delta}+\frac{1}{\gamma-\delta} = 0$ then T(A)=(3,5,7) and a SAGBI basis is given by
	$$(x-\alpha)(x-\beta)(x-\gamma)\left[1+(x-\delta))^j\right]$$ where $j=0,2,4$
	
\end{enumerate}
\end{theorem}
\begin{proof} Either $A$ is defined without any derivations which results in case 1, or $A$ is obtained by adding a derivation to a subalgebra $B$ of codimension two. In the second case we either add an $\alpha$-derivation where $\alpha \in Sp(B)$ to the case $s=4$ in Theorem \ref{th:codim2} which results in case 2, or we add an $\alpha$-derivation where $\alpha \not \in Sp(B)$ to one of the $s=3$ cases in codimension two resulting in cases 3 and 4.

For the first case it is obvious that $A$ contains exactly all polynomials of the form $c+g(x)(x-\alpha)(x-\beta)(x-\gamma)(x-\delta)$ where $c$ is any constant and $g(x)$ any polynomial. This algebra is clearly of type (4,5,6,7), and it is easy to verify that the given set of polynomials are in $A$ and of the exactly those degrees, and hence constitute a SAGBI basis.

For the second case we have two subcases depending on the type of $B$. If $\alpha+\beta=\gamma+\delta$ then $B$ is of type (2,5). We start from a SAGBI basis for $B$: $g_1=(x-\alpha)(x-\beta)$, $g_2=(x-\alpha)(x-\beta)^2(x-\gamma)(x-\delta)$ and then obtain $A$ from $B$ by adding a condition $D(f)=0$. We may (after interchanging labels of the elements in $Sp(B)$) WLOG assume that $D$ is an $\alpha$-derivation. From section \ref{sec:22k+1} we know that such a derivation is of the form $D(f)=af'(\alpha)+bf'(\beta)$. Now $D(g_1)=(\alpha-\beta)(a-b)$. Thus $A$ contains an element of degree two if and only if $a=b$. If $a=0$ we may assume $b \neq 0$ since we otherwise get codimension two. Thus, for $a=0$, degree two is missing in $A$ and hence its type (4,5,6,7) is obtained by removing two from the semigroup of $B$. A direct check shows that the given basis polynomials are all in $A$ and we conclude that they are a SAGBI basis since they are of appropriate degrees. (If $b$ but not $a$ equals zero we get the same case $A$ by letting $\alpha$ and $\beta$ switch names.)

If $a=b \neq 0$ then $A$ does contain $g_1$ and the only type of codimension two containing degree two is (2,7). Again a direct check shows that the polynomials given in the theorem are a SAGBI basis.

If $a \neq b$ then $A$ does not contain an element of degree two and as $B$ is type (2,5) $A$ must be of type (4,5,6,7). Applying $D$ to the SAGBI basis for the case $a=0$ we find that all elements except the second belong to $A$. Using the recipe from section \ref{sec:SAGBI} we replace this element by one belonging to $A$.

Let us now turn to case II, that is when $\Delta \neq 0$. This corresponds to $B$ being of type (3,4,5). We then have a SAGBI basis for B: $g_1=(x-\alpha)(x-\beta)(x-\tau)$, where $\tau$ is uniquely determined to satisfy $g_1(\gamma)=g_1(\delta)$, $g_2=(x-\alpha)(x-\beta)(x-\gamma)(x-\delta)$ and $g_3=(x-\alpha)(x-\beta)^2(x-\gamma)(x-\delta)$.

Again, we obtain $A$ from $B$ by adding a condition $D(f)=0$ of the form $D(f)=af'(\alpha)+bf'(\beta)$. The possible types of $A$ are (4,5,6,7), (3,5,7) or (3,4).

Note that if $A$ contains an element $g$ of degree three then $A$ must contain $g_1$ since $g$ is also in $B$ and the only way to build $g$ from the SAGBI basis of $B$ is as $g_1$ (up to multiplication and addition of a constant).

$A$ contains polynomials of degree three $\Leftrightarrow$ $a(\tau-\alpha)-b(\tau-\beta)=0$
Moreover $A$ contains polynomials of degree four $\Leftrightarrow$ $g_2+cg_1 \in A$ for some $c$ $\Leftrightarrow$ $D(g_2+cg_1)=0$ for some $c$, so clearly if $A$ does not contain degree three, then it will contain degree four. (We already knew this from the list of possible types however.) If $A$ does contain degree three then it may contain degree four or not. It will depend on the derivation used in the definition of $A$. Again we have cases:

If $a=0$ we may WLOG assume that $b=1$. Then we do not have any element of degree three in $A$ so its type is (4,5,6,7). We construct a SAGBI basis by forming elements of degree 4,5,6,7 from the known SAGBI basis $g_1=(x-\alpha)(x-\beta)(x-\tau)$,$g_2=(x-\alpha)(x-\beta)(x-\gamma)(x-\delta)$, $g_3=(x-\alpha)(x-\beta)^2(x-\gamma)(x-\delta)$ of $B$. All new basis elements except the one of degree four are in $A$ and we modify that element using, once again, the method from section \ref{sec:SAGBI}.

$a=b \neq 0$ In this case $A$ does not contain degree three, so the type of A must be (4,5,6,7). We proceed as in the above case with the exception that both the generators of degree four and five need to be modified.

$a \neq b$, $a \neq 0$, $b \neq 0$. We may assume WLOG that $b=1$. In this case the type of $A$ depends both on a spectrum condition $C=0$ where $C=(\alpha+\beta)(\gamma+\delta)-2\alpha\beta-\gamma^2-\delta^2$ and on the value of $a$.

If both $C=0$ and $a=\frac{\tau-\beta}{\tau-\alpha}$ hold then T(A)=(3,4) and $$(x-\alpha)(x-\beta)(x-\tau),$$ $$(x-\alpha)(x-\beta)(x-\gamma)(x-\delta)$$ is a SAGBI basis.

If $a=\frac{\tau-\beta}{\tau-\alpha}$ but $C \neq 0$ then T(A)=(3,5,7) and a SAGBI basis is given by
$$(x-\alpha)(x-\beta)(x-\tau)$$
$$(x-\alpha)(x-\beta)(x-\gamma)(x-\delta)(c(x-\alpha)+d)$$
$$(x-\alpha)^3(x-\beta)^2(x-\gamma)(x-\delta)$$ where $c=a(\alpha-\gamma)(\alpha-\delta)-(\beta-\gamma)(\beta-\delta)$ and $d=-(\alpha-\beta)(\gamma-\beta)(\delta-\beta)$

If $b \neq
\frac{\tau-\alpha}{\tau-\beta}$ T(A)=(4,5,6,7). A SAGBI basis is given by
$$(x-\alpha)(x-\beta)(c(x-\gamma)(x-\delta)+d(x-\tau))$$
$$(x-\alpha)(x-\beta)(c(x-\alpha)(x-\gamma)(x-\delta)+e(x-\tau))$$
$$(x-\alpha)^k (x-\beta)^2(x-\gamma)(x-\delta), k=2,3$$
Here $c=a(\alpha-\tau)-(\beta-\tau)$, $d=(\beta-\gamma)(\beta-\delta)$ and $e=-(\beta-\alpha)(\beta-\gamma)(\beta-\delta)$

For case three we start from $B=\{f(x)| f'(\gamma)=0, f'(\delta)=0\}$ and add a condition $L(f)=f(\alpha)-f(\beta)=0$

For $B$ we know that $g_1=(x-\gamma)^2(2x+\gamma-3\delta), g_2=(x-\gamma)^2(x-\delta)^2, g_3=(x-\gamma)^3(x-\delta)^2$ constitute a SAGBI basis. The conditions given in the first case are exactly those needed for both $g_1$ and $g_2$ to be in A, and the second case corresponds to $g_1$ but not $g_2$ being in $A$. From this the given types of $A$ follow. To find SAGBI bases for A we use elements from the basis of $B$ in the first case. In the second case we can use $g_1$ but need to modify $g_2$. The third basis element is obviously in $A$. When the type is (4,5,6,7) we proceed in the same way but need to modify both $g_2$ and $g_3$ while the two highest degree elements are obviously in A.

In the fourth case the elements of the subalgebra can be found explicitly. To simplify computations we change variables to $y=x-\delta$ so that we may assume $\delta=0$. Then note that an element is in $A$ if and only if it is of the form $(y-\alpha)(y-\beta)(y-\gamma)g(y)$ with $g(y)$ any polynomial satisfying $g'(0)=0$. This is equivalent to $g(y)=c(1+(\frac{1}{\alpha-\delta}+\frac{1}{\beta-\delta}+\frac{1}{\gamma-\delta})y)$ plus any terms of degree two or more in $y$. Unless $\frac{1}{\alpha-\delta}+\frac{1}{\beta-\delta}+\frac{1}{\gamma-\delta}=0$ such $g$ exist of all degrees from one and up, showing that T(A)=(4,5,6,7). We pick such elements of the required degrees and then change variables back to $x$. In the exceptional case $\frac{1}{\alpha-\delta}+\frac{1}{\beta-\delta}+\frac{1}{\gamma-\delta}=0$ we find that the type of A is (3,5,7) and a basis can again easily be picked in the set of polynomials in $A$.

\end{proof}

\section[Codimension $3$, $s=5,6$]{Subalgebras of codimension three with large spectrum.}
It remains to consider large values of $s,$ namely five and six. Theorem \ref{th:spectrumsize} gives us a direct description of those subalgebras and we only need to detect their types and construct SAGBI bases.

\begin{theorem}If the algebra $A$ of codimension three has a spectrum consisting six elements then $$A=\{f(x)|f(\alpha)=f(\beta); f(\gamma)=f(\delta); f(\lambda)=f(\mu).$$  Here $\alpha,\beta,\gamma,\delta,\lambda,\mu$ are pairwise different numbers from $\MK.$ Depending on the relations between them we have the following alternatives:

A) If $\alpha+\beta=\gamma+\delta=\lambda+\mu$ then the type is (2,7) and a
SAGBI basis can be chosen as:
  $$(x-\alpha)(x-\beta), (x-\alpha)^2(x-\beta)(x-\gamma)(x-\delta)(x-\lambda)(x-\mu).$$

If $\Delta=\gamma+\delta-\alpha-\beta\neq 0$ let
$q(x)=(x-\alpha)(x-\beta)(x-\tau),$ $$ p_i(x)=(x-\alpha)^i(x-\beta)(x-\gamma)(x-\delta),$$ with
$$\tau=
\frac{\gamma^2+\gamma\delta+\delta^2+\alpha\beta-
(\alpha+\beta)(\gamma+\delta)}{\Delta}.$$

B) If $q(\lambda)=q(\mu)$ and $p_1(\lambda)=p_1(\mu)$ then the type is $(3,4)$ and $p_1,q$ form a SAGBI basis. Example \ref{ex:spectere4} shows that such subalgebras really exist.

C) If $q(\lambda)=q(\mu)$ but $p_1(\lambda)\neq p_1(\mu)$
then the type is (3,5,7) and a
SAGBI basis can be chosen as:
  $$q,\quad (x-\alpha)(x-\beta)(x-\gamma)(x-\delta)(x-c)$$
 $$ r=(x-\alpha)^2(x-\beta)(x-\gamma)(x-\delta)(x-\lambda)(x-\mu),$$
 where the constant $c$ can be found from the condition $f(\lambda)=f(\mu).$

D) If $q(\lambda)\neq q(\mu)$ then the type is $(4,5,6,7)$ and a (non-normalized)
SAGBI basis can be chosen as:
  $$(q(\lambda)-q(\mu)p_i(x)-(p_i(\lambda)-p_i(\mu))q(x);\quad i=1,2,3,4.$$

 \end{theorem}
\begin{proof} We can get the algebra $A$ from $B$ defined by $f(\alpha)=f(\beta); f(\gamma)=f(\delta)$ so the only question is how to choose a SAGBI basis in both and how to find the type.
If we have a polynomial of degree two in $A$ we can use Theorem \ref{th:deg2} and only slightly modify it by choosing another element of degree seven (that obviously belongs to $A.$)
Otherwise the type of $B$ can be chosen as $(3,4,5)$ and we need to check which of those three degrees disappears. Because $p_i(x)$ obviously belong to $B$ the result follows from the algorithm of constructing SAGBI basis described above, though for the case $(3,5,7)$ we can choose the basis more explicitly. (The chosen elements are obviously in $A$ and have the right degrees.) This would be possible for type $(4,5,6,7)$ as well
e.g. choosing $r,\frac{r}{x-\alpha}$ for the degrees seven and six, but the elements of degree four and five hardly look nice explicitly. We prefer a shorter description.
\end{proof}

Now we consider the case when we have five elements in the spectrum.

\begin{theorem}\label{th:codim3-s5}
	Let $A$ be a subalgebra of $\mathbb{K}[x]$ of codimension three with the spectrum $\{\alpha, \beta, \gamma, \delta, \lambda\}$ and let
	\begin{align*}
		p_i(x) &= (x - \alpha)(x - \beta)(x - \gamma)(x - \delta)(x-\lambda)^i, \\
		g(x) &= \left( x - \frac{\alpha + \beta}{2}\right)^2, \\
		q(x) &= (x-\alpha)(x-\beta)(x-\tau), \\
		\tau &= \frac{\gamma^2 + \gamma \delta + \delta^2 + \alpha \beta - (\alpha + \beta)(\gamma + \delta)}{\Delta}, \\
		\Delta &= \alpha + \beta - \gamma - \delta.
	\end{align*}
	Then $A$ can be categorized as one of the following.
	\begin{itemize}
		\item $A = \{f(x) | f(\alpha) = f(\beta) = f(\lambda); f(\gamma) = f(\delta)\}$.
			In this case let $L(f) = f(\alpha) - f(\lambda)$. The possible types of $A$ are given as follows.
			\begin{itemize}
				\item $T(A) = (4,5,6,7)$. This occurs when $\Delta = 0$ or when $\Delta \not = 0$ and $L(q) \not = 0$. A possible SAGBI basis for the first case is given by
					$$
					L(p_0)g(x) - p_0(x)L(g),\ p_1(x),\ p_2(x),\ p_3(x),
					$$
					and one for the second case it is given by
					$$
					L(q)p_0(x) - q(x)L(p_0),\ p_1(x),\ p_2(x),\ p_3(x).
					$$
				\item $T(A) = (3, 5, 7)$. This occurs when $\Delta \not = 0$, $L(q) = 0$ and $L(p_0) \not = 0$. A possible SAGBI basis for this case is given by
					$$
					q(x), p_1(x), p_3(x).
					$$
				\item $T(A) = (3, 4)$. This occurs when $\Delta \not = 0$ and $L(q) = L(p_0) = 0$. A possible SAGBI basis for this case is given by
					$$
					q(x), p_0(x).
					$$
			\end{itemize}
		\item $A = \{f(x) | f(\alpha) = f(\beta); f(\gamma) = f(\delta); f'(\lambda) = 0\}$
			In this case let $D(f) = f'(\lambda).$
			The possible types of $A$ are given as follows.
			\begin{itemize}
				\item $T(A) = (2, 7)$. This occurs when $\Delta = 0$ and $\lambda = \frac{\alpha + \beta}{2}$. A possible SAGBI basis for this case is given by
					$$
					g(x), p_3(x).
					$$
				\item $T(A) = (4,5,6,7)$. This occurs when $\Delta = 0$ and $\lambda \not = \frac{\alpha + \beta}{2}$,
					or when $\Delta \not = 0$ and $L(q) \not = 0$. A possible SAGBI basis for the first case is given by
					$$
					L(p_0)g(x) - p_0(x)L(g),\ $$
                   $$L(p_1)g(x) - p_1(x)L(g),\ p_2(x),\ p_3(x),
					$$
					and one for the second case is given by
					$$
					L(p_0)q(x) - p_0(x)L(q), $$
                   $$ L(p_1)q(x) - p_1(x)L(q),\ p_2(x),\ p_3(x),
					$$
				\item $T(A) = (3, 5, 7)$. This occurs when $\Delta \not = 0$, $L(q) = 0$ and $L(p_0) \not = 0$. A possible SAGBI basis for this case is given by
					$$
					q(x), L(p_0)p_1(x) + p_0(x)L(p_1), p_3(x).
					$$
				\item $T(A) = (3, 4)$. This occurs when $\Delta \not = 0$ and $L(q) = L(p_0) = 0$. A possible SAGBI basis for this case is given by
					$$
					q(x), p_0(x).
					$$
			\end{itemize}
	\end{itemize}
	\end{theorem}
	\begin{proof}
		First note that there is no other combination of conditions that
		specify a subalgebra of codimension three with a spectrum of size five. We
		must have three conditions and as each condition only can contribute with at most two elements to the spectrum, two of these conditions
		must contribute with two elements to the spectrum. There is only one
		type of condition that adds two new elements to the spectrum, namely
		conditions of the form $L(f) = f(x) - f(y)$ where $x, y$ are not
		previously part of the spectrum. Thus we are really only free in
		setting the third condition on $A$. Either we can add a condition of
		type $L(f) = f(x) - f(y)$ where $x$ previously belonged to the
		spectrum, or we can add an $\alpha$-derivation. Note that such an
		$\alpha$-derivation must be trivial, as Theorem \ref{th:noderivations}
		tells us that no non-trivial $\alpha$-derivations add more elements to the spectrum. Now, it remains to justify each case. But
		first, some additional
		notation. \\

		Let $B = \{f(x) | f(\alpha) = f(\beta); f(\gamma) = f(\delta)\}$. We
		will consider $B$ as the subalgebra from which $A$ is created by adding an extra condition. Note that $p_i \in B$ for all positive
		$i$. Furthermore, Theorem \ref{th:codim2} states that when $\Delta = 0$,
		we have $T(B) = (2,5)$ and $g(x), p_1(x)$ is a SAGBI basis of $B$. If
		however $\Delta \not = 0$, then $T(B) = (3,4,5)$ and $q(x), p_0(x),
		p_1(x)$ is a SAGBI basis of $B$. \\

		We begin by treating all cases when $A$ is defined by equality
		conditions only, no derivations. If $\Delta = 0$ then $T(B) = (2,5)$
		and as $A$ does not satisfy the conditions outlined in Theorem
		\ref{th:deg2}, we must have $T(A) = (4,5,6,7)$.  It is easily verified
		that the constructed SAGBI basis resides in $B$, satisfies the extra
		condition, and generates an algebra of the specified type. \\

		If instead $\Delta \not = 0$, then $T(B) = (3, 4, 5)$ and we can use
		the SAGBI basis of $B$, namely $q(x), p_0(x), p_1(x)$, to construct a
		SAGBI basis for $A$. We include and modify basis elements depending on
		which of them satisfy the added condition. \\

		Now, we treat the cases when $A$ is derived as the kernel of some
		derivation on $B$. Here we also branch on wether $\Delta$ is zero or
		not. \\

		If $\Delta = 0$ and $T(B) = (2,5)$ then $A$ satisfies Theorem
		\ref{th:deg2} only if $A = \text{ker}\ D$ where $D(f) = af(\lambda)$
		and $\lambda = \frac{\alpha + \beta}{2}$.  If $\lambda$ is given as
		such, then it follows that $T(A) = (2,7)$ and the basis is easily
		verified.  If however $\lambda \not =  \frac{\alpha + \beta}{2}$, then
		$T(A) = (4,5,6,7)$ and again, it is easy to verify the basis.

		When $\Delta \not = 0$ we proceed in the same fashion as before and
		construct a SAGBI basis from $q(x)$ and the $p_i(x)$ depending on which
		of the polynomials satisfy the added condition.
	\end{proof}

\chapter{Creating derivations}

To prove the main conjecture we probably need to understand the nature of derivations: how new derivations are obtained when some former derivations are turned into subalgebra conditions. This is far from trivial and here we discuss some observations.

\section{Integral}
Let $A$ be obtained from $B$ as the kernel of an $\alpha-$derivation $L.$ We call a polynomial $a$ an integral if for any $f\in A$ we have that $af'\in B.$
For example, if $B=\MK[x]$ then any $a$ is an integral. For $B=<x^2,x^3>$ and $A=<x^2,x^5>$ we find that $x$ is an integral.
\begin{theorem}If $a$ is an integral then the map $D:f\rightarrow L(af')$ is an $\alpha-$derivation of the subalgebra $A.$
\end{theorem}
\begin{proof}
We have $$af'\in B$$ thus $L$ is well defined and linear. Besides that if $f,g\in A$ then $L(f)=L(g)=0$ and
$$D(fg)=L(a(fg)')=L(af'g)+L(ag'f)=$$
$$L(af')g(\alpha)+ 0+L(ag')f(\alpha)+0=D(f)g(\alpha)+f(\alpha)D(g).$$
\end{proof}

This idea can be generalized. Consider a map $F:A\rightarrow B$ such that $F(pq)=F(p)q+pF(q).$
\begin{theorem}The map $D= L\circ F$ is an $\alpha-$derivation of the subalgebra $A.$
\end{theorem}
\begin{proof} We have for $p,q\in A=\ker L$
$$D(pq)=L(F(p)q+pF(q))=$$
$$L(F(p))q(\alpha))+F(p)(\alpha)L(q)+L(p)F(q)(\alpha)+p(\alpha)L(F(q))=$$
$$D(p)q(\alpha)+p(\alpha)D(q).$$
\end{proof}

\section{Single element in the spectrum}
Unfortunately not each derivation can be created using integrals. To understand how the derivations can appear  we want to study a special concrete case. When $A$ has a single element $\alpha$ in the spectrum,
where already have proved the main conjecture.
 First of all if $p'(\alpha)=0$ for any $p\in A$, then $D_2:p\rightarrow \frac{p''(\alpha)}{2!}$ and $D_3:p\rightarrow \frac{p'''(\alpha)}{3!}$ are two $\alpha-$derivations. Consider the following list of  the maps created with the help of Maple:
$$                                   D_1      $$
$$                                D_3 - c D_2;$$
$$                               D_5 - 2 c D_4;$$
$$                                 D_7 - 3 cD_6 + 3 c^3D_4;$$
$$                  D_9 - 4 cD_8 + 11 c^3D_6 - 11 c^5D_4;$$
$$          D_{11} - 5 cD_{10} + 26 c^3D_8 - 78 c^5D_6 + 78 c^7D_4;$$
$$  D_{13} - 6 cD_{12} + 50 c^3D_{10} - 294 c^5D_8 + 882 c^7D_6 - 882 c^9D_4;$$
$$D_{15} - 7 c D_{14} + 85 c^3D_{12}- 816 c^5D_{10} + 4811 cD_8^7 - 14433 c^9D_6
   + 14433 c^{11}D_4;$$
$$     D_{17} - 8 cD_{16} + 133 c^3D_{14} - 1881 c^5D_{12} + 18145 c^7D_{10}$$
 $$        - 106989 c^9  D_8 + 320967 c^{11}   D_6 - 320967 c^{13}   D_4;$$
$$D_{19} - 9 cD_{18} + 196 c^3D_{16} - 3822 c^5D_{14} + 54399 c^7D_{12}$$
$$   - 524880 c^9  D_{10} + 3094881 c^{11}  D_8 - 9284643 c^{13}   D_6 + 9284643 c^{15}   D_4.$$
Here $D_k$ is the map $D_k:p\rightarrow \frac{p^{(k)}(\alpha)}{k!}$ and $c$ is a constant.

We know that the first map is an $\alpha-$derivation. But what is more interesting is that if the first $k$ maps defines a subalgebra inside $A$ (as the intersection $C$ of their kernels with $A$)
then the next map will be a derivation of $C.$

The numerical coefficients $C_i$ with $D_n$ have an interesting property:
$$C_0=0$$
$$C_2+C_3=0;$$
$$C_4+2C_5+C_6=0$$
$$C_6+3C_7+3C_8+C_9=0$$
$$C_8+4C_9+6C_{10}+4C_{11}+C_{12}=0$$
$$\ldots$$
$$C_{2m}+\binom m 1C_{2m+1}+\binom m2C_{2m+2}+\cdots+\binom m{m-1}C_{3m-1}+C_{3m}=0$$
\begin{theorem} Let $n=2k+1$ be an odd number. If we demand
\begin{itemize}
  \item $C_n^{(n)}=1$ and $C_i^{(n)}=0$ for all other odd $i$;
  \item  $C_i^{(n)}=0$ for all even $i>n;$
  \item $C_{2m}^{(n)}+\binom m 1C_{2m+1}^{(n)}+\binom m2C_{2m+2}^{(n)}+\binom m3C_{2m+3}^{(n)}+\cdots\\ +\binom m{m-1}C_{3m-1}^{(n)}+C_{3m}=0$ for all $m$
\end{itemize} then the numbers $C_i^{(n)}$ are uniquely determined.
\end{theorem}
\begin{proof} For odd numbers it is trivial. For even numbers $i>n$we have zeros and they satisfies the equations with $2m>n.$
For $C_{2k}^{(n)}$ we have
$$C_{2k}^{(n)}+\binom k 1C_{2k+1}+0+\cdots=0\Rightarrow C_{2k}^{(n)}=-\binom k 1=-k.$$
If $C_i^{(n)}$ is defined for all $i>2m$ then we have
$$C_{2m}^{(n)}=-\left[\binom m 1C_{2m+1}^{(n)}+\binom m2C_{2m+2}^{(n)}+\cdots +\binom m{m-1}C_{3m-1}^{(n)}+C_{3m}\right]$$
and all is uniquely defined by induction.
\end{proof}

Now for each odd $n$ we can define
$$L_n=\sum_{i=0}^n C_i^{(n)}c^iD_{n-i}.$$
\begin{Conj}If $L_1(f)=L_3(f)=\cdots=L_{n-2}(f)=0$ for each $f\in A$ then $L_n$ is an $\alpha-$derivation in $A.$
\end{Conj}

\chapter{Further development}

Here we want to discuss some possible ways to generalize the obtained results. We have several restrictions. Can we skip them?

First of all we can consider subalgebras of infinite codimension. Then we need infinitely many conditions, so spectra can be infinite as well.
But there are many interesting questions here.

Next we have restrictions on the field. Characteristics zero seems to be important, otherwise we have problem already for monomial algebras.
But we probably can work with the divided powers.

The demand that the field is algebraically closed probably is less restrictive, at least if we allow the spectrum elements to belong to the algebraic closure of the field.
But to understand under what circumstances a subalgebra over the field of complex numbers consists of real elements is an interesting question here.

Constructing the SAGBI basis is also interesting, because
the main tool - the subalgebra $B$ codimension one less is absent, even though we in the real case can find a subalgebra  of codimension two less.
This is an interesting area for further investigation.

Perhaps, the most interesting generalization is to allow more than one variable. Here we need to use partial derivatives and for example the monomial subalgebras get a similar description as in the univariate case. So there is a realistic hope for the theory to be extendable to several variables. One problem is that it is not clear that the spectrum cannot contain ghost elements if we increasing the number of variables.

The main tool - the subalgebra B still exists but now we need to speak about $(\alpha,\beta)-$ derivations. The SAGBI bases seem to  be constructed in a similar way and therefore should still be finite.
But there are many differences. First of all $f(\alpha,\beta)=0$ does not give us a factor in $f(x,y)$ which is a fact that we have relied substantially on in the one-dimensional case.
Therefore we have no direct analogs of the proofs for theorems corresponding to theorems \ref{th:power}, \ref{th:main2}, \ref{th:derivativenotinspectrum}. It would be interesting to know if they are still valid.

Another difference is that there exists proper subalgebras in $\MK[x,y]$ with empty spectrum.
An example inspired by \cite{newman} is the subalgebra
$A=<x,xy,xy^2-y>.$

Indeed, $$f(\alpha,\beta)=f(\gamma,\delta)$$
applied to $x$ gives $\alpha=\gamma.$ Then $\beta\neq\delta$ and $\alpha\beta=\alpha\delta$ implies $\alpha=0.$ Now
$$\alpha\beta^2-\beta=\alpha\delta^2-\delta\Rightarrow \beta=\delta,$$ a contradiction.

Similarly  $$af'_x(\alpha,\beta)+bf'_y(\alpha,\beta)=0$$ applied to $x$ gives $a=0.$ Thus $b\neq 0$ and
application to $xy$ gives $\alpha=0.$ But then
$$b(xy^2-y)'_y(0,\beta)=-b\neq 0.$$

To check that it is a proper subalgebra suppose that
$$y=F(x,xy,xy^2-y).$$ If we put $y=\frac{1}{x}$ here then we obtain $\frac{1}{x}=F(x,1,0)$ - a contradiction.

In fact no $y^k$ belongs to $A$ and we have, as expected, infinite codimension while $\MK[x,y]$ is the only subalgebra of finite codimension that contains $A.$

But it is impossible to construct such examples with finite codimension or in the one-variable case.

An interesting question is to find a homological interpretation. Some kind of homological algebra should be here.

The characteristic polynomial is especially interesting. What is the correct definition?  Can it be introduced for several variables?
Can it be interpreted as the characteristic polynomial  of some operator on $V^2$ or $V\times V^*$, where $V=\MK[x]/A?$

The size of spectrum. Is it an inner property of subalgebra?   Because $<x^2>$ has an infinite spectrum, probably the size of spectrum depends  on the embedding of the subalgebra in  $\MK[x].$ But maybe  this is not the case if we restrict ourselves by finite codimension only.

Applications. The spectrum open many possible applications. As exciting example  we can consider is the Jacobian conjecture. What we need to prove first is that the spectrum of the subalgebra the polynomials define has zero spectrum.
Probably an equivalent condition is (as in one variable case) that all derivations are trivial. Then a non-zero jacobian could be another equivalent condition.
\chapter{Acknowledgements}
We are thankful to our mathematical department which gave us an opportunity to work on this  project despite the difficult pandemic situation.
The starting point of the project  was the Master degree defence of the first author, where the last author was the scientific adviser ant the third was the opponent. It was the observation that
the subalgebra $<x^3-x,x^2>$ can be defined by the conditions $f(1)=f(-1)$ the gave the last author the idea to study subalgebra conditions. He suggested that the main theorem should be true, introduced the main definitions and plan for how the theorem can be proven. During one year we divided between us different parts of the work to carry out this plan and discussed how to develop the ideas.

Trying to classify together the subalgebras of type $(3,4)$ we got the idea of the characteristic polynomial and the spectrum. The idea to use derivations came much later but became a main tool in the induction approach. SAGBI bases was always the important tool.

Prof. Arne Meurman was always participating in our regular meetings and we are very thankful him for his valuable remarks. Another student, Hugo Eberhard, was participating in part of discussions as well. We are thankful him as well.

Later another student, the second author, joined the project and actively participated in the classification part.

We were glad to share the joy to be a mathematician and do not consider the project as finished. But somewhere we need to set a point and publish the result obtained so far.

\printindex
\vspace{2mm}
\begin{center}
\begin{parbox}{118mm}{\footnotesize
Victor Ufnarovski

\vspace{3mm}

\noindent Lund Institute of Technology/ Centre for Mathematical Sciences.

\noindent Email: victor.ufnarovski@math.lu.se, anna.torstensson@math.lu.se

}

\end{parbox}
\end{center}

\end{document}